\documentclass[reqno]{amsart}
\usepackage{url}

\makeatletter
\@addtoreset{equation}{section}
\makeatother

\newtheorem{theorem}{Theorem}[section]

\newtheorem{lemma}[theorem]{Lemma}
\newtheorem{proposition}[theorem]{Proposition}
\newtheorem{corollary}[theorem]{Corollary}

\newtheorem{definition}[theorem]{Definition}
\newtheorem{remark}[theorem]{Remark}

\usepackage{mathtools,amsmath,amsfonts,amssymb,amsthm,verbatim}
\usepackage{graphicx,float,caption,subcaption}
\usepackage{hyperref}
\usepackage{dsfont}
\usepackage{todonotes} 
\usepackage{breqn}

\DeclareMathOperator*{\argmax}{arg\,max}
\DeclareMathOperator*{\argmin}{arg\,min}

\newcommand{\per}{\mbox{Per}}

\begin{document}

\title{Consistency of modularity clustering on random geometric graphs}
\author{Erik Davis and
Sunder Sethuraman}
\date{}

\address{\noindent Department of Mathematics, University of Arizona,
  Tucson, AZ  85721
\newline
e-mail:  \rm \texttt{edavis@math.arizona.edu}
}

\address{\noindent Department of Mathematics, University of Arizona,
  Tucson, AZ  85721
\newline
e-mail:  \rm \texttt{sethuram@math.arizona.edu}
}

\begin{abstract}
We consider a large class of random geometric graphs constructed from samples $\mathcal{X}_n = \{X_1,X_2,\ldots,X_n\}$ of
independent, identically distributed observations of an underlying
probability measure $\nu$ on a bounded domain $D\subset \mathbb{R}^d$.  The popular `modularity' clustering method
specifies a partition $\mathcal{U}_n$ of the set $\mathcal{X}_n$ as the solution of an optimization problem. 
In this paper, under conditions on $\nu$ and $D$, we derive scaling limits of the modularity
clustering on random geometric graphs. Among other results, we show  a
geometric form of consistency: When the number of clusters is a priori
bounded above, the discrete optimal partitions $\mathcal{U}_n$ converge in a certain sense to a continuum partition $\mathcal{U}$ of the underlying domain $D$, characterized as the solution of a type of Kelvin's shape optimization problem.
\end{abstract}



\maketitle

\section{Introduction}

One of the basic tasks in understanding the structure and function of complex networks
is the identification of community structure or modular
organization, where by a community we mean a subset of densely
interconnected nodes, only sparsely connected to outsiders (cf. \cite{Fortunato2010}, \cite{Porter_Onnela_Mucha}). 
A widely popular approach to community detection is the method of modularity clustering, introduced by Newman and Girvan
(cf. \cite{newman2004finding}, \cite{newman2006modularity}), which specifies an optimal clustering -- that is, a partition of the network -- as the solution of a certain optimization problem (see Section \ref{model} for precise definitions). In particular, the method is used for 
a variety of networks arising in scientific contexts,
including metabolic networks
\cite{guimera2005functional}, epigenetic networks
\cite{mill2008epigenomic}, brain networks \cite{hagmann2008mapping}, and networks encoding ecological
\cite{fortuna2010nestedness} and political interactions \cite{porter2005network}.

On the other hand, a popular model of complex networks with geometric structure is the random geometric graph, where vertices are sampled from a geometric
domain and edge weights are determined by a function of the distance between
vertices \cite{franceschetti_meester}, \cite{meester_roy}, \cite{penrose2003random}.  We note, a well-studied case is the unweighted version, when the connectivity function is a threshold function of distance.  These graphs are well-established mathematical models of various physical phenomena, such as continuum percolation.  They have also found use in a number of applied
settings, including the modeling of ad-hoc
wireless networks \cite{gupta2000capacity}, \cite{bettstetter2002minimum},
\cite{gamal2004throughput}, protein-protein interaction networks
\cite{prvzulj2004modeling}, as well as the study of combinatorial
optimization problems \cite{diaz2001approximating},
\cite{diaz2002survey}.  In clustering studies of these graphs and their variants, the modularity functional is often used to assess the quality of the clustering obtained \cite{antonioni}, \cite{dhara}.

 In these contexts, it is a natural question to ask about the consistency of modularity clustering with respect to random geometric graphs.  
That is, one would like to know how the modularity clusterings converge or stabilize as the number of sampled data points or vertices grow, and how to characterize geometrically any limit clusterings.  From the point of view of applications, establishing consistency is relevant to benchmarking performance.   We will focus on a class of random geometric graphs, where a length scale is introduced in the connectivity function, so that distances between adjacent vertices shrink as the number of data points grow, and spatial scaling limits can be taken.  

In this article, we study
two questions:  The first asks about the large limit
behavior of the modularity functional on these random geometric graphs, evaluated on a
fixed partition of the underlying geometric domain.  The second question asks whether the
optimal modularity clusterings for the discrete graphs converge, and if so to what geometric limit clustering.

With respect to the first question, when the fixed partition of the domain involves at most $K\geq 1$ sets, we show that the limit of the modularity functional, known to be a priori bounded between $-1$ and $1$, as the sample size grows, is of the form $1-1/K$ plus a term involving the partition, which vanishes when the partition is `balanced', that is when each set in the partition has the same volume with respect to an underlying measure (Theorem \ref{thm:asymptotics}). As a corollary, we obtain, for these random geometric graphs, that the maximum of the modularity functional taken over all partitions, as the sample size grows, tends to $1$ (Corollary \ref{cor:main}).  These limits prove, in the context of random geometric graphs, some heuristics given in the literature (cf. Subsection \ref{discussion}).

With respect to the second question, we show that, given a fixed upper
bound $K$ on the number of clusters, the optimal modularity clusterings of the discrete random graphs converge in a distributional sense to an optimal clustering, satisying a certain shape optimization problem (Theorem \ref{thm:mainthm}).  This geometric continuum optimization problem, of interest in itself, is a form of
Kelvin's problem: Informally, find a partition of the domain, where each set has the same volume, but where the perimeters between sets is minimized.  Noting the first result above, it seems natural that a constraint specifying equal volumes would emerge in the limiting shape problem.  Nonetheless, it seems intriguing that a form of Kelvin's shape optimization problem (cf. Subsection \ref{continuum problem}) would appear.

Previously, in the literature, a type of statistical consistency of
modularity clustering for stochastic block models has been considered.
In the stochastic block model, each data point is assigned a group
from $K$ groups according to a probability $\pi$.  Then, if the two
points belong to groups $a$ and $b$ respectively, one puts an edge between them with probability $F_{a,b}$.  Modularity clustering now gives an empirical group assignment to each data point.  One of the main results shown is that the proportion of misclassification of empirical group assignments, with respect to the a priori assignment, vanishes in probability, 
Bickel and Chen \cite{Bickel-Chen} for the model above, Zhao, Levina and Zhu \cite{Zhao-Levina-Zhu} for degree-corrected models, Rohe, Chatterjee and Yu \cite{Rohe-Chatterjee-Yu} for high-dimensional models, and Le, Levina and Vershynin \cite{Le-Levina-Vershynin} via low rank approximation.

In this context, the main focus of the article is to understand the geometry of optimal partitions, and other asymptotics with respect to the modularity functional.  
Our results represent limit results on a `geometric' form of consistency of the discrete modularity clusterings, which seem not to have been considered before.  Moreover, as a corollary of this type of consistency, we show that the proportion of misclassification of the data points into sets given by the modularity functional, with respect to the limit optimization problem, vanishes a.s. (Corollary \ref{cor:proportion}).

We use 
the recent framework of optimal transport introduced
by Garc{\'\i}a Trillos and Slep{\v{c}}ev in \cite{trillos2014continuum}, in the context of continuum limits of graph variational problems, to help relate point cloud empirical measures to absolutely continuous ones.  We rewrite, after some manipulation, the modularity functional in terms of a `graph total variation' term and a `balance' term (Section \ref{reformulation_sect}).  The proofs, with respect to the first question on asymptotics of the modularity functional of a particular clustering, make use of concentration estimates for these terms through U-statistics bounds.  

On the other hand, with respect to the second question, we observe that maximizers of the modularity functional are the same as minimizers of an energy functional built as the sum of the `graph total variation' and `balance' terms (cf. Subsection \ref{reformulation_min_sect}).  In this energy functional, the coefficient in front of the `balance' term diverges as the reciprocal of the length scale when the number of data points grows.  In the limit, the soft penalty `balance' term becomes a hard constraint. We show convergence of a subsequence of minimizers to an optimizer of the limit problem by use a modified notion of Gamma convergence that we formulate (cf. Subsection \ref{gamma_conv}), together with a compactness principle. For the `liminf' part of Gamma convergence, although we have to treat the balance constraint, the analysis of the graph total variation term follows from work in \cite{trillos2014continuum}, which handles a similar expression.  

However, dealing with the `balance' constraint represents a serious
difficulty with respect to the `limsup' or `recovery sequence' part of the
Gamma convergence setup.  Without the constraint, one can form the
recovery sequence by approximating with piecewise smooth partitions. However, such
approximations become more complicated when also enforcing the
`balance' constraint.  But, the probabilistic result for the first
question (Theorem \ref{thm:asymptotics}), given for general
measurable clusterings, already can be seen to yield a recovery sequence, with
respect to the modified notion of Gamma convergence for random functionals.
Interestingly, this notion of Gamma convergence has the same strength,
in terms of yielding subsequential convergence, as the more usual
stringent formulation (cf. Remark \ref{rmk:compactness}). Perhaps of use in other problems, we observe that this more probabilistic approach offers a diferent perspective.

With respect to previous work on statistical clustering methods, consistency of $K$-means methods have been considered by Pollard in \cite{pollard1981strong}, and more recently, via Gamma convergence, by Thorpe, Theil, Johansen, and Cade in \cite{thorpe_theil_joh_cade}.   On single linkage hierarchical clustering, consistency has been shown by Hartigan in \cite{Hartigan}. 
On Fuzzy C-Means clustering, consistency has been considered by Sabin in \cite{Sabin}.  With respect to spectral clustering, consistency has been considered by Belkin and Niyogi in \cite{Belkin-Niyogi1}, Von Luxburg, Belkin and Bousquet in \cite{von2008consistency}, and Garc{\'\i}a Trillos and Slep{\v{c}}ev in \cite{Trillos_consistency}, the last article, employing the framework of \cite{trillos2014continuum}, as in this paper.  
We also mention, using Gamma convergence (or `epi-convergence') techniques, consistency of maximum likelihood and other estimators was studied by Wets in \cite{Wets} and references therein.

Also, related, Arias-Castro and Pelletier \cite{Arias_Pelletier} considered the consistency of the dimension reduction algorithm `maximum variance unfolding', and Arias-Castro, Pelletier and Pudlo \cite{APP} and  Garc{\'\i}a Trillos, Slep{\v{c}}ev, von Brecht, Laurent and Bresson \cite{trillos2014consistency} studied the consistency of Cheeger and ratio cuts.   Pointwise estimates between graph Laplacians and their continuum counterparts
were considered by Belkin and Niyogi \cite{Belkin-Niyogi}, Coifman and Lafon \cite{Coiffman_Lafon}, Gin\'e and Koltchinskii \cite{Gine_Koltch}, Hein, Audibert and Von Luxburg \cite{Hein_et_al}, and Singer \cite{Singer}.  
In addition,
spectral convergence in more general contexts was considered by Ting, Huang and Jordan \cite{Ting_Huang_Jordan} and Singer and Wu \cite{Singer_Wu}.

In addition, we mention there is a large literature on Gamma convergence of discrete lattice based variational expressions to continuum optimization problems (cf. Braides \cite{braides2002gamma}, Braides and Gelli \cite{braides_gelli} and references therein).  More recently, see van Gennip and Bertozzi \cite{vanGennip_Bertozzi} which studies Gamma convergence of Ginzburg-Landau graph based functionals to continuum limits.

The plan of the paper is the following.  In Section \ref{model}, we define the random geometric graph model and state and discuss the two main results (Theorems \ref{thm:asymptotics} and \ref{thm:mainthm}) on modularity clustering.  In Section \ref{preliminaries}, we discuss preliminaries with respect to optimal transport distances and Gamma convergence--we remark that the proof of Theorem \ref{thm:mainthm} relies on this section, but the proof of Theorem \ref{thm:asymptotics} does not.  In Section \ref{reformulation_sect}, we develop the modularity functional into a convenient form amenable to later analysis.  In Section \ref{proof_asymptotics} and \ref{proof_mainthm}, the proofs of the two main theorems are given respectively.  Finally, in Section \ref{appendix}, some technical calculations and proofs, referenced in previous sections, are collected.

\section{Model and Results}
\label{model}

We first introduce in Subsection \ref{graph partitioning}
the modularity functional on graphs with weighted
edges.  In Subsection \ref{continuum problem}, we
discuss a form of Kelvin's continuum shape optimization problem.  In Subsection \ref{results}, we state our main theorems, and make some remarks in Subsection \ref{discussion}.

\subsection{Graph Partitioning by Modularity Maximization}
\label{graph partitioning}

Let $\mathcal{G} = (X,W)$ be a weighted graph with vertex set $X \coloneqq
\{x_1,x_2,\ldots,x_n\}$ and weight matrix
$W$, where the entry $W_{ij} \geq 0$ denotes the weight of the (undirected) edge between
vertex $x_i$ and $x_j$. The degree of vertex $x_i$ is $d_i \coloneqq \sum_{j}
W_{ij}$, and the total weight of the graph 
is $m
\coloneqq \frac{1}{2}\sum_i d_i = \frac{1}{2}\sum_{i,j} W_{ij}$. When
convenient, we will refer to the vertex $x_i$ simply as 'vertex
$i$'.

Let $\mathcal{U} = \{U_k\}_{k=1}^K$ be a partition of the vertex set $X$. We shall sometimes refer to the nontrivial sets $U_k$ of $\mathcal{U}$ as `clusters', and we denote by $|\mathcal{U}|$ the number of clusters. We may have $|\mathcal{U}| < K$, if one of the sets $U_k$ is empty. The partition $\mathcal{U}$ is therefore equivalent to an assignment $\{c_i\}_{i=1}^n$ of labels $c_i \in \{1,\ldots,K\}$ to vertices, where $U_k = \{ x_i : c_i = k \}$ for $1 \leq k \leq K$.

Modularity was originally introduced as a quantitative measure of the bias towards of vertices in a given cluster to be connected to other vertices in the same cluster \cite{newman2004finding}.
Informally, the idea is to
measure the proportion of edge weight that runs between vertices in the same cluster, and compare it to the expected proportion if the
edge weight was redistributed at random, according to a `null' or `ground truth' model.

The total proportion of edge weight 
between vertices in the same cluster is
\begin{equation*}
\frac{1}{2m}\sum_{i,j} W_{ij}\delta(c_i,c_j),
\end{equation*}
where $\delta(c_i,c_j)$ is the indicator that $c_i = c_j$, and the factor of $1/2$ arises because distinct vertex
pairs appear twice in the sum.

Let $E_{ij}$ denote the expected edge weight
between vertex $i$ and vertex $j$ under a random redistribution model, which we
specify below in different forms. Then,
the expected proportion of edge weight between vertices in the same cluster is
\begin{equation*}
\frac{1}{2m} \sum_{i,j} E_{ij}\delta(c_i,c_j).
\end{equation*}
Then the modularity $Q$ is defined to be
\begin{equation*}
Q(\mathcal{U}) = \frac{1}{2m}\sum_{i,j} \Big( W_{ij} - E_{ij}\Big) \delta(c_i,c_j),
\end{equation*}
and one has $-1\leq Q\leq 1$. The guiding thought is that a partition $\mathcal{U}$ with large modularity $Q(\mathcal{U})$ would represent a good clustering of the graph.

The most popular choice of null model, and the one originally introduced in \cite{newman2004finding}, specifies $E_{ij} = \frac{d_i
  d_j}{2m}$.  For unweighted graphs, where $W$ is the adjacency matrix, this choice corresponds to the \textit{configuration
  model}.  Namely, with respect to the vertex degrees
$d_1,\ldots,d_n$, consider the following distribution over graphs with $n$ vertices and $m$ edges.
At each vertex $i$ place $d_i$ half-edges.  Then, successively choose a pair of half-edges at random without replacement and connect them to form an edge with unit weight.
Then, to dominant order, the
expected number of edges between vertex $i$ and vertex $j$ behaves as $\frac{d_i d_j}{2m}$.

Another natural choice corresponds
to the \textit{Erd\H{o}s-R\'{e}nyi model}, when $E_{ij}
= \frac{2m}{n^2}$.  Namely, for unweighted graphs, when $W$ again is the adjacency matrix, if the
endpoints of each edge is chosen uniformly from the vertex set, with $m$
edges total, the
expected number of edges between vertex $i$ and vertex $j$ is $E_{ij}
= \frac{2m}{n^2}$ to dominant order.

More generally, we may define $E_{ij}$ which interpolates in and among these two possibilities.
For $\alpha \in \mathbb{R}$, let $S = \sum_{i=1}^n d_i^{\alpha}$,
and $E_{ij} = 2m \frac{d_i^{\alpha} d_j^{\alpha}}{S^2}$. Define the `$\alpha$'-modularity $Q$ to be
\begin{equation} \label{eq:generalmodularity}
Q(\mathcal{U}) \coloneqq \frac{1}{2m} \sum_{i,j}\Big(W_{ij} - 2m\frac{d_{i}^{\alpha}
  d_j^{\alpha}}{S^2}\Big) \delta(c_i,c_j).
\end{equation}
When
$\alpha = 0$ or $1$, this reduces to the modularity corresponding to the Erd\H{o}s-R\'{e}nyi model or the configuration model respectively.

The {\it modularity maximization} problem is the following: Given a graph $\mathcal{G}$, find the partition for which the modularity $Q$ is maximized. In other words one solves the following optimization problem,
\begin{equation} \label{prob:modoriginal}
\begin{aligned}
& \underset{\mathcal{U}}{\text{maximize}}
& & Q(\mathcal{U}),
\end{aligned}
\end{equation}
where the maximization occurs over all partitions $\mathcal{U}$ of the vertex set. One is particularly interested in optimal partitions $\mathcal{U}^* \in \argmax_{\mathcal{U}} Q(\mathcal{U})$, representing a division of the graph into natural communities. 

We also consider the following variant of problem \eqref{prob:modoriginal}, in which we restrict the partitions to have at most $K$ classes: 
\begin{equation} \label{prob:kmodoriginal}
\begin{aligned}
& \underset{\mathcal{U}}{\text{maximize}}
& & Q(\mathcal{U}), \\
& \text{subject to} & & |\mathcal{U}| \leq K.
\end{aligned}
\end{equation}

We remark that finding a global maximum of the modularity is known to
be NP-hard (\cite{brandes2008modularity}).  Nonetheless, there are various algorithms to compute approximate maximizers, among them
greedy algorithms
(cf. \cite{clauset2004finding}, \cite{blondel2008fast}), and relaxation methods (cf. \cite{newman2006finding},
\cite{newman2013spectral}, \cite{hu_et_al}).
There is also a Potts model interpretation of modularity which offers another computational perspective (cf. \cite{reichardt2006statistical}, \cite{guimera2004modularity}).

\subsection{Geometric Partitioning} \label{continuum problem}

We now discuss a continuum partitioning problem, which will emerge as a scaled limit of the discrete graph partitioning optimization.
Consider a domain $D \subset{\mathbb{R}}^d$, and a partition $\mathcal{U} = \{ U_k\}_{k=1}^K$ of $D$. 
For what follows, we take $K\geq 1$ to be fixed.

Suppose that we have a probability measure $\mu$ on $D$. 
We say that a
partition $\mathcal{U}$ is balanced with respect to
the measure $\mu$ if each $U_k$ has equal $\mu$-measure, that is
\begin{equation*} \label{eq:balanced}
 \mu(U_k) = \frac{1}{K} \hspace{0.5cm} k=1,\ldots,K. 
\end{equation*}

Because $\mathcal{U}$ is a partition, any two distinct sets $U_k$ and $U_l$
are disjoint. However, if they are adjacent their boundaries will
intersect nontrivially as a $d-1$ dimensional surface.  When the boundaries of $\{U_k\}_{k=1}^K$ are piecewise-smooth, we may measure the size of the interface
between $U_k$ and $U_l$, with respect to a density $\phi$ on $D$, by
\begin{equation*} 
\int_{\partial U_k \cap \partial U_l \cap D} \phi(x) d\mathcal{H}^{(d-1)}(x),
\end{equation*}
where $d\mathcal{H}^{(d-1)}$ denotes the Euclidean $(d-1)$-dimensional surface measure.
The measure of the total interface or perimeter between the sets of
$\mathcal{U}$ is therefore
\begin{equation} \label{eq:totalinterface}
\frac{1}{2} \sum_{1 \leq k \neq l \leq K} \int_{\partial U_k \cap \partial U_l \cap D}
\phi(x)d\mathcal{H}^{(d-1)}(x),
\end{equation}
noting the factor $1/2$ accounts for the fact that each distinct
pair $k,l$ contributes twice to the sum.

Because $\partial U_k \cap D = \cup_{l \neq
  k} \partial U_k \cap \partial U_l \cap D$, the quantity
(\ref{eq:totalinterface}) is equal to
\begin{equation} \label{eq:totalperimeter}
\frac{1}{2} \sum_{1 \leq k \leq K} \int_{\partial U_k \cap D} \phi(x)d\mathcal{H}^{(d-1)}(x).
\end{equation}

More generally, for partitions $\mathcal{U}$ consisting of measurable sets $\{U_k\}_{k=1}^K$,
we may define the weighted perimeter of
$U_k$ as follows:
\begin{equation*}
\per(U_k;\phi) \coloneqq TV(\mathds{1}_{U_k}; \phi),
\end{equation*}
where $TV(\mathds{1}_{U_k}; \phi)$, defined below, is the weighted total variation of the indicator function
$\mathds{1}_{U_k}$.  For sufficiently regular sets, this definition agrees
with the informal notion of perimeter \eqref{eq:totalperimeter}.

Let $L^\infty(D,\theta)$ and $L^p(D,\theta)$ be the usual spaces of functions $u$ where $\inf\{C : |u(x)|\leq C \ {\rm for} \ \theta-{\rm a.e.}\  x\} < \infty$ and $\int_D |u(x)|^p d\theta(x)<\infty$ for $1\leq p<\infty$ respectively. When $\theta$ is Lebesgue measure on $D$ and the underlying domain $D$ is understood, we will abbreviate $L^p:= L^p(D,\theta)$.  The weighted total variation of a function $u \in L^1$ is given by
\begin{equation*}
TV(u;\phi) \coloneqq \sup_{\substack{\Phi \in C^1_c(D;\mathbb{R}^d) \\
    |\Phi(x)| \leq \phi(x)}} \int_D u(x) \mbox{div } \Phi(x)\, dx,
\end{equation*}
where $C^1_c(D;\mathbb{R}^d)$ is the space of continuously
differentiable, compactly supported vector fields on $D$, and
$\mbox{div } \Phi(x) = \sum_{i=1}^d \frac{\partial \Phi_i}{\partial
  x_i}$.  In our later applications, the density $\phi$ will be
bounded above and below on $D$ by positive constants.  In this case,
the weighted total variation has many of the same properties as the total variation with respect to the uniform density $\mathds{1}_D$, discussed in detail in Chapter 3 of \cite{ambrosio2000functions}. 

Consider now the following geometric partitioning problem.  Among all balanced
$K$-partitions of $D$, choose that which minimizes the
total perimeter of its sets:
\begin{equation} \label{prob:continuumk}
\begin{aligned}
& \underset{\mathcal{U}}{\text{minimize}}
& & \frac{1}{2} \sum_{k=1}^K \per(U_k;\phi) \\
& \text{subject to} & & \mu(U_k) = \frac{1}{K}, \hspace{0.5cm} k = 1,\ldots,K.
\end{aligned}
\end{equation}
See Figure \ref{fig:continuumk} for the behavior when $D$ is a square.

\begin{figure}[H]
    \centering
    \begin{subfigure}[b]{0.3\textwidth}
        \includegraphics[width=\textwidth]{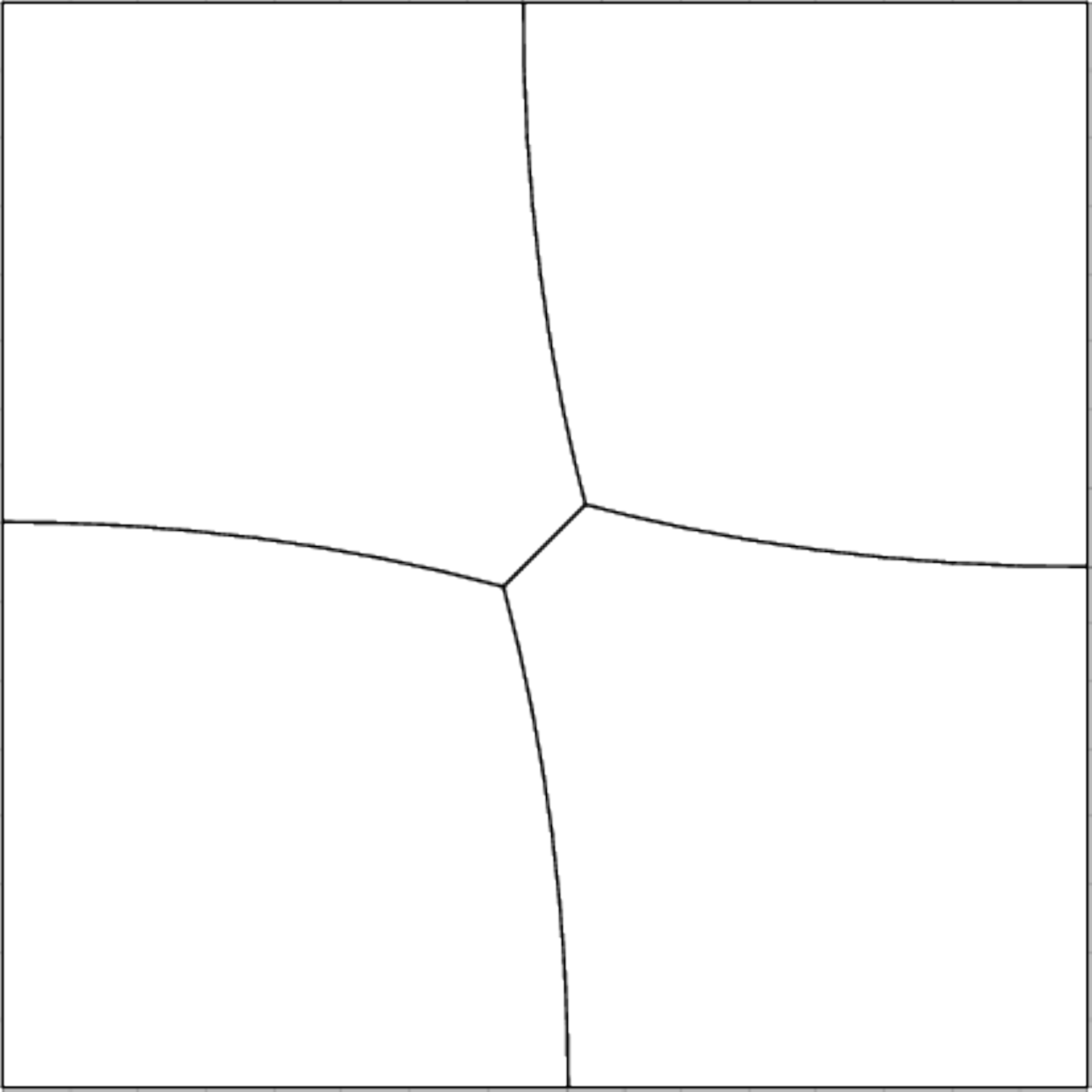}
        \caption{$K=4$}
        \label{fig:min4}
    \end{subfigure}
    ~ 
    \begin{subfigure}[b]{0.3\textwidth}
        \includegraphics[width=\textwidth]{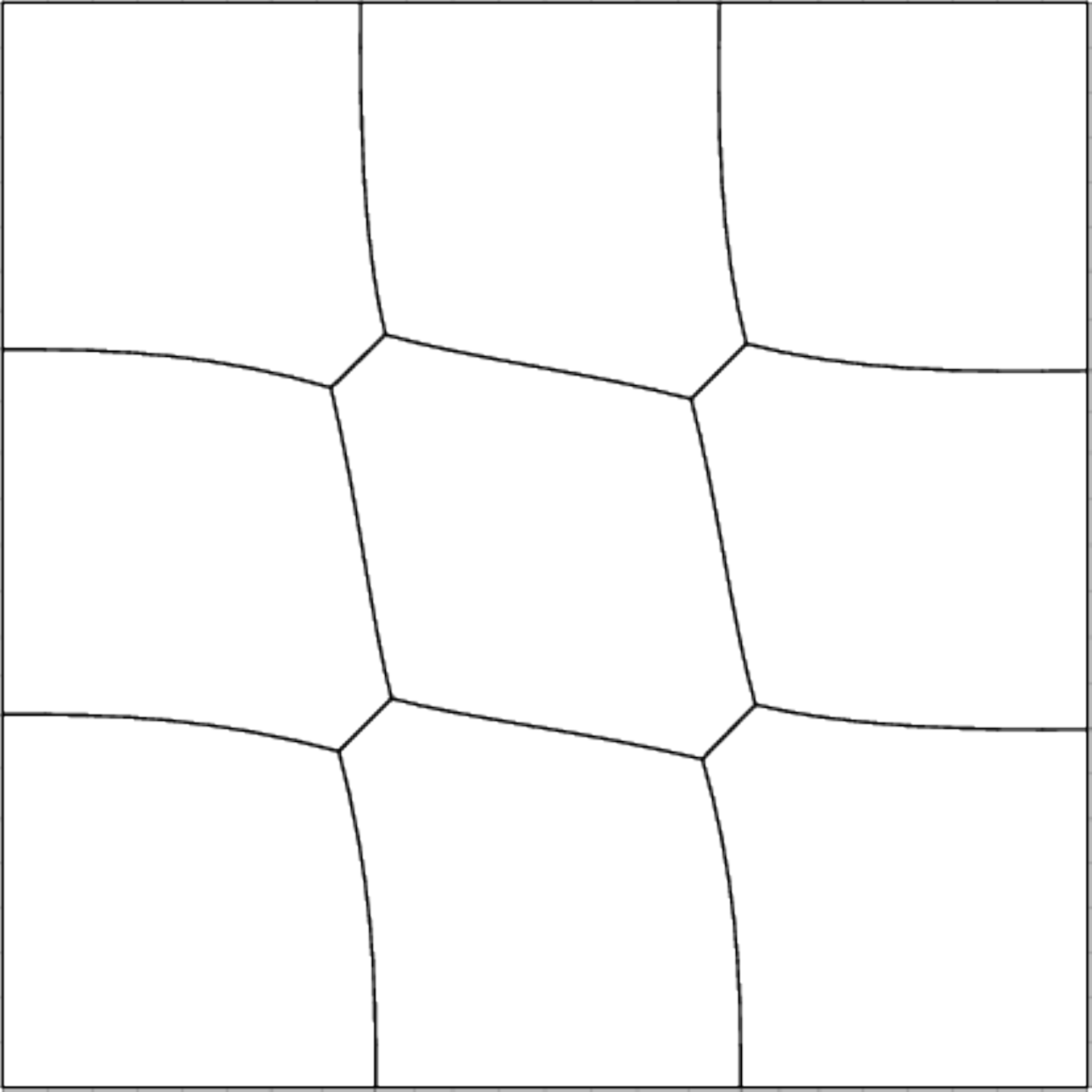}
        \caption{$K=9$.}
        \label{fig:min9}
    \end{subfigure}
\begin{subfigure}[b]{0.3\textwidth}
        \includegraphics[width=\textwidth]{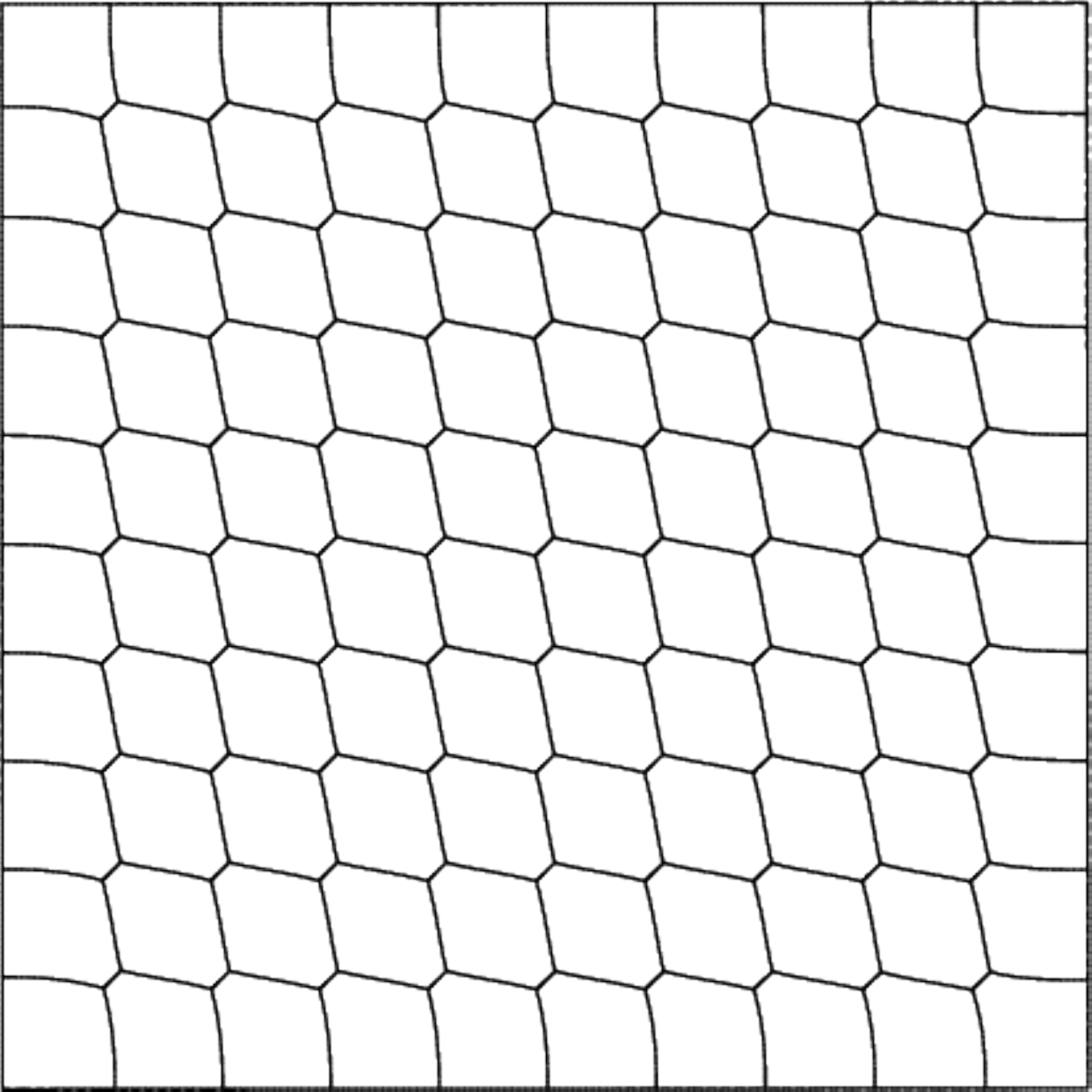}
        \caption{$K=100$.}
        \label{fig:min100}
    \end{subfigure}
    ~ 
    \caption{Local minimizers of Problem \ref{prob:continuumk} on $D =
      (0,1)^2$, with $\phi = 1$, $d\mu = dx$, produced using \cite{brakke1992surface}. \label{fig:continuumk}}
\end{figure}
 It is clear, by dividing $D$ in terms of a moving hyperplane, that
balanced $K$-partitions with finite perimeter exist.  Also, depending
on $D$, the problem may or may not have a unique solution (modulo
relabeling of sets):  Suppose $\phi = 1$ and $d\mu = dx$.  When $D$ is a long,
thin rectangle in $\mathbb{R}^2$, there is only one perimeter
minimizing division into two sets of equal area.  However, when $D$ is
a disc in $\mathbb{R}^2$, there are infinitely many such divisions, given by cuts of the disc along a diameter. 

The general problem \eqref{prob:continuumk} can be seen as a type of `soap bubble' problem.  It can also be seen as a bounded domain form of Kelvin's problem:  Find a tiling of $\mathbb{R}^d$ where each tile has unit $\mu$-volume so that the $\phi$-perimeter between tiles is minimized.  We note in passing, when $\phi = 1$ and $d\mu = dx$, in $d=2$, this problem has been solved in terms of hexagonal tiles, and is the subject of much research in $d\geq 3$. See Morgan \cite{Mor08}, which discusses these and other related optimizations.

We remark, when $\phi=1$ and $d\mu = dx$, the problem \eqref{prob:continuumk} has been considered in the literature. Ca\~{n}ete and Ritore \cite{CR03} have studied minimal partitions of the disc into three regions, and in this context proved that the regions are bounded by circular arcs making perpendicular contact with the boundary of the disc and meeting at a 120 degree triple point within. Some conjectures about minimizers for larger values of $K$ and other domains are presented in Cox and Flikkema \cite{cox2010minimal}. Oudet \cite{Oudet} derives a numerical algorithm, via Gamma convergence, for computing approximate solutions.

\subsection{Results}
 \label{results}

Before stating the two main theorems, we first define the random geometric graphs under consideration. We make the following standing assumptions throughout on the domain $D\subset \mathbb{R}^d$, ground measure $\nu$ on $D$, and underlying edge weight structure.
 
\begin{itemize}
\item[(D)] $D$ is a bounded, open, connected subset of $\mathbb{R}^d$ with Lipschitz
boundary. 
\item [(M)] In $d\geq 1$, $\nu$ is a probability measure on $D$ with distribution function $F_\nu$ and density
$\rho$ such that $\rho$ is Lipschitz and bounded above
and below by positive constants.  

Further, in $d=1$, $\rho$ satisfies additional conditions: 
(a)  $\rho$ is differentiable on $D:=(c,d)$ and (b) $\rho$ is increasing in some interval with left endpoint $c$ and decreasing in some interval with right endpoint $d$.
\end{itemize}

Let $\{X_i\}_{i \in \mathbb{N}}$ be a
collection of i.i.d. samples from $\nu$, and define $\mathcal{X}_n =
\{X_i\}_{i=1}^n$.  We will denote by $\mathbb P$ and $\mathbb E$ the probability and expectation with respect to the underlying probability space.  

The random geometric graph is constructed from the points $\mathcal{X}_n$ through a
kernel function $\eta : \mathbb{R}^d \to \mathbb{R}$, where $\eta$ satisfies:
\begin{enumerate}
\item[(K1)] $\int_{\mathbb{R}^d} \eta(x)\, dx = 1$.
\item[(K2)] $\eta$ is radial and non-increasing, i.e. $\eta(z) \leq \eta(z_0)$ if $\|z\| \geq \|z_0\|$.
\item[(K3)]  $\eta(0) > 0$ and $\eta$ is continuous at $0$.
\item[(K4)] $\eta$ is compactly supported.
\end{enumerate}
There is a large class of kernels admitted under assumptions
$(K1)-(K4)$, including the kernel associated with the well-known random
geometric graph, where $\eta$ is the indicator function of a
ball. 

Between vertex $X_i$ and $X_j$, in terms of a parameter $\epsilon_n$, we attach an edge
with weight
\begin{equation}  \label{eq:weightdef}
W_{ij}  = \begin{cases} \frac{1}{\epsilon_n^d}\eta\Big(\frac{X_i -
    X_j}{\epsilon_n}\Big) =: \eta_{\epsilon_n}(X_i-X_j) & \text{if } i \neq j, \\
0 & \text{otherwise.}
\end{cases}
\end{equation}
The parameter $\epsilon_n$ serves as a length scale. For example,
if the support of $\eta$ is contained in a ball of radius one, then two
vertices $X_i$ and $X_j$ are connected by an edge only if they
are separated by a distance no more than $\epsilon_n$ (cf. Figure \ref{fig:random_geometric_graph}).  Since the
modularity functional $Q$ is the same under weights $W$ and $cW$, for
$c > 0$, the normalization factor $\epsilon_n^{-d}$ in \eqref{eq:weightdef} was chosen so that the expected value of $W_{ij}$ is of order $1$.

\begin{figure}[h] 
  \centering
    \includegraphics[trim={4.2cm 0cm 4cm 0cm},width=\textwidth]{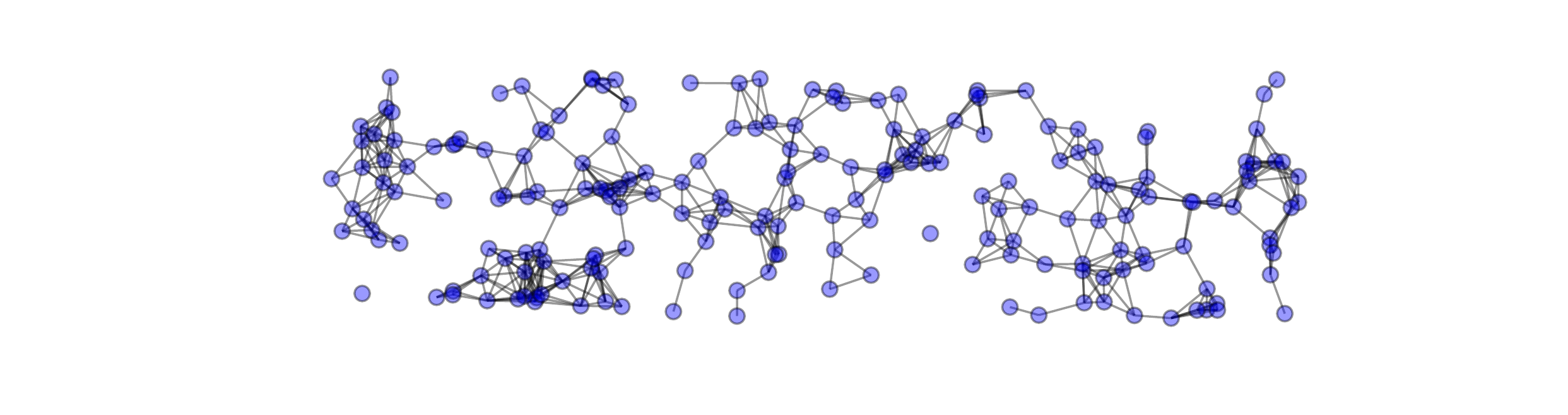}
    \caption{Random geometric graph, constructed from $n = 200$ uniformly
      distributed points on the strip $D = (0,4)\times(0,1)$ with
      connectivity function $\eta = \mathds{1}_{|x| < 1}$ and
      $\epsilon = n^{-0.3}$. \label{fig:random_geometric_graph}}
\end{figure}

Under all circumstances, we have $\epsilon_n \to 0$, although the
specific rate at which $\epsilon_n$ vanishes will depend on the dimension
$d$, as well as the parameter $\alpha$.  There are in fact two different
sets of assumptions, (I1) being more restrictive than (I2), corresponding to the two main
results given below.  We first mention typical examples, satisfying the assumptions.

Taking $\epsilon_n = n^{-\beta}$, condition (I1) will hold if 
 $$\beta<\left\{\begin{array}{rl}
 2/(d+1) & \ {\rm if \ }\alpha = 0, 1\\
 1/(d+1) & \ {\rm if \ }\alpha \neq 0,1.
 \end{array}\right.
 $$
 However, condition (I2) would be satisfied if
 \begin{equation*} \label{eq:epsrateconverge}
 \beta < \begin{cases} \min(1/d,1/2), & \text{if }
   \alpha = 0,1 \\
 {1}/{d+1}, & \text{if }\alpha\neq 0,1.
 \end{cases}
 \end{equation*}

More precisely, (I1) and (I2) are the following:
\begin{itemize}
\item[(I1)] The sequence $\{\epsilon_n\}_{n\in \mathbb{N}}$ is such that $\epsilon_n>0$ and $\epsilon_n\rightarrow 0$.  When $\alpha =0$ or $\alpha =1$, we also suppose that
\begin{equation*}
\sum_{n=1}^\infty \exp(- n \epsilon_n^{(d+1)/2}) < \infty.
\end{equation*}
However, when $\alpha\neq 0, 1$, we suppose that
\begin{equation*}
\sum_{n=1}^\infty n \exp(- n \epsilon_n^{d+1}) < \infty.
\end{equation*}
\end{itemize}

\begin{itemize}
\item[(I2)] The sequence $\{\epsilon_n\}_{n\in \mathbb{N}}$ is such that $\epsilon_n>0$ and $\epsilon_n\rightarrow 0$.  When $\alpha = 0$ or $\alpha = 1$, we also suppose that
\begin{equation*}
\begin{aligned}
&\lim_{n \to \infty} \frac{\sqrt{2 \log \log n}}{\sqrt{n}}
\frac{1}{\epsilon_n} = 0, \hspace{0.3cm} & {\rm if \ }d=1\\
&\lim_{n \to \infty} \frac{(\log n)^{3/4}}{n^{1/2}}
\frac{1}{\epsilon_n} = 0 \hspace{0.3cm} &{\rm if \ }d=2 \\ 
&\lim_{n \to \infty} \frac{(\log n)^{1/d}}{n^{1/d}} \frac{1}{\epsilon_n} = 0
\hspace{0.3cm} &{\rm \ if \ }d \geq 3.
\end{aligned}
\end{equation*}
However, when $\alpha\neq 0,1$, we suppose that
\begin{equation*} \label{eq:epsratealpha}
 \sum_{n=1}^\infty n \exp(-n\epsilon_n^{d+1})<\infty.
\end{equation*}
\end{itemize}
In Subsection \ref{discussion} we discuss the motivation behind
these assumptions in more detail.

Now, given the set $\mathcal{X}_n$ and a sequence $\{ \epsilon_n \}_{n \in \mathbb{N}}$, denote by $W_n$ the weight matrix with entries given by (\ref{eq:weightdef}), and denote by $\mathcal{G}_n =
(\mathcal{X}_n,W_n)$ the corresponding weighted random geometric graph.
 We let
\begin{center}
$Q_n$ denote the modularity functional (\ref{eq:generalmodularity}) 
corresponding to $\mathcal{G}_n$. 
\end{center}

Consider a partition $\mathcal{U} = \{U_k\}_{k=1}^K$ of the domain $D$. For any $n \geq 1$, this induces a partition $\mathcal{U}_n = \{U_{n,k}\}_{k=1}^K$ of the sample $\mathcal{X}_n$, where
\begin{equation*}
U_{n,k} = U_k \cap
\mathcal{X}_n, 
\end{equation*}
for $1 \leq k \leq K$.

\begin{theorem}[Asymptotics] \label{thm:asymptotics}
Suppose $\{\epsilon_n\}_{n\in\mathbb{N}}$ satisfies condition (I1).  Fix $K\geq 1$ and let $\mathcal{U}=\{U_k\}_{k=1}^K$ be a partition of $D$ where each $U_k$ is a subset of finite perimeter, ${\rm Per}(U_k;\rho^2)<\infty$.  Let $\mathcal{U}_n$ be the induced partition of $\mathcal{X}_n$ for $n\geq 1$.
Then, the modularity $Q_n(\mathcal{U}_n)$ satisfies
\begin{equation} \label{eq:asympresult1}
1 - 1/K - Q_n(\mathcal{U}_n) \xrightarrow{a.s.} \sum_{k=1}^K
\Big(\mu(U_k) - 1/K\Big)^2,
\end{equation}
as $n\rightarrow \infty$, where $d\mu(x) = \frac{\rho^{1+\alpha}(x)\, dx}{\int_D
  \rho^{1+\alpha}(x)\, dx}$. 

Further, if $\sum_{k=1}^K (\mu(U_k) - 1/K)^2 = 0$, we
have
\begin{equation} \label{eq:asympresult2}
\frac{1 - 1/K - Q_n(\mathcal{U}_n)}{\epsilon_n} \xrightarrow{a.s.}
 C_{\eta,\rho} \sum_{k=1}^K\per(U_k; \rho^2),
\end{equation}
as $n\rightarrow\infty$, where 
\begin{equation*}
C_{\eta,\rho} = \frac{\int_{\mathbb{R}^n} \eta(x)|x_1|\,dx }{2\int_D
\rho^2(x)\, dx},
\end{equation*}
and $x_1$ denotes the first coordinate of $x$.
\end{theorem}

One way to interpret these law of large numbers limits is that the nonnegative quantity, a `balance' term, 
\begin{equation*}
\sum_{k=1}^K
\Big(\mu(U_k) - \frac{1}{K}\Big)^2,
\end{equation*}
 is a measure of how
balanced the partition $\mathcal{U}$ is with respect to the
measure $\mu$. In our model, the limiting modularity of a partition is predominantly determined by the number of clusters and
the extent to which they are balanced.

We state a corollary of Theorem \ref{thm:asymptotics}, which follows by considering balanced partitions where $K$ is not restricted.
\begin{corollary}
\label{cor:main}
Suppose the assumptions of Theorem \ref{thm:asymptotics} are satisfied.  Let $Q_{n}^* = \max_{\mathcal{U}_n} Q_{n}(\mathcal{U}_n)$ denote the maximum modularity
associated to
$\mathcal{G}_n$, where the maximum is over all partitions $\mathcal{U}_n$ of $\mathcal{X}_n$. 
Then, we have 
\begin{equation*}
Q_n^* \xrightarrow{a.s.} 1,
\end{equation*}
as $n \to \infty$.
\end{corollary}

Our second main result is a characterization of the
behavior of optimal clusterings $U_n^* \in \argmax_{ |\mathcal{U}_n| \leq
    K} Q_n(\mathcal{U}_n)$,
as $n \to \infty$. To this end, we introduce a suitable notion of convergence.

To a sequence of sets $\{U_n\}_{n\in \mathbb{N}}$ with $U_n \subset \mathcal{X}_n$,
we associate the `graph measures' $\{\gamma_n\}_{n\in\mathbb{N}}$, where 
$\gamma_n =   \frac{1}{n}\sum_{i=1}^n \upsilon_{(X_i,\mathds{1}_{U_{n}}(X_i))},$
and $\upsilon_\cdot$ denotes a point mass.  In words, the measure
$\gamma_{n}$ is the distribution of the
graph of $\mathds{1}_{U_{n}}$ under the empirical measure $\nu_n$ on $\mathcal{X}_n$.  Let also $U\subset D$ and define $\gamma$ as the distribution of the graph of $\mathds{1}_U$ under $\nu$.
   We will write, with respect to a realization of $\{X_i\}_{i \in
     \mathbb{N}}$, that
\begin{equation}
\label{single_set_conv_defn}
U_n {\rm \ converges \ {\it weakly} \ }({\rm denoted \ }\xrightarrow{w}) {\rm \ to \ } U \  {\rm \ if \ }\gamma_n \xrightarrow{w} \gamma.
\end{equation}

Correspondingly, consider a sequence of partitions $\mathcal{U}_n = \{ U_{n,k} \}_{k=1}^K$ of $\mathcal{X}_n$, and a partition $\mathcal{U} = \{U_k\}_{k=1}^K$ of $D$. Since the sets in the collections $\mathcal{U}_n$ and $\mathcal{U}$ are unordered, we say that
\begin{center}
 $\mathcal{U}_n$ converges \textit{weakly} (denoted $\xrightarrow{w}$)
to $\mathcal{U}$
\end{center}
 if there exists a sequence $\{\pi_n
\}_{n \in \mathbb{N}}$ of
permutations $\pi_n : \{1,\ldots,K\} \to \{1,\ldots,K\}$ such that 
\begin{equation} \label{eq:convergeeq}
\gamma_{n,\pi_n k} \xrightarrow{w}\gamma_k, \hspace{0.5cm}
\text{ for $k = 1,\ldots,K.$}
\end{equation}

\begin{theorem}[Optimal Clusterings] \label{thm:mainthm}
Suppose $\{\epsilon_n\}_{n\in\mathbb{N}}$ satisfies condition (I2). 

\noindent 
Let $\mathcal{U}_n^* \in \argmax_{ |\mathcal{U}_n| \leq
    K} Q_n(\mathcal{U}_n)$ be an optimal partition of $\mathcal{X}_n$, for $n\geq 1$.
    Let also, with respect to problem \eqref{prob:continuumk}, $\phi  = \rho^2$ and $d\mu= \rho^{1+\alpha}dx/\int_D\rho^{1+\alpha}(x)dx$.
    
If
$\mathcal{U}^*$ is the unique solution (modulo reordering of its
constituent sets) of
problem (\ref{prob:continuumk}),
then
a.s.
$\mathcal{U}_n^*$ converges weakly to
$\mathcal{U}^*$, in the sense of (\ref{eq:convergeeq}). 

If there is more than one solution
to the problem (\ref{prob:continuumk}), then a.s. $\mathcal{U}_n^*$ converges weakly along a subsequence, in the sense of (\ref{eq:convergeeq}), to a solution $\mathcal{U}^*$ of (\ref{prob:continuumk}).
\end{theorem}

Figure \ref{fig:binary_convergence} illustrates this result. We remark
that, in the proof of Theorem \ref{thm:mainthm}, we will in fact show a
stronger form of convergence, in terms of Wasserstein,
Kantorovich-Rubenstein type $(TL^1)^K$ distances, via a
Gamma convergence statement (Theorem \ref{thm:maingamma}) for certain
energy functionals.

\begin{figure}[h]
  \centering
    \includegraphics[trim={0cm 4cm 0cm 4cm},width=\textwidth]{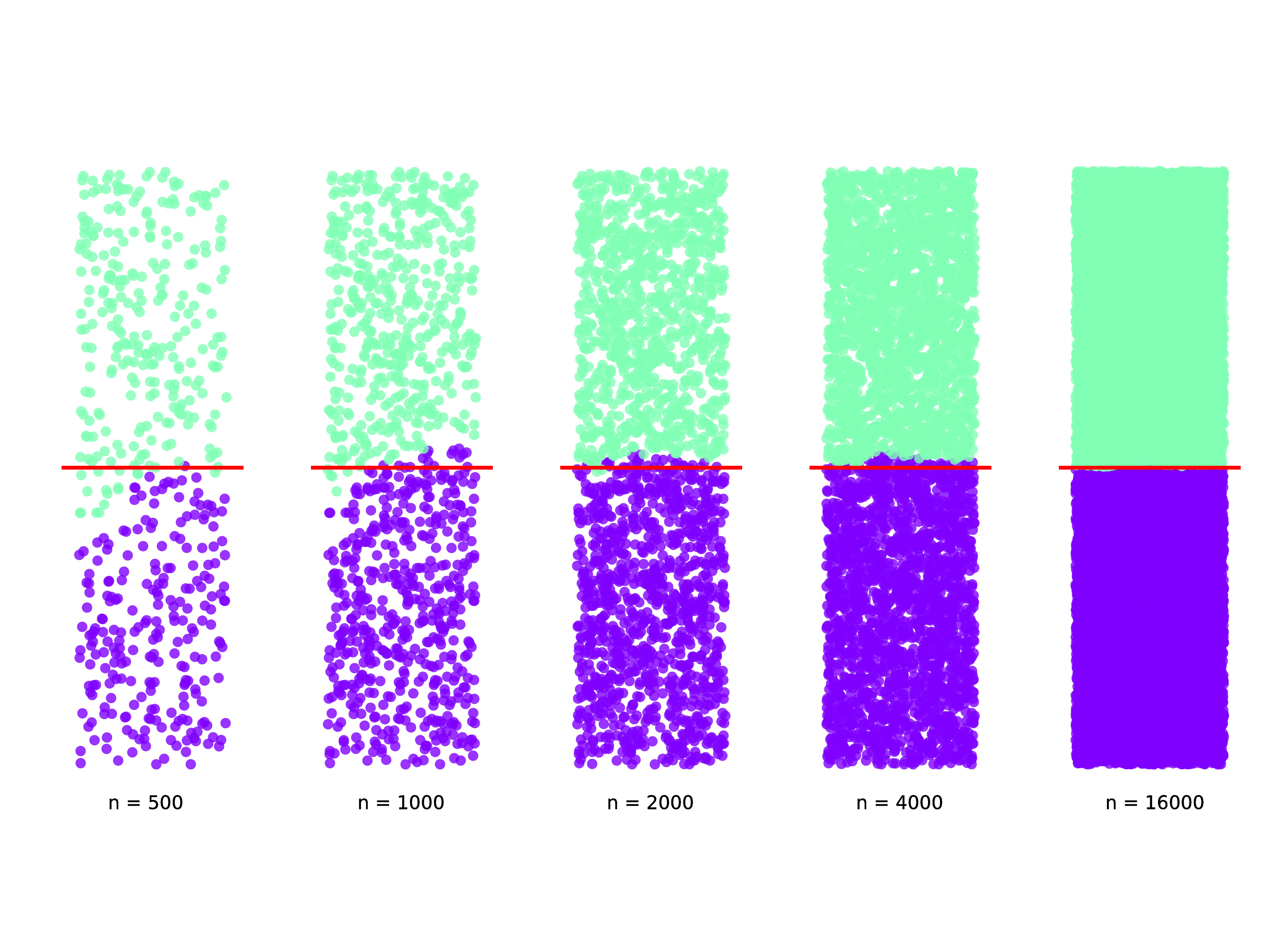}
    \caption{Binary partitions produced by modularity clustering of
      random geometric graphs with various values of $n$ on the domain $D = (0,1) \times (0,4)$, using the algorithm of \cite{newman2006finding}.
      Here $\alpha = 1$, the density $\rho$ is uniform, $\eta = \mathds{1}_{|x| < 1}$, and
      $\epsilon_n = n^{-0.3}$. The red lines indicate the optimal cut
      associated with the continuum partitioning problem for $K=2$.\label{fig:binary_convergence}}
\end{figure}

We also note that the parameter $\alpha$, which parametrizes a family of null models in the modularity functional \eqref{eq:generalmodularity}, appears in the balancing
measure $\mu$ with respect to the continuum problem. As $\alpha$ increases, the measure $\mu$ puts more mass
near the modes of the density $\rho$, of course altering the optimal clusterings, as illustrated in Figure \ref{fig:alpha_varying}.

\begin{figure}[h]
    \centering
    \begin{subfigure}[b]{0.23\textwidth}
        \includegraphics[width=\textwidth]{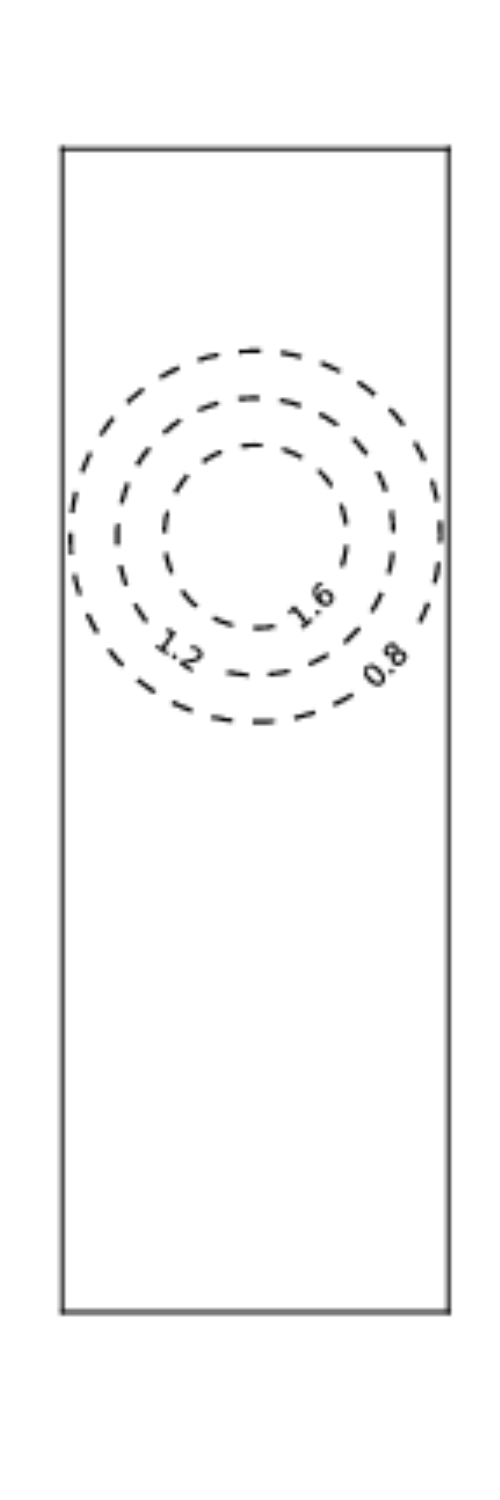}
        \caption{Level sets of $\rho$}
    \end{subfigure}
    ~ 
    \begin{subfigure}[b]{0.23\textwidth}
        \includegraphics[width=\textwidth]{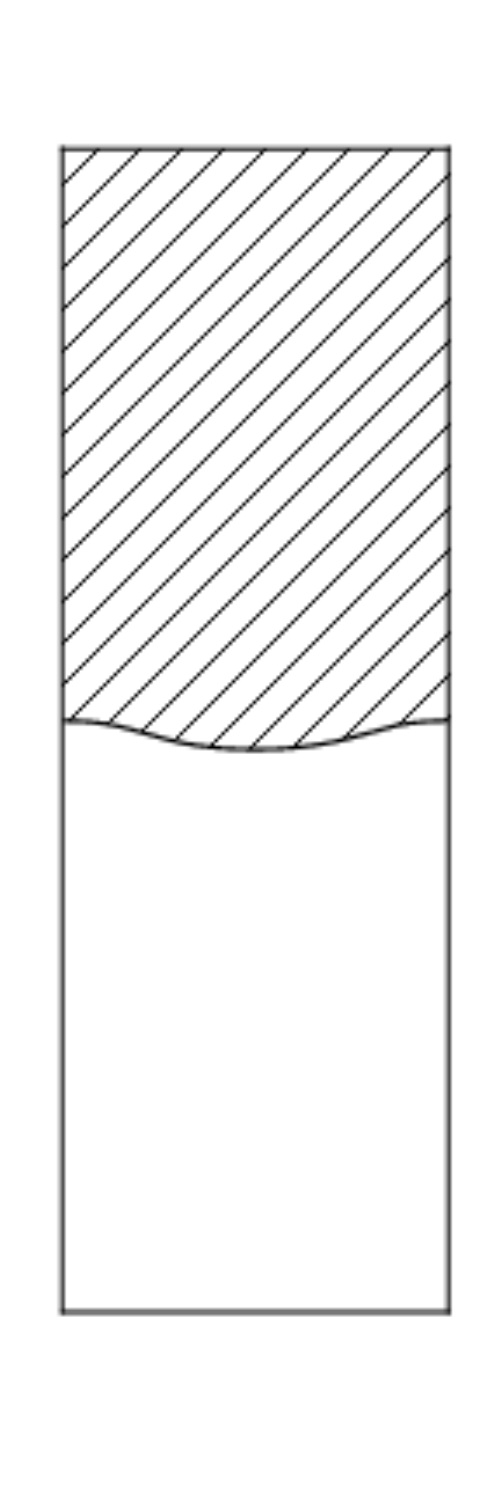}
        \caption{$\alpha = -1$}
    \end{subfigure}
    ~ 
    \begin{subfigure}[b]{0.23\textwidth}
        \includegraphics[width=\textwidth]{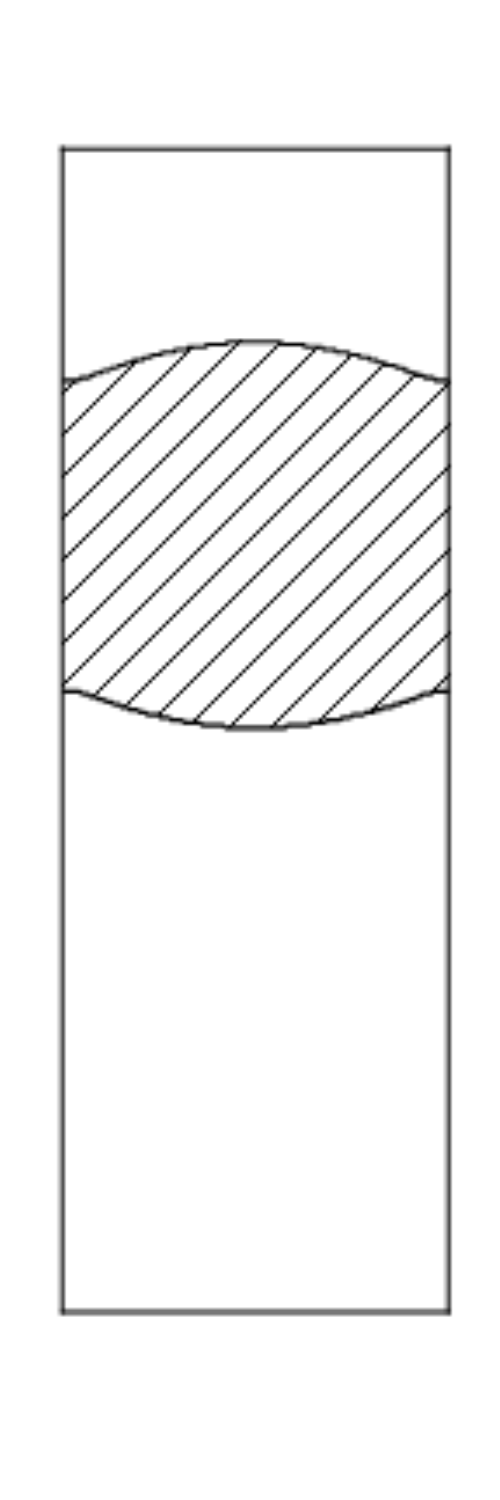}
        \caption{$\alpha = 0$}
    \end{subfigure}
    \begin{subfigure}[b]{0.23\textwidth}
        \includegraphics[width=\textwidth]{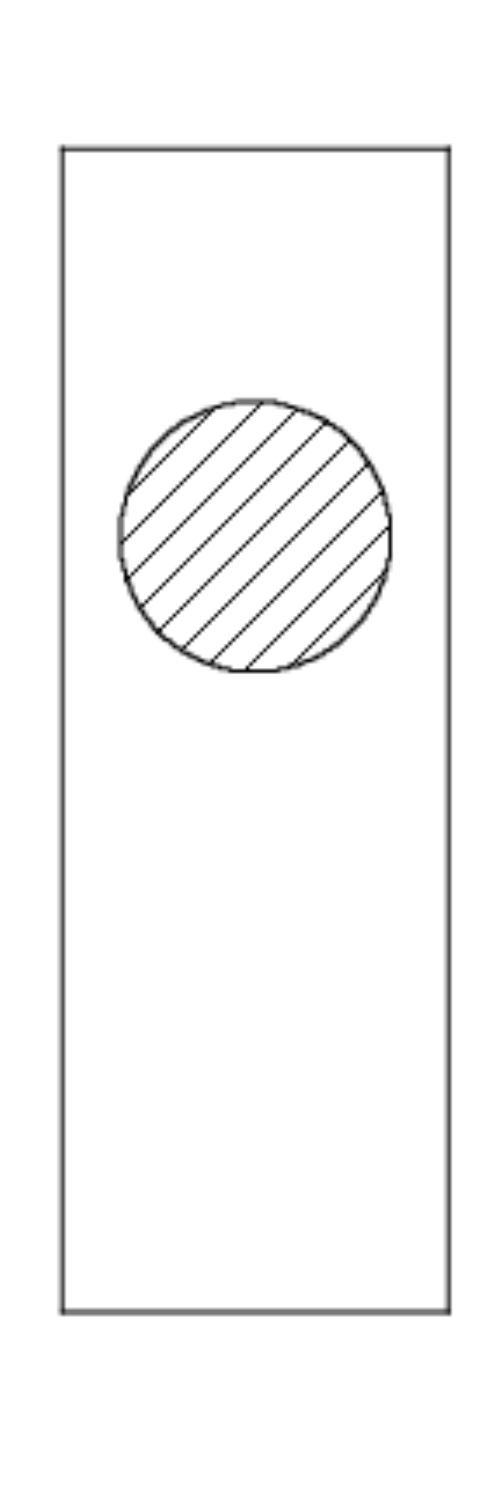}
        \caption{$\alpha = 1$}
    \end{subfigure}
    \caption{Behavior of Problem \eqref{prob:continuumk} for $K = 2$
      and various $\alpha$ on $D = (0,1) \times
      (0,3)$, computed using a modified version of the algorithm in \cite{Oudet}. Here $\phi = \rho^2$, $d\mu = \rho ^{1+\alpha} (x)\,dx /
      \int_D \rho ^{1+\alpha} (x)\, dx$, for $\rho(x) \propto \min(2 \exp({ - 4 ||x-z||^2}), 1/2)$, with $z
      = (1/2,2)$. \label{fig:alpha_varying}}
\end{figure}

We also observe that Theorem \ref{thm:mainthm} implies that the proportion of
  misclassified points vanishes almost surely.
  
  \begin{corollary} 
  \label{cor:proportion}
  Consider the setting of Theorem \ref{thm:mainthm}.  Let $\{\mathcal{U}^*_{n_m}\}_{m\in \mathbb{N}}$ be a subsequence of optimal $Q_{n_m}$-modularity partitions of the sample space which a.s. converges weakly, in the sense of \eqref{eq:convergeeq}, to a solution $\mathcal{U}^*$ of \eqref{prob:continuumk}.  Then, the proportion of correctly classified points converges to $1$. That is, a.s. as $m\rightarrow\infty$, 
  $$p_{n_m}:=   \max_{1\leq k\leq K} \frac{ | \{ x \in \mathcal{X}_{n_m} : x \in U^*_k \cap U^*_{n_m, \pi_{n_m}(k)} \} |}{|\{ x \in
  \mathcal{X}_{n_m} : x \in U^*_k \}|} \rightarrow 1.$$
    
  \end{corollary}
  
  \begin{proof}  It is enough to show, with respect to a sequence of
    sets $V_m\subset \mathcal{X}_{n_m}$ and a set $V\subset D$ where
    a.s. $V_m \xrightarrow{w} V$, in the sense of \eqref{single_set_conv_defn}, that
  $$  \frac{ | \{ x \in \mathcal{X}_{n_m} : x \in V \cap V_{m} \} |}{|\{ x \in
  \mathcal{X}_{n_m} : x \in V \}|} \rightarrow 1 \ \ \ {\rm a.s.}$$
  as $m\rightarrow\infty$. Then, the statement in the corollary would follow by application of this limit with $V_m = U^*_{n_m,\pi_{n_m}(k)}$ and $V = U^*_k$ for $1\leq k\leq K$.
  
  In terms of the measures $\gamma_m$ and $\gamma$, which govern the graphs of $\mathds{1}_{V_m}$ and $\mathds{1}_{V}$ under $\nu_n$ and $\nu$ respectively, we can write
  $$\frac{ | \{ x \in \mathcal{X}_{n_m} : x \in V \cap V_{m} \} |}{|\{ x \in
  \mathcal{X}_{n_m} : x \in V \}|} = \frac{ \int_{D \times \mathbb{R}}
  \mathds{1}_V(x)y\,d\gamma_m(x,y) }{ \int_{D \times \mathbb{R}} \mathds{1}_V(x)\, d\gamma_m(x,y)}.$$
  Since a.s. $\gamma_m\xrightarrow{w}\gamma$ as $m\rightarrow\infty$, by approximating $(x,y)\mapsto \mathds{1}_V(x)y$ and  $(x,y) \mapsto \mathds{1}_V(x)$ by bounded, continuous functions, we have
  $$ \frac{ \int_{D \times \mathbb{R}}
  \mathds{1}_V(x)y\,d\gamma_m(x,y) }{ \int_{D \times \mathbb{R}} \mathds{1}_V(x)\, d\gamma_m(x,y)} \xrightarrow{a.s.} 
  \frac{ \int_{D \times \mathbb{R}}
  \mathds{1}_V(x)y\,d\gamma(x,y) }{ \int_{D \times \mathbb{R}}
  \mathds{1}_V(x)\, d\gamma(x,y)} = 1,$$
  as $m\rightarrow\infty$, concluding the argument.
  \end{proof}  
  
\subsection{Discussion}
\label{discussion}
\mbox{}

{\bf 1.}  As alluded to in the introduction, the phenomena shown in Corollary \ref{cor:main} for random geometric graphs has been considered before in other models. Indeed, in
\cite{guimera2004modularity} the authors provide heuristic arguments
for the limiting behavior $Q_n^* \to 1$ under two regimes: i) when the graphs are
regular lattices, and ii) when the graphs are Erd\H{o}s-R\'{e}nyi graphs
with edge probability $p = 2/n$. In \cite{good2010performance}, the
authors derive the limiting behavior $Q_n^* \to 1$ under a sparse
graph model, in which modules of some characteristic size are adjoined
to the graph. Further, these asymptotics are consistent with the empirical
results associated with large
real-world graphs \cite{blondel2008fast}.

\vskip .1cm
{\bf 2.} It has been observed in the literature that modularity optimization may fail to identify clusters smaller than a certain level, depending on the total size and interconnectedness of the graph. In other words, modularity possesses a `resolution limit' in terms of its clustering (cf. \cite{fortunato2007resolution}, \cite{good2010performance}).
An extreme example is when the graph contains a pair
of cliques (complete subgraphs) connected by a single edge,
but modularity would lump them into a
common cluster (cf. Figure 3 of
\cite{fortunato2007resolution}).

In \cite{reichardt2006statistical}, the authors consider a variant of the modularity given by
\begin{equation}
Q^{\lambda} = \frac{1}{2m} \sum_{i,j} \Big( W_{ij} - \lambda \frac{d_i d_j}{2m}\Big) \delta(c_i,c_j)
\end{equation}
where $\lambda$ is a parameter. In \cite{lancichinetti2011limits}, the parameter $\lambda$
is related to the resolution limit phenomena: Namely, higher values of $\lambda$ allow for smaller cluster sizes.

The methods used to prove Theorem \ref{thm:mainthm} give the following asymptotic behavior of optimal $Q^\lambda$ modularity clusterings:
When $\lambda=\lambda_n := \kappa \epsilon_n^{\beta}$ is scaled with $n$, for $\beta\geq 0$ and $\epsilon_n$ satisfying (I2), 
three distinct possibilities
arise for the limiting problem.  When $0\leq \beta < 1$, the continuum partitioning
problem remains as it is in (\ref{prob:continuumk}).  When $\beta = 1$,
the hard
constraints $\mu(U_k) = 1/K$ for $1\leq k\leq K$ of the limiting problem get replaced by a soft balancing condition, resulting
in  \begin{equation}
\underset{\mathcal{U}}{\text{minimize}}\, \frac{1}{2} \sum_{k=1}^K
  \per(U_k;\phi) + \kappa \sum_{k=1}^K (\mu(U_k) - 1/K)^2.
\end{equation}
When $\beta > 1$, the continuum problem degenerates to a perimeter
minimization problem with no balancing condition, which has as its
solution a single global cluster $D$ (and the other $K-1$ sets being empty).

\vskip .1cm
{\bf 3.} One can ask about the reasons behind the assumptions (I1) and (I2).  With respect to Theorem \ref{thm:asymptotics}, a lower bound for $\epsilon_n$ should be informed by the fluctuations of the functional.  In fact, the variance of $Q_n(\mathcal{U}_n)$ can be seen to be of order $(n^2\epsilon_n^{d+1})^{-1}$ when $\alpha = 0,1$ (by a computation with formula \eqref{eq:oneminusmod}), so that condition (I1) makes sense.   However, when $\alpha\neq 0,1$, a worse bound is useful to control the nonlinearity of the functional.

On the other hand, assumption (I2) in Theorem \ref{thm:mainthm} is
informed by the connectivity radius of the random geometric
graphs. For instance, if $\epsilon_n$ were to vanish too quickly, the
underlying graphs would contain $O(n)$ disconnected components
(cf. Theorem 13.25 in \cite{penrose2003random}). Then, presumably, one
would be able to find a $\mathcal{U}_n^*$ such that $(1/2 -
Q_n(\mathcal{U}_n^*)) / \epsilon_n \xrightarrow{a.s.} 0$ and
consequently obtain a continuum cluster point $\mathcal{U}^*$. This would
be a contradiction, as the resulting partition $\mathcal{U}^*$ would
have zero perimeter --  in other words, one of the sets in
$\mathcal{U}^*$ would be $D$ itself -- and so could not satisfy the balance
conditions. This is a version of the argument in Remark 1.6 of \cite{trillos2014continuum}.

The threshhold that $\epsilon_n$ should be larger than for the graphs to be connected is known:  In $d\geq 1$, it is of order $(\log(n))^{1/d}/n^{1/d}$ (cf. \cite{penrose2003random}).  Viewed from this lens, when $\alpha=0,1$, condition (I2) is more optimal than when $\alpha\neq 0,1$, again due to the nonlinearity of the functional in this case.

{\bf 4.} We briefly discuss the assumptions on $\rho$, $D$, and
$\eta$. The proof of Theorem \ref{thm:mainthm} makes use of certain
`transport maps' (cf. Proposition \ref{thm:tmaps}). For $d \geq 2$, we use the optimal transport results of
\cite{trillos2014rate}, which require that $\rho$ be bounded above and below by positive constants.
 Likewise, the assumptions on $D$ are made there so that comparisons to results on cubes may be made. On the other hand, for $d=1$, it is not necessary that $\rho$ be bounded
below to define a suitable transport map--here, the technical condition required is \eqref{eq:shorackassumption}. However, in
all dimensions, we require a lower bound on $\rho$, as this enables us
to handle the general $\alpha\neq 0,1$ case via a
Lipschitz inequality for the map $x \mapsto x^{\alpha}$ (cf. Lemma
\ref{lemma:quadasconvergea}). The boundedness of $D$ is also used
in several
intermediate technical results. The Lipschitz continuity of $\rho$ is used for handling the `balance' term in the proof of Theorem \ref{thm:mainthm} (principally in Lemma \ref{lemma:lambdaepslim}, by way of Lemma \ref{lemma:mollified}). We remark that this condition could be weakened to H\"{older} continuity, with exponent greater than $1/2$.

With respect to $\eta$, the radial and monotone assumptions are convenient in relating certain graph functionals to their nonlocal
analogues (cf. Lemma \ref{lemma:glambdalimit}). The continuity at zero
is used in the proof of the compactness property, Lemma
\ref{lemma:gtvcompact}. Finally, we remark that the compact support of
$\eta$ allows to analyze behavior near the boundary of $D$, although this assumption could be weakened to a
suitable condition on the decay of $\eta$ at infinity.

\section{Preliminaries}
\label{preliminaries}
Before entering into the main derivations, we first discuss in Subsection \ref{framework} the $TL^1$ topology and framework, introduced by Garc\'{\i}a Trillos and Slep\v{c}ev in \cite{trillos2014continuum}, and connections to weak convergence of graph measures.  Then,
in Subsection \ref{gamma_conv}, we define a variant of Gamma convergence for random energy functionals that we will use to prove Theorem \ref{thm:mainthm}.
 
\subsection{$TL^1$ topology and framework}
\label{framework}

Given a measurable space $S\subset \mathbb{R}^d$, we let $\mathcal{B}(S)$
denote the Borel $\sigma$-algebra on $S$, and similarly
let $\mathcal{P}(S)$ denote the set of Borel probability measures on
$S$. 
Also, given two spaces, $S_1$ and $S_2$, a measurable
map $T : S_1 \to S_2$, and a measure $\mu \in
\mathcal{P}(S_1)$, we define the \textit{push-foward} $T_{\sharp} \mu
\in \mathcal{P}(S_2)$ by
\[ T_{\sharp} \mu (A) = \mu(T^{-1}(A)), \hspace{0.5cm} \text{for } A
\in \mathcal{B}(S_2). \]
In particular, $T_{\sharp}\mu$ is the distribution of $TX$ where $X$ has distribution $\mu$.

Given
measures $\mu,\theta \in \mathcal{P}(S)$, recall that a coupling between $\mu$ and $\theta$ is a probability measure $\pi$ on $S \times S$ such that the marginal with respect to the first variable is $\mu$, and the marginal with respect to the second variable is $\theta$. Consider the set of couplings 
\begin{align*}
\Gamma(\mu,\theta) \coloneqq \{ \pi \in \mathcal{P}(S \times S) : (\forall U \in \mathcal{B}(S) \ \pi(U\times S) = \mu(U) \ \mbox{and} \ \pi(S \times U) = \theta(U) \}.
\end{align*}
Define the distance on $\mathcal{P}(S)$ by
\begin{equation*}
d_{1,S}(\mu,\theta) \coloneqq \inf_{\pi\in \Gamma(\mu,\theta)}\int_S |x-y| \ d\pi(x,y).
\end{equation*}
This is a metric on the subset of probability measures in $\mathcal{P}(S)$ with finite first moment.

When $S$ is complete, a case of a more general result (see Theorem 6.9 \cite{Villani_old_new}) is the following:  Let $\{\mu_n\}_{n\in \mathbb{N}}$ and $\mu$ be measures in $\mathcal{P(S)}$.  Then, as $n\rightarrow\infty$,
\begin{equation}\label{one}
\mu_n \xrightarrow{w}\mu \ \ {\rm if \ \ } d_{1,S}(\mu_n,\mu)\rightarrow 0.
\end{equation}

Now, as in \cite{trillos2014continuum}, to understand weak convergence of `graph measures', define the space $TL^1(S)$ by
\begin{align*}
TL^1(S) \coloneqq \{ (\mu,f) : \mu \in \mathcal{P}(S), \|f\|_{L^1(S,\mu)}<\infty\},
\end{align*}
and, for $(\mu,f)$ and $(\theta,g)$ in $TL^1(S)$, define the distance
\begin{equation*}
d_{TL^1,S}((\mu,f), (\theta,g)) \coloneqq \inf_{\pi \in \Gamma(\mu,\theta)} \int \int _{S \times S} |x-y| + |f(x)-g(y)| \ d\pi(x,y).
\end{equation*}

One may identify an element $(\mu,f)\in TL^1(S)$ with a graph measure $(Id\times f)_{\sharp}\mu\in \mathcal{P}(S\times \mathbb{R})$, whose support is contained in the graph of $f$.  
Consider now, with respect to $(\mu,f), (\theta,g)\in TL^1(S)$, the graph measures $\gamma = (Id\times f)_{\sharp}\mu, \widetilde\gamma = (Id\times g)_{\sharp}\theta \in \mathcal{P}(S\times \mathbb{R})$.  
It may be seen (cf. Proposition 3.3 of \cite{trillos2014continuum}) that 
\begin{equation}
\label{two}
d_{1,S\times \mathbb{R}}(\gamma,\widetilde\gamma) =d_{TL^1, S}((\mu,f), (\theta,g)).
\end{equation}

We now restrict $S$ to be the bounded domain $D\subset\mathbb{R}^d$ introduced in Subsection \ref{results}.  We will abbreviate $TL^1:= TL^1(D)$.  Then, for $(\mu,f)\in TL^1$, the graph measure $\gamma = (Id\times f)_\sharp\mu$ has a finite first moment in that $\int_D |x| + |f(x)|d\mu(x)<\infty$.  Hence, by \eqref{two},
$TL^1$ can be viewed as a metric space with metric $d_{TL^1, D}$. 

With respect to graph measures $\gamma'$ and $\gamma''$ on $D\times \mathbb{R}$, consider their extensions $\bar\gamma'$ and $\bar\gamma''$ to $\overline{D\times \mathbb{R}}$ by setting $\bar\gamma'(\partial D\times \mathbb{R}) = \bar\gamma''(\partial D\times \mathbb{R})=0$.  Then, the distance 
\begin{equation}
\label{three}
d_{1,\overline{D\times \mathbb{R}}}(\bar\gamma', \bar\gamma'')\leq d_{1,D\times \mathbb{R}}(\gamma',\gamma'').
\end{equation}
  Suppose now $(\mu_n, f_n)\xrightarrow{TL^1} (\theta, g)$, and $\gamma_n$ and $\gamma$ are the associated graph measures for $n\geq 1$.  Then, by \eqref{two} and \eqref{three}, $d_{1,\overline{D\times \mathbb{R}}}(\bar\gamma_n,\bar\gamma) \leq d_{1,D\times \mathbb{R}}(\gamma_n,\gamma)\rightarrow 0$ as $n\rightarrow\infty$.  In particular, as $\overline{D\times \mathbb{R}}$ is complete, by \eqref{one}, $\bar\gamma_n \xrightarrow{w}\bar\gamma$ in $\mathcal{P}(\overline{D\times \mathbb{R}})$, and so $\gamma_n\xrightarrow{w}\gamma$ in $\mathcal{P}(D\times \mathbb{R})$, as $n\rightarrow\infty$.

We now make a remark on definition \eqref{eq:convergeeq} in connection with the product space $(TL^1)^K$.  Fix a realization $\{X_i\}_{i\in \mathbb{N}}$.  Recall the empirical measures $\nu_n$ and probability measure $\nu$ on $D$ from the beginning of Subsection \ref{results}.
Let $\mathcal{U}_n = \{U_{n,k}\}_{k=1}^K$ be a
partition of $\mathcal{X}_n$ for $n\geq 1$, and $\mathcal{U} =
\{U_k\}_{k=1}^K$ be a partition of $D$.  We say the sequence $(\nu_n,
\mathcal{U}_n):=\big((\nu_n, \mathds{1}_{U_{n,k}})\big)_{k=1}^K$ converges in $(TL^1)^K$ to
$(\nu, \mathcal{U}):=\big((\nu,\mathds{1}_{U_k})\big)_{k=1}^K$ if $(\nu_n, \mathds{1}_{U_{n,k}})\xrightarrow{TL^1} (\nu, \mathds{1}_{U_k})$ for $1\leq k\leq K$.
Hence, by the comment below \eqref{three},
as convergence in the metric
$d_{TL^1,D}$ implies weak convergence in $\mathcal{P}(D\times
\mathbb{R})$, we obtain
\begin{equation}
\label{equivalent_convergence}
\mathcal{U}_n \xrightarrow{w} \mathcal{U} \ {\rm if } \ \ (\nu_n,\mathcal{U}_n) \xrightarrow{(TL^1)^K} (\nu, \mathcal{U}),
\end{equation}
in the sense of definition \eqref{eq:convergeeq}, by choosing the
identity permutations

We now discuss when this convergence may be formulated in terms of transportation
maps.
We say that a measurable function $T : D \to D$  is a {\it transportation map} between the measures $\mu \in
\mathcal{P}(D)$ and $\theta \in \mathcal{P}(D)$ if $\theta = T_{\sharp} \mu$. 
In this context, for $f \in
L^1(\theta)$, the change of variables formula holds
\begin{equation*}
\int_D f(y)\,d\theta(y) = \int_D f(T(x))\,d\mu(x).
\end{equation*}

A transportation map $T$ yields a coupling $\pi_T \in \Gamma(\mu,\theta)$ defined by $\pi_T \coloneqq (Id \times T)_{\sharp} \mu$ where $(Id \times T)(x) = (x,T(x))$.  It is well known, when $\theta$ is absolutely continuous with respect to Lebesgue measure on $D$, that the infimum $d_{1,D}(\mu,\theta)$ can be achieved by a coupling $\pi_T$ induced by a transportation map $T$ between $\mu$ and $\theta$. 
Indeed, we note briefly that this is only one result among many others which relate various `Monge' and `Kantorovich' distances via optimal transport theory. See \cite{Villani_old_new} and references therein; see also \cite{Ambrosio_Gigli_Savare}, \cite{Villani_intro_transportation}.

We will say that a sequence $\{T_n\}_{n\in \mathbb{N}}$ of transportation maps, with $T_{n\sharp} \theta = \theta_n$, with respect to a sequence of measures $\{\theta_n\}_{n\in \mathbb{N}}\subset \mathcal{P}(D)$, is \textit{stagnating} if
\begin{equation*}
\lim_{n \to \infty} \int_D |x-T_n(x)|\, d\theta(x) = 0.
\end{equation*}

The following is Proposition 3.12 in \cite{trillos2014continuum}.
\begin{lemma}\label{lemma:tl1convergence}
Consider a measure $\theta \in \mathcal{P}(D)$ which is absolutely continuous with respect to the Lebesgue measure. Let $(\theta,f) \in TL^1(D)$ and let $\{(\theta_n,f_n)\}_{n \in \mathbb{N}}$ be a sequence in $TL^1(D)$. The following statements are equivalent:
\begin{enumerate}
\item[(i)] $(\theta_n,f_n) \xrightarrow{TL^1} (\theta,f)$.
\item[(ii)] $\theta_n \xrightarrow{w} \theta$ and there exists a stagnating sequence of transportation maps $T_{n\sharp}\theta = \theta_n$ such that
\begin{equation}\label{eq:transportlim}
\lim_{n \to \infty} \int_D |f(x)-f_n(T_n(x))|\, d\theta(x) = 0.
\end{equation}
\item[(iii)] $\theta_n \xrightarrow{w} \theta$ and for any stagnating sequence of transportation maps $T_{n\sharp} \theta= \theta_n$, the convergence (\ref{eq:transportlim}) holds.

\end{enumerate}
\end{lemma}

 In order to make use of the above result on $TL^1$ convergence, we will
need to find a suitable stagnating sequence $\{ T_n \}_{n \in \mathbb{N}}$ of
transportation maps. 

\begin{proposition}\label{thm:tmaps} Recall, from the beginning of Subsection \ref{results}, the assumptions on the probability measure $\nu$ on $D$, and that $\nu_n$ denotes the empirical measure corresponding to i.i.d. samples drawn from $\nu$.

Then, 
there is a constant $C > 0$ such that, with respect to realizations of $\{X_i\}_{i\in \mathbb{N}}$ in a probability $1$ set $\Omega_0$, a sequence of transportation maps $\{T_n\}_{n \in \mathbb{N}}$ exists where $T_{n\sharp} \nu = \nu_n$ and
\begin{equation*}
\limsup_{n \to \infty} \frac{n^{1/2} \|Id - T_n\|_{L^\infty}}{(2 \log
  \log n)^{1/2}} \leq C \hspace{0.4cm}  {\rm if \ }d=1,
\end{equation*}
\begin{equation*}
\limsup_{n \to \infty} \frac{n^{1/d} \|Id - T_n\|_{L^\infty}}{(\log
  n)^{3/4}} \leq C \hspace{0.4cm}  {\rm if \ }d=2,
\end{equation*}
\begin{equation*}
 \limsup_{n \to \infty} \frac{n^{1/d}\|Id-T_n\|_{L^\infty}}{(\log n)^{1/d}} \leq C \hspace{0.4cm}  {\rm if \ }d \geq 3.
\end{equation*}
\end{proposition}

\begin{proof}
We prove the $d = 1$ case in the appendix (Proposition \ref{lemma:transportd1}), as a consequence of quantile transform results for the empirical measure, making use of the technical conditions assumed on $\rho$.  In Garc{\'\i}a Trillos and Slep{\v{c}}ev \cite{trillos2014rate}, the $d = 2$ and $d \geq 3$ cases are first discussed, in the context of concentration estimates in the literature when $D$ is a cube and $\nu$ is the uniform measure,
and then proved for general $D$ and nonuniform $\nu$. 
\end{proof}

Although a result of Varadarajan (cf. Theorem 11.4.1 in \cite{dudley}) implies that a.s. $\nu_n \xrightarrow{w}\nu$, Proposition \ref{thm:tmaps} gives a way to specify the probability $1$ set on which the weak convergence holds.
\begin{corollary} \label{empconverge}
On the probability $1$ set $\Omega_0$ of Proposition \ref{thm:tmaps}, the empirical measures $\nu_n$ converge weakly to $\nu$ as $n \to \infty$. 
\end{corollary}

\begin{proof}
Let $f : D \to \mathbb{R}$ be a bounded, Lipschitz continuous function. Since 
$$\frac{1}{n}\sum_{i=1}^n f(X_i) = \int_D f(T_n x)\, d\nu(x),$$
 we may write
\begin{align*}
\Big|\frac{1}{n}\sum_{i=1}^n f(X_i) - \int_D f(x)\, d\nu(x)\Big| &\leq \int_D |f(T_nx) - f(x)|\, d\nu(x) \\
& \leq C \int_D |x - T_n x|\, d\nu(x) \\
& \leq C ||Id - T_n ||_{L^{\infty}},
\end{align*}
where $C$ is a Lipschitz constant for $f$. By Proposition \ref{thm:tmaps}, for each realization of $\{X_i\}_{i\in \mathbb{N}}$ in $\Omega_0$, we have $\frac{1}{n}\sum_{i=1}^n f(X_i) \to \int_D f(x)\, d\nu(x)$ as $n\rightarrow\infty$. Hence, by the Portmanteau theorem (Theorem 3.9.1 in \cite{durrett}), we have the weak convergence $\nu_n \xrightarrow{w} \nu$ as $n\rightarrow\infty$. 
\end{proof}


\subsection{On Gamma convergence of random functionals}
\label{gamma_conv}

Here, we introduce a type of $\Gamma$-convergence, with respect to random functionals, which will be an important tool in the proof of Theorem \ref{thm:mainthm} in Section \ref{proof_mainthm}, and may be of interest in its own right.  
 For what follows, let $X$ denote a metric space with metric $d$ and let $F_n : X \to [0,\infty]$ be functionals on this space. 

We first state the definition with respect to deterministic functionals.
\begin{definition}
\label{defn_det}
The sequence $\{F_n\}_{n \in \mathbb{N}}$ $\Gamma$-converges with respect to the topology on $X$ if the following conditions hold:
\begin{enumerate}
\item \textbf{Liminf inequality}: For every $x \in X$ and every sequence $\{x_n\}_{n \in \mathbb{N}}$ converging to $x$,
\begin{align*}
F(x) \leq \liminf_{n \to \infty} F_n(x_n),
\end{align*}
\item \textbf{Limsup inequality}: For every $x \in X$, there exists a sequence $\{x_n\}_{n\in \mathbb{N}}$ converging to $x$ satisfying 
\begin{align*}
\limsup_{n \to \infty} F_n(x_n) \leq F(x).
\end{align*}
\end{enumerate}
The function $F$ is called the $\Gamma$-limit of $\{F_n\}_{n\in \mathbb{N}}$, and we write $F_n \xrightarrow{\Gamma} F$. 
\end{definition}
When we wish to make the dependence on the metric $d$ explicit, we say
that $\{F_n\}_{n \in \mathbb{N}}$ $\Gamma(d)$-converges to $F$, or $F$ is
the $\Gamma(d)$-limit of $\{F_n\}_{n \in \mathbb{N}}$, etc.

\begin{remark} \rm
If the liminf inequality holds, the limsup inequality is equivalent to the following condition:  For
every $x \in X$, there exists a sequence $\{x_n\}_{n
  \in \mathbb{N}}$ with $x_n \to x$ and  $\lim_{n \to \infty} F_n(x_n)
= F(x)$. The sequence $\{x_n\}_{n \in \mathbb{N}}$ is referred to as a
\textit{recovery sequence} for $x$.
\end{remark}

A basic consequence of Definition \ref{defn_det} is the following (cf. \cite{braides2002gamma}, Theorem 1.21).
\begin{theorem}
Let $F_n : X \to [0,\infty]$ be a sequence of functionals
$\Gamma$-converging to $F$.  Suppose $\{x_n\}_{n \in \mathbb{N}}$ is a
relatively compact sequence in $X$ with 
\begin{equation} \label{eq:approxminimizers}
\lim_{n \to \infty} \Big(F_n(x_n) - \inf_{x \in X} F_n(x)\Big) = 0.
\end{equation}
Then,
\begin{enumerate}
\item $F$ attains its minimum value and
\begin{align*}
\min_{x\in X} F(x) = \lim_{n \to \infty} \inf_{x \in X} F_n(x).
\end{align*}
\item The sequence $\{x_n\}_{n \in \mathbb{N}}$ has a cluster point,
  and every cluster point of the sequence
  is a minimizer of $F$.
\end{enumerate}
\end{theorem}

For this theorem to be applicable, it is standard to put some
condition on $\{F_n\}_{n \in \mathbb{N}}$ so that
(\ref{eq:approxminimizers}) implies that the sequence $\{x_n\}_{n \in
  \mathbb{N}}$ is relatively compact in $X$. 

\begin{definition} \label{def:compactness}
We say that the sequence of nonnegative functionals $\{F_n\}_{n \in \mathbb{N}}$ has the compactness property if for any sequence $\{x_n\}_{n \in \mathbb{N}}$, the following two conditions,
\begin{enumerate}
\item[(i)] $\{x_n\}_{n \in \mathbb{N}}$ is bounded in $X$
\item[(ii)] The energies $\{F_n(x_n)\}_{n \in \mathbb{N}}$ are bounded,
\end{enumerate}
imply that $\{x_n\}_{n \in \mathbb{N}}$ is relatively compact in $X$.
\end{definition}

We now extend
the above notions to the random setting.  Here, we have a
probability space $(\Omega, \mathcal{F}, \mathbb{P})$ and a sequence
of functionals $F_n : X \times \Omega \to [0,\infty]$. 

\begin{definition}  \label{def:nondeterministicgamma}
  We say the (random) sequence $\{F_n\}_{n \in \mathbb{N}}$ $\Gamma$-converges to the
deterministic functional $F : X \to [0,\infty]$ if 
\begin{enumerate}
\item \textbf{Liminf inequality} With probability $1$, the following
  statement holds:  For any $x \in X$ and any sequence $\{x_n\}_{n \in \mathbb{N}}$ with $x_n \to x$,
\begin{equation*}
F(x) \leq \liminf_{n \to \infty} F_n(x_n).
\end{equation*}
\item \textbf{Recovery sequence} For any $x \in X$,  there exists a
  (random) sequence $\{x_n\}_{n \in \mathbb{N}}$ with
  $x_n \xrightarrow{a.s.} x$ and $F_n(x_n) \xrightarrow{a.s} F(x)$.
\end{enumerate}
\end{definition}

\begin{definition}
\label{def:random_compactness}
We say the (random) sequence $\{F_n\}_{n \in \mathbb{N}}$ has the compactness property if with probability $1$, the sequence $\{F_n(\cdot, \omega)\}_{n\in \mathbb{N}}$ has the compactness property in Definition \ref{def:compactness}.
\end{definition} 

\begin{remark}
\label{rmk:compactness}\rm
The definition for $\Gamma$-convergence of random functionals,
Definition \ref{def:nondeterministicgamma}, is weaker than the one in
\cite{trillos2014continuum}, which prescribes that Definition \ref{defn_det} holds with probability $1$.  However, in our Definition \ref{def:nondeterministicgamma}, with respect to the recovery sequence, the probability $1$ set may depend on the sequence, and therefore is an easier condition to verify, say with probabilistic arguments.  Interestingly, this weaker definition has the same strength in terms of its application in the following Gamma convergence statement, Theorem \ref{thm:nondeterministicgamma}, a main vehicle in the proof of Theorem \ref{thm:mainthm}.

In passing, we also note that the compactness criterion of random functionals, Definition \ref{def:random_compactness}, can also be weakened, in that the probability $1$ set may depend on the particular bounded sequence in Definition \ref{def:compactness}, without altering the statement of the Gamma convergence Theorem \ref{thm:nondeterministicgamma}, and with virtually the same proof.

\end{remark}
 
\begin{theorem} \label{thm:nondeterministicgamma}
Let $F_n : X \times \Omega \to [0,\infty]$ be a sequence of
random functionals $\Gamma$-converging to a limit $F: X \to
[0,\infty]$, in the sense of Definition \ref{def:nondeterministicgamma}, which is not identically equal to $\infty$. Suppose that
$\{F_n\}_{n \in \mathbb{N}}$ has the compactness property, in the sense of Definition \ref{def:random_compactness}, and also the
following condition holds:  For $\omega$ in a probability $1$ set,
there exists a bounded sequence, $x_n= x_n(\omega)$, whose bound may depend on $\omega$, such that
\begin{align*}
\lim_{n \to \infty} \Big(F_n(x_n) - \inf_{x \in X} F_n(x)\Big) = 0.
\end{align*}
Then, with probability $1$,
\begin{enumerate}
\item $F$ attains its minimum value and
\begin{align*}
\min_{x\in X} F(x) = \lim_{n \to \infty} \inf_{x \in X} F_n(x).
\end{align*}
\item The sequence $\{x_n\}_{n \in \mathbb{N}}$ has a cluster point,
  and every cluster point of the sequence
  is a minimizer of $F$.
\end{enumerate}
\end{theorem}

\begin{proof}
Pick $\tilde{x} \in X$, along with a recovery sequence
$\{\tilde{x}_n\}_{k \in \mathbb{N}}$, so that on a probability $1$ set $\Omega_1$ we have $\lim_{n\rightarrow\infty}\tilde{x}_n = \tilde{x}$ and $\lim_{n\rightarrow\infty}
F_n(\tilde{x}_n) = F(\tilde{x})$. 
Let $\Omega_2$ be a probability $1$ set on which $x_n=x_n(\omega)$ is a bounded sequence where $\lim_{n \to \infty} (F_n(x_n) -
\inf_{x \in X} F_n(x)) = 0$.   Hence, on the probability $1$ set $\Omega_1\cap \Omega_2$, we obtain
\begin{equation} \label{eq:limsupas1}
\limsup_{n \to \infty} F_{n}(x_{n}) \leq F(\tilde{x}).
\end{equation}

Applying the argument for (\ref{eq:limsupas1}) with respect to a countable collection
$\{\tilde{x}^{(m)}\}_{m \in \mathbb{N}}$ with $\lim_{m \to \infty}
F(\tilde{x}^{(m)}) = \inf_{x \in
  X} F(x) $, we obtain on a probability $1$ set $\Omega_3\subset \Omega_2$ that
\begin{equation} \label{eq:limbound1}
\limsup_{n \to \infty} F_{n}(x_{n}) \leq \inf_{x
  \in X} F(x). 
\end{equation}

Now, because $F$ is not identically equal to $\infty$, the right hand
side of the above inequality is finite.  Then, on the probability $1$
set $\Omega_3$, the sequences $\{x_n\}_{n\in \mathbb{N}}$
and $\{F_n(x_n)\}_{n\in\mathbb{N}}$ are bounded.  Let $\Omega_4$ be
the probability $1$ set on which the compactness property for
$\{F_n\}_{n\in \mathbb{N}}$ holds.  In particular, on $\Omega_5 =
\Omega_3\cap\Omega_4$, the bounded sequence $\{x_n\}_{n\in
  \mathbb{N}}$ is relatively compact.  With respect to the set $\Omega_5$, let $\{x_{n_k}\}_{k\in \mathbb{N}}$ be a subsequence converging to a cluster point $x^*$, that is  $\lim_{k\rightarrow\infty}x_{n_k}=x^*$. 
  
Let $\Omega_6$ be a probability $1$ set on which the liminf
inequality holds.  Then, on $\Omega_7=\Omega_5\cap\Omega_6$, we have
\begin{equation} \label{eq:limbound2}
\inf_{x \in X} F(x) \leq F(x^*) \leq \liminf_{k \to \infty}
F_{n_k}(x_{n_k}).
\end{equation}

Combining (\ref{eq:limbound1}) and (\ref{eq:limbound2}) shows, since $\Omega_7\subset \Omega_3$, that on the set $\Omega_7$ we have
\begin{equation}
\label{eq:limbound3}
\limsup_{k\rightarrow\infty}F_{n_k}(x_{n_k}) \leq \limsup_{n\rightarrow\infty}F_n(x_n) \leq \inf_{x \in X} F(x) \leq F(x^*) \leq \liminf_{k\rightarrow\infty} F_{n_k}(x_{n_k}).
\end{equation}
Hence, we conclude that $F$ attains its minimum value, $F(x^*) = \inf_{x\in X}F(x)$ and $x^*$ is a minimizer of $F$, proving part of the first statement.  In fact, the second statement also follows:  With respect to the probability $1$ set $\Omega_7$, every cluster point of $\{x_n\}_{n\in \mathbb{N}}$ is a minimizer of $F$.

We now show the remaining part of the first statement.  With respect to the set $\Omega_7$, let $\{x_{m_k}\}_{k\in \mathbb{N}}$ be a subsequence of $\{x_n\}_{n\in \mathbb{N}}$ where, for some cluster point $x^{**}$,
$$\lim_{k\rightarrow\infty}F_{m_k}(x_{m_k}) = \liminf_{n\rightarrow\infty}F_n(x_n) {\rm \ \ and \ \ } \lim_{k\rightarrow\infty}x_{m_k} = x^{**}.$$
  Then, by \eqref{eq:limbound3}, we conclude on $\Omega_7$ that $\lim_{n\rightarrow\infty}F_n(x_n) = \inf_{x\in X}F(x)$.  Since $\Omega_7\subset \Omega_2$, we have on $\Omega_7$ that $\lim_{n\rightarrow\infty}(F_n(x_n) - \inf_{x\in X}F_n(x)) = 0$.  Hence, we conclude on $\Omega_7$ that $\lim_{n\rightarrow\infty}\inf_{x\in X}F_n(x_n) = \inf_{x\in X}F(x)$.        
\end{proof}

\section{Reformulation of the modularity functional}
\label{reformulation_sect}
In this section, we write the modularity functional as a sum of a `graph total variation' term and a `balance' or `quadratic' term, which will aid in its subsequent analysis.  A similar, but different decomposition was used in \cite{hu_et_al}.

Recall the modularity functional in Subsection \ref{results} acting on a partition $\mathcal{U}_n = \{U_{n,k}\}_{k=1}^K$ of the data points $\mathcal{X}_n$ into $K\geq 1$ sets:
\begin{equation} \label{eq:modagain}
Q_n(\mathcal{U}_n) = \frac{1}{2m}\sum_{i,j} \Big(W_{ij} -
2m \frac{d_i^{\alpha} d_j^{\alpha}}{S^2}\Big) \delta(c_i,c_j).
\end{equation}
Here, $d_i = \sum_j W_{ij}$, $2m = \sum_{ij}
W_{ij}$, and $S = \sum_i (\sum_{j} W_{ij})^{\alpha}$ and the weights
$W_{ij} = \eta_{\epsilon_n}(X_i-X_j)$ if $i\neq j$ and equal $0$
otherwise.  The label $c_i = k$ is assigned to the point $X_i$ if $X_i
\in U_{n,k}$ for $1\leq k\leq K$.

Define $I_n(D)$ as the collection of indicator functions of subsets of $\mathcal{X}_n$.  Natural members of $I_n(D)$, in the above context, are $u_{n,k} = \mathds{1}_{U_{n,k}}$ for $1\leq k\leq K$.  Note that the collection $\{u_{n,k}\}_{k=1}^K$ satisfies $\sum_{k=1}^K u_{n,k} = \mathds{1}_{\mathcal{X}_n}$.

Observe now that $\delta(c_i,c_j)$, signifying that $X_i$ and $X_j$ have the same label, can be expressed in two ways:
\begin{eqnarray}
\delta(c_i,c_j) = 1 -\frac{1}{2}\sum_{k=1}^K |u_{n,k}(X_i)
- u_{n,k}(X_j)| = \sum_{k=1}^K u_{n,k}(X_i)u_{n,k}(X_j).
\label{delta_twoways}
\end{eqnarray}

Applying the first identity in \eqref{delta_twoways} to the first term in (\ref{eq:modagain}) gives
\begin{equation*} 
\begin{aligned}
\frac{1}{2m}\sum_{i,j} W_{ij} \delta(c_i,c_j) &=  \frac{1}{2m}\sum_{i,j} W_{ij} - \frac{1}{2}\frac{1}{2m} \sum_{k=1}^K \sum_{i,j}
W_{ij} |u_{n,k}(X_i) - u_{n,k}(X_j)| \\
&= 1 - \frac{1}{2}\frac{1}{2m} \sum_{k=1}^K \sum_{i,j}
W_{ij} |u_{n,k}(X_i) - u_{n,k}(X_j)|.
\end{aligned}
\end{equation*}

Define the \textit{graph total variation} $GTV_n(u)$, acting on $u:\mathcal{X}_n\rightarrow \mathbb{R}$, to be
\begin{equation}\label{eq:gtvepsdef}
GTV_{n}(u) \coloneqq \frac{1}{\epsilon_n} \frac{1}{n(n-1)} \sum_{1 \leq i\neq j \leq n}
\eta_{\epsilon_n}(X_i - X_j) |u(X_i) - u(X_j)|.
\end{equation}
Then, we may write
\begin{equation}\label{eq:multitv}
\frac{1}{2m}\sum_{i,j} W_{ij} \delta(c_i,c_j) =  1 - \epsilon_n \frac{n(n-1)}{4m}\sum_{k=1}^{K} GTV_{n}(u_{n,k}).
\end{equation}

Similarly, the second relation in \eqref{delta_twoways}
gives
\begin{equation*}
\begin{aligned}
\sum_{i,j} \frac{d_i^{\alpha}d_j^{\alpha}}{S^2} \delta(c_i,c_j) &=
\frac{1}{S^2} \sum_{k=1}^K \sum_{i,j} d_i^{\alpha}d_j^{\alpha} u_{n,k}(X_i)u_{n,k}(X_j) \\
&= \frac{1}{S^2} \sum_{k=1}^K \Big( \sum_{i} d_i^{\alpha}
u_{n,k}(X_i)\Big)^2 \\
&= \frac{1}{S^2} \sum_{k=1}^K \Big( \sum_{i} 
\Big(\sum_{\stackrel{1 \leq j \leq n}{j\neq i}} \eta_{\epsilon_n}(X_i -
X_j)\Big)^{\alpha} u_{n,k}(X_i)\Big)^2.
\end{aligned}
\end{equation*}
Define $G\Lambda_n(u)$, for $u:D\rightarrow \mathbb{R}$, by
\begin{equation}\label{eq:glambdadef}
G\Lambda_{n}(u) \coloneqq  \frac{1}{n} \sum_{i=1}^n \Big(
\frac{1}{n-1} \sum_{\stackrel{1 \leq  j \leq n}{j\neq i}}
  \eta_{\epsilon_n}(X_i - X_j)\Big )^{\alpha} u(X_i).
\end{equation}
Then,
\begin{equation*}
\sum_{i,j} \frac{d_i^{\alpha}d_j^{\alpha}}{S^2} \delta(c_i,c_j) =
\frac{n^2(n-1)^{2\alpha}}{S^2} \sum_{k=1}^K \big( G\Lambda_n (u_{n,k})\big)^2.
\end{equation*}
Note that $G\Lambda_n (1) = S / (n(n-1)^{\alpha})$. With a bit of algebra, we obtain
\begin{eqnarray*}
&&
\sum_{k=1}^K \Big( G\Lambda_n (u_{n,k})\Big)^2 \\
&& = \sum_{k=1}^K
\Big[ (G\Lambda_n (u_{n,k} - 1/K))^2 \\
&&\ \ \ \ \ \ \ \ \ \ + 2 G\Lambda_n(
u_{n,k} - 1/K)G\Lambda_n (1/K) + (G\Lambda_n(1/K))^2\Big],
\end{eqnarray*}
which further equals
\begin{eqnarray*}
&& \sum_{k=1}^K (G\Lambda_n (u_{n,k} - 1/K))^2 \\
&&\ \ \ \ \ \ \ \ \ \ + 2G\Lambda_n
\Big(\sum_{k=1}^K (u_{n,k} - 1/K)\Big) G\Lambda_n(1/K)+ \frac{1}{K}G\Lambda_n(1)^2 \\
&&  = \sum_{k=1}^K (G\Lambda_n (u_{n,k} - 1/K))^2 +
\frac{1}{K}G\Lambda_n(1)^2.
\end{eqnarray*}
We have used the relation $\sum_{k=1}^K u_{n,k}  = \mathds{1}_{\mathcal{X}_n}$ in the
last equality.

Hence,
\begin{equation}  \label{eq:multiquad}
\sum_{i,j} \frac{d_i^{\alpha} d_j^{\alpha}}{S^2}\delta(c_i,c_j) =  \frac{n^2(n-1)^{2\alpha}}{S^2}\Big[\sum_{k=1}^K (G\Lambda_n (u_{n,k} - 1/K))^2 \Big] +
1/K.
\end{equation}

Combining (\ref{eq:multitv}) and (\ref{eq:multiquad}) gives
\begin{eqnarray} \label{eq:oneminusmod}
1 - 1/K - Q_n(\mathcal{U}_n) &=& \frac{n^2(n-1)^{2\alpha}}{S^2}\Big[\sum_{k=1}^K (G\Lambda_n (u_{n,k} -
1/K))^2\Big] \\
&&\ \ \ \ + \epsilon_n \frac{n(n-1)}{4m} \sum_{k=1}^K
GTV_{n}(u_{n,k}). \nonumber
\end{eqnarray}

\section{Proof of Theorem \ref{thm:asymptotics}: Asymptotic Formula}
\label{proof_asymptotics}
We analyze the `graph total variation' and `quadratic' terms, identified in the decomposition of the modularity functional in Section \ref{reformulation_sect}, in the first two subsections.  Then, in Subsection \ref{asymptotic_formulas_sect}, we collect estimates and prove Theorem \ref{thm:asymptotics}.

In this section, in accordance with the assumptions of Theorem
\ref{thm:asymptotics}, we suppose that the partition
$\mathcal{U}_n=\{U_{n,k}\}_{k=1}^K$ of the data points $\mathcal{X}_n$
is induced by a `continuum' partition $\mathcal{U}= \{U_k\}_{k=1}^K$
of $D$ into $K\geq 1$ sets with finite perimeter ${\rm
  Per}(U_k;\rho^2)<\infty$, where $U_{n,k} =
\{X_i \in \mathcal{X}_n | X_i\in U_k\}$, for $1\leq k\leq K$.

Define $I(D)$ as the collection of measurable indicator functions of subsets $U\subset D$.  Let $u_k = \mathds{1}_{U_k}$, and note that $u_k\in I(D)$ is an extension of the indicator $u_{n,k} = \mathds{1}_{U_{n,k}}$, defined on $\mathcal{X}_n$, for $1\leq k\leq K$.    Of course, the family $\{u_k\}_{k=1}^K$ satisfies $\sum_{k=1}^K u_k = \mathds{1}_D$.

\subsection{Convergence of graph total variation}  To show a.s. convergence of the graph total variations, we first state that its expectations converge, and then use concentration ideas to elicit convergence of the random quantitites.

Let $u \in L^1(D)$. We define the \textit{nonlocal total
  variation} of $u$ to be
\begin{equation*}\label{eq:tvepsdef}
TV_{\epsilon}(u;\rho) \coloneqq \frac{1}{\epsilon}\int_D\int_D \eta_{\epsilon}(x-y)
|u(x)-u(y)|\rho(x)\rho(y) \,dx\,dy.
\end{equation*}
Note that, if $X$ and $Y$ are independent random variables with
density $\rho$,
we have
\begin{equation*}
\mathbb{E}\Big[ \frac{1}{\epsilon} \eta_{\epsilon}(X-Y)|u(X) -
u(Y)|\Big] = TV_{\epsilon}(u;\rho).
\end{equation*}
Recalling the definition (\ref{eq:gtvepsdef}) of the graph total
variation, we therefore have
\begin{equation*}
\mathbb{E}\Big[GTV_{n}(u)\Big] =  TV_{\epsilon_n}(u;\rho).
\end{equation*}

Let also $$\sigma_\eta := \int_D \eta(x)|x_1|dx.$$

\begin{lemma} \label{lemma:tvlim}
Let $u \in L^1(D)$ such that $TV(u; \rho^2)<\infty$. Then, we have
\begin{equation} \label{eq:tvlimnonlocal}
 \lim_{\epsilon \to 0} TV_{\epsilon}(u;\rho) = \sigma_{\eta}
TV(u;\rho^2).
\end{equation}
\end{lemma}

\begin{proof}

For general $\rho$, continuous on $D$ and bounded above and below by
positive constants, and $d\geq 2$, the result follows from part of the proof of
\cite{trillos2014continuum}, Theorem 4.1 (see Remark 4.3 in
\cite{trillos2014continuum}), which is a much more involved result.  This proof also holds in $d=1$. 

For the convenience of the reader, we make a few remarks (see also remarks in \cite{trillos2014continuum}),
addressing the case $\rho = \mathds{1}_D$.
Integrals of the form
\begin{equation}
\int_{D \times D} \eta_{\epsilon}(x-y)
\frac{|u(x) - u(y)|}{|x-y|}\, dx \, dy
\end{equation} are considered in \cite{bourgain2001another}, where the limiting behavior as $\epsilon
\to 0$ is established for $u \in W^{1,1}(D)$. In \cite{davila2002open}
and \cite{ponce2004new},
this is extended to general $u$ such that $TV(u;\mathds{1}_D)<\infty$, with the result being
\begin{equation} \label{eq:nonlocalconvergenceponce}
\lim_{\epsilon \to 0} \int_{D \times D} \eta_{\epsilon}(x-y)
\frac{|u(x) - u(y)|}{|x-y|}\, dx \, dy = K_{1,d} TV(u;1),
\end{equation}
where $K_{1,d} = \frac{1}{\mathcal{H}^{d-1}(S^{d-1})} \int_{S^{d-1}} |e_1 \cdot
  \sigma|\, d\mathcal{H}^{d-1}$ is the average of the $e_1$ component
  of the unit normal vector field on the unit sphere $S^{d-1}$ (see
  \cite{ponce2004new}, Corollary 1.3). 

The limit
(\ref{eq:nonlocalconvergenceponce})  implies
(\ref{eq:tvlimnonlocal}). Indeed, 
\begin{equation}
\begin{aligned}
TV_{\epsilon}(u;\mathds{1}_D) &= \frac{1}{\epsilon}\int_{D \times D}
\eta_{\epsilon}(x-y)|u(x)-u(y)|\, dx \, dy \\
&= \int_{D \times D}
\tilde{\eta}_{\epsilon}(x-y) \frac{|u(x)-u(y)|}{|x-y|}\, dx \, dy,
\end{aligned}
\end{equation}
with $\tilde{\eta}(z) = \eta(z)|z|$ and $\tilde{\eta}_{\epsilon}(z) =
\tilde{\eta}(z/\epsilon) / \epsilon^d$. Letting $c =
\int_{\mathbb{R}^d} \eta(z)|z|\, dz$, it follows from
(\ref{eq:nonlocalconvergenceponce}) that
\begin{equation}
\begin{aligned}
\lim_{\epsilon \to 0} TV_{\epsilon}(u;\mathds{1}_D) &= \lim_{\epsilon \to 0} \int_{D \times D}
\tilde{\eta}_{\epsilon}(x-y) \frac{|u(x)-u(y)|}{|x-y|}\, dx \, dy \\
&= c K_{1,d} TV(u;\mathds{1}_D).
\end{aligned}
\end{equation}
Finally, one may verify that $c K_{1,d} = \int_{\mathbb{R}^d}
\eta(z)|z_1|\, dz=\sigma_{\eta}$, as desired.
\end{proof}

We now proceed to the almost sure convergence of the graph total
variation to its continuum limit.  

\begin{lemma} \label{lemma:gtvasconverge}
Fix $u \in I(D)$ where $TV(u;\rho^2)<\infty$, and let $\{\epsilon_n\}_{n \in \mathbb{N}}$ be a
sequence converging to zero such that 
\begin{equation} \label{eq:gtvlemmaeps}
\sum_{n=1}^{\infty} \exp(-n\epsilon_n^{(d+1)/2}) < + \infty.
\end{equation}
Then, 
\begin{equation*} GTV_n(u) \xrightarrow{a.s.}
\sigma_{\eta}TV(u;\rho^2)
\end{equation*}
as $n \to \infty$.
\end{lemma}

\begin{proof}
Let $\displaystyle{f_n(x,y) =\frac{1}{\epsilon_n} \eta_{\epsilon_n}(x-y)|u(x) - u(y)|}$, so that we have
\begin{equation*}
GTV_{n}(u) = \frac{1}{n(n-1)} \sum_{1 \leq i \neq j \leq n} f_n(X_i,X_j).
\end{equation*}

We let $\mathbb{E}f_n$ denote $\mathbb{E}[f_n(X_i,X_j)]= TV_{\epsilon_n}(u;\rho)$, so that
\begin{equation*}
GTV_{n}(u) = \frac{1}{n(n-1)} \sum_{1 \leq i \neq j \leq n} \Big(f_n(X_i,X_j) -
\mathbb{E}f_n\Big) +
\mathbb{E}f_n.
\end{equation*}
By Lemma \ref{lemma:tvlim},  we have $\lim_{n \to \infty} \mathbb{E}f_n = \sigma_{\eta}
TV(u;\rho^2)$. Therefore, it remains to argue that
$\displaystyle{ \lim_{n \to \infty} \frac{1}{n(n-1)} \sum_{1 \leq i \neq j \leq n} \Big(f_n(X_i,X_j) -
\mathbb{E}f_n\Big) = 0}$ almost surely.

We introduce a bit of notation.  Let 
\begin{equation*}
\begin{aligned}
l_n(X_i) &= \int_D f_n(X_i,y)\rho(y)\, dy \\
m_n(X_j) &=
\int_D f_n(x,X_j)\rho(x)\, dx \\
h_n(X_i,X_j) &= f_n(X_i,X_j) - l_n(X_i) - m_n(X_j) + \mathbb{E}f_n,
\end{aligned}
\end{equation*}
so that we have
\begin{equation*} 
f_n(X_i,X_j) - \mathbb{E}f_n= h_n(X_i,X_j) + (l_n(X_i) -
\mathbb{E}f_n) +
(m_n(X_j) - \mathbb{E}f_n).
\end{equation*}
Summing this gives
\begin{eqnarray} 
&&\frac{1}{n(n-1)} \sum_{1 \leq i \neq j \leq n} \Big(f_n(X_i,X_j) -
\mathbb{E}f_n\Big) = \frac{1}{n(n-1)} \sum_{1 \leq i \neq j \leq n}
h_n(X_i,X_j) \nonumber\\
&&\ \ \ \ \ + \frac{1}{n}\sum_{1 \leq i \leq n} (l_n(X_i) -
\mathbb{E}f_n)  + \frac{1}{n} \sum_{1 \leq j \leq n} (m_n(X_j) -
\mathbb{E}f_n).
\label{eq:gtvfdecomp}
\end{eqnarray}

We handle the three terms on the right hand side of the above equation separately.

First, note that $\mathbb{E}l_n = \mathbb{E}f_n$.  
An application of Bernstein's inequality 
(\cite{van2000asymptotic}, Lemma 19.32) in this setting yields, for $s
> 0$,
\begin{equation*}
\mathbb{P}\Big(\Big| \frac{1}{n} \sum_{i} l_n(X_i) - \mathbb{E}f_n
\Big| > s\Big) \leq 2 \mbox{exp}\Big( -\frac{1}{4} \frac{ns^2}{\mathbb{E}l_n^2 + s \|l_n\|_{L^\infty}}\Big).
\end{equation*}

In Lemma \ref{lemma:gtvestimates} of the Appendix, we prove the upper
bounds $\mathbb{E}l_n^2 \leq C / \epsilon_n$ and $\|l_n\|_{L^\infty}
\leq C / \epsilon_n$ where $C$ is a constant independent of $n$. Hence,
we have
\begin{equation*}
\begin{aligned}
\mathbb{P}\Big(\Big| \frac{1}{n} \sum_{i} l_n(X_i) - \mathbb{E}f_n
\Big| > s\Big)   &\leq 2 \mbox{exp}\Big(-\frac{\epsilon_n n}{C}\frac{ s^2}{s + 1}\Big).
\end{aligned}
\end{equation*}

By our assumption (\ref{eq:gtvlemmaeps}), this implies
\begin{equation*}
\sum_{n=1}^{\infty} \mathbb{P}\Big(\Big| \frac{1}{n} \sum_{i} l_n(X_i) - \mathbb{E}f_n
\Big| > s\Big) < \infty,
\end{equation*}
and so, as $n\rightarrow\infty$, \begin{equation} \label{eq:gtvlconvergeas}
\frac{1}{n} \sum_{i=1}^n l_n(X_i) -
    \mathbb{E}f_n \xrightarrow{a.s.} 0.
\end{equation}
Similarly, we have, as $n\rightarrow\infty$,
\begin{equation} \label{eq:gtvmconvergeas}
\frac{1}{n} \sum_{j=1}^n m_n(X_j) -
    \mathbb{E}f_n \xrightarrow{a.s.} 0.
\end{equation}

What remains is the double sum $\displaystyle{\frac{1}{n(n-1)}\sum_{1
    \leq i \neq j \leq n} h_n(X_i,X_j)}$. Let $Y_1,\ldots,Y_n$ be
independent copies of $X_1,\ldots,X_n$. By the decoupling inequality of de la Pe\~{n}a and Montgomery-Smith
\cite{de1995decoupling}, there is a constant $C$ independent of
$n$ and $h$ such that
 \begin{eqnarray} 
&&\mathbb{P}\Big( \Big|\frac{1}{n(n-1)} \sum_{1\leq i\neq j\leq n} h_n(X_i,X_j) \Big| > s \Big) \\
\label{eq:decouple}
&&\ \ \ \ \leq C\, \mathbb{P}\Big( C
\Big|\frac{1}{n(n-1)} \sum_{1\leq i\neq j\leq n} h_n(X_i,Y_j) \Big| > s \Big).\nonumber
\end{eqnarray}

The sum $\sum_{1 \leq i \neq j \leq n} h_n(X_i,Y_j)$ is canonical, that is, $\mathbb{E}_X[h_n(X_i,Y_j)] = 0\text{ a.s.}$ and
$\mathbb{E}_Y[h_n(X_i,Y_j)] = 0 \text{ a.s.}$, where $\mathbb{E}_X$
and $\mathbb{E}_Y$ denote expectation with respect to the first and
second variables respectively.  A general concentration inequality for U-statistics given by Gin\'e, Lata{\l}a and Zinn in Theorem 3.3 of \cite{gine2000exponential},
states, for canonical kernels $\{ h_{i,j} \}_{1 \leq i, j \leq n}$, that
\begin{align*}
\mathbb{P}\Big( \Big| \sum_{1\leq i,j \leq n} h_{i,j}(X_i, Y_j) \Big| > s\Big) \leq L \exp\Big[-\frac{1}{L} \min\big(\frac{s^2}{R^2}, \frac{s}{Z}, \frac{s^{2/3}}{B^{2/3}}, \frac{s^{1/2}}{A^{1/2}}\big)\Big],
\end{align*}
for all $s > 0$, where $L$ is a constant not depending on $\{ h_{i,j} \}_{1 \leq i, j \leq n}$ or $n$, and 
\begin{align*}
& A = \max_{i,j} \|h_{i,j}\|_{L^{\infty}}, \hspace{0.5cm} R^2 = \sum_{i,j} \mathbb{E} h_{i,j}^2, \\
&B^2 = \max_{i,j} \Big[ \| \sum_i \mathbb{E}_X h_{i,j}^2(X_i,y) \|_{L^{\infty}}, \| \sum_j \mathbb{E}_Y h_{i,j}^2 (x,Y_j)\|_{L^{\infty}}\Big], \\
& Z = \sup \{ \mathbb{E} \sum_{i,j} h_{i,j}(X_i,Y_j) f_i(X_i) g_j(Y_j) : \mathbb{E} \sum_i f_i^2(X_i) \leq 1, \mathbb{E} \sum_j g_j^2(Y_j) \leq 1 \}.
\end{align*}

In our context, we take $h_{i,j} = h_n$ for $i \neq j$, and $h_{i,j} = 0$ otherwise, which gives the constants $A = \| h_n \|_{L^{\infty}}$, $B^2 = (n-1) \max( \|\mathbb{E}_X h_n^2\|, \|\mathbb{E}_Y h_n^2\|)$, $R^2 = n(n-1) \mathbb{E} h_n^2$ and, after a manipulation, $Z \leq n \|h_n\|_{L^2 \to L^2}$, where $$\|h_n\|_{L^2 \to L^2} \coloneqq \sup \{ \mathbb{E} h(X,Y) f(X)g(Y) : \mathbb{E} f^2(X) \leq 1, \mathbb{E} g^2(Y) \leq 1 \}.$$
It follows that
\begin{align} \label{eq:ustatexpbound}
\mathbb{P}&\Big( \Big|\frac{1}{n(n-1)} \sum_{1\leq i\neq j\leq n} h_n(X_i,Y_j) \Big| >
s \Big) \\
 & \leq L \mbox{exp}\Big[ -\frac{1}{L'} \min\Big(
\begin{aligned}[t]
&\frac{n^2 s^2}{\mathbb{E}h_n^2}, \frac{ns}{\|h_n\|_{L^2 \to L^2}},\nonumber\\
&\frac{n^{2/3} s^{2/3}}{[\max (\|\mathbb{E}_X h_n^2\|_{L^\infty},
  \|\mathbb{E}_Yh_n^2\|_{L^\infty})]^{1/3}},
\frac{ns^{1/2}}{\|h_n\|_{L^\infty}^{1/2}}
\Big)\Big]\nonumber,
\end{aligned}
\end{align}
for some constant $L'$ independent of $n$ and $h$.

In Corollary \ref{cor:gtvestimates} of the Appendix, we prove the upper
bounds

\begin{center}

$\mathbb{E}h_n^2 \leq C / \epsilon_n^{d+1}$, \hspace{0.5cm}
$\|\mathbb{E}_Y h_n^2\|_{L^\infty} \leq C / \epsilon_n^{d+2}$,  \hspace{0.5cm}
$\|\mathbb{E}_X h_n^2\|_{L^\infty} \leq C  / \epsilon_n^{d+2}$, \\ \vspace{0.2cm}
$\|h_n\|_{L^\infty} \leq C / \epsilon_n^{d+1}$, \hspace{0.5cm}
$\|h_n\|_{L^2 \to L^2} \leq C/\epsilon_n$. 
\end{center}
Hence, the minimum in the right hand side of (\ref{eq:ustatexpbound})
simplifies to
\begin{equation*}
C\min(n^2 \epsilon_n^{d+1} s^2, n \epsilon_n s, n \epsilon_n^{(d+2)/3}
s^{2/3}, n \epsilon_n^{(d+1)/2} s^{1/2}).
\end{equation*}

We claim, for sufficiently large $n$, this minimum will be
attained by $n \epsilon_n^{(d+1)/2}s^{1/2}$:  Indeed,
by \eqref{eq:gtvlemmaeps}, $n\epsilon_n^{(d+1)/2}\rightarrow \infty$, and so $n\epsilon_n^{(d+1)/2}$ is smaller than $n^2\epsilon_n^{d+1}$.  Also, $n\epsilon_n$ is larger than $n\epsilon_n^{(d+1)/2}$ since $\epsilon_n\rightarrow 0$ and $d\geq 1$.  In addition, $n\epsilon_n^{(d+2)/3}$ is larger than $n\epsilon_n^{(d+1)/2}$ as $d\geq 1$.  

  Hence, we have 
\begin{equation} \label{eq:htailbound}
\mathbb{P}\Big( \Big|\frac{1}{n(n-1)} \sum_{1\leq i\neq j\leq n} h_n(X_i,Y_j) \Big| >
s \Big) \leq L \exp \Big( - C n\epsilon_n^{(d+1)/2} s^{1/2}\Big),
\end{equation}
where $C$ is a constant independent of $n$.
Combining (\ref{eq:decouple}) and (\ref{eq:htailbound}), gives
\begin{equation*}
\sum_{n=1}^{\infty} \mathbb{P}\Big( \Big| \frac{1}{n(n-1)} \sum_{1 \leq i
  \neq j \leq n} h_n(X_i,X_j) \Big| > s\Big) < \infty,
\end{equation*}
for all $s$. Therefore, 
\begin{equation} \label{eq:gtvhconvergeas}
\frac{1}{n(n-1)} \sum_{1 \leq i
  \neq j \leq n} h_n(X_i,X_j) \xrightarrow{a.s.} 0,
\end{equation}
as $n \to \infty$.
Applying (\ref{eq:gtvlconvergeas}), (\ref{eq:gtvmconvergeas}), and
(\ref{eq:gtvhconvergeas}) to (\ref{eq:gtvfdecomp}) completes the
proof.
\end{proof}

\subsection{Quadratic Term} We first consider convergence of certain `mean-values', and then treat the random expressions, for various values of $\alpha$, in the subsequent subsubsections. 

For $u \in L^1(D)$ and $\epsilon>0$, define $\Lambda_{\epsilon}(u)$ by  
\begin{equation*} 
\Lambda_{\epsilon}(u) \coloneqq
\int_D\Big(\int_D \eta_{\epsilon}(x-y)\rho(y)\,dy\Big)^{\alpha} u(x)\rho(x)\,dx.
\end{equation*}
Let 
\begin{equation}
\label{rho_defn}
\rho_{\epsilon}(x) := \int_D \eta_{\epsilon}(x-y) \rho(y)\,dy,
\end{equation}
and write, with this notation,
\begin{equation}
\label{new_lambda_def}
\Lambda_{\epsilon}(u)
= \int_D u(x)(\rho_{\epsilon}(x))^{\alpha}\rho(x) \, dx.
\end{equation}
Define also
\begin{equation}
\label{eq:lambdadef}
\Lambda(u) \coloneqq \int_D u(x) \rho^{1+\alpha}(x)dx.
\end{equation}

\begin{lemma}  \label{lemma:lambdaepslim}
Let $g$ be a bounded, measurable function on the domain $D$. Then,
there exists a constant $C$, independent of $g$, such that
\begin{equation} \label{eq:lambdaepsbound}
|\Lambda_{\epsilon}(g) - \Lambda(g)| \leq C\|g\|_{L^{\infty}(D)} \epsilon,
\end{equation}
for all sufficiently small $\epsilon$.
Further, suppose there is a sequence $\{g_{\epsilon}\}_{\epsilon >
  0}$ with $g_{\epsilon} \xrightarrow{L^1} g$ as $\epsilon \rightarrow
0$. Then, we have
\begin{equation} \label{eq:lambdaepslim}
\lim_{\epsilon \rightarrow 0} \Lambda_{\epsilon}(g_{\epsilon}) = \Lambda(g).
\end{equation}
\end{lemma}

\begin{proof}
We first prove inequality (\ref{eq:lambdaepsbound}).
By Lemma \ref{lemma:mollified} in the Appendix, there exist positive constants $A,B$ such that,
for sufficiently small $\epsilon$, both $\rho$ and $\rho_{\epsilon}$
take values in the interval $[A,B]$. Then we have
\begin{equation*}
\begin{aligned}
| \Lambda_{\epsilon}(g) - \Lambda(g) | &= \Big|\int_D g(x)
(\rho_{\epsilon}(x))^{\alpha} \rho(x)\, dx - \int_D g(x)
(\rho(x))^{\alpha} \rho(x)\, dx \Big| \\
&\leq B \|g\|_{L^\infty} \int_D |(\rho_{\epsilon}(x))^{\alpha} -
(\rho(x))^{\alpha}|\, dx \\
&\leq C \|g\|_{L^\infty} \int_D |\rho_{\epsilon}(x) - \rho(x)|\, dx,
\end{aligned}
\end{equation*} 
where the last inequality follows from the observation that $x \mapsto
x^{\alpha}$ is Lipschitz on the interval $[A,B]$.
By Lemma \ref{lemma:mollified} again, where 
$\int_D
|\rho_{\epsilon}(x) - \rho(x)|\, dx \leq C'\epsilon$ is proved, we obtain
\begin{equation*}
|\Lambda_{\epsilon}(g) - \Lambda(g)| \leq C'' \|g\|_{L^\infty} \epsilon.
\end{equation*}

We now prove (\ref{eq:lambdaepslim}). Suppose we have a family
$\{g_{\epsilon}\}_{\epsilon > 0}$ with $g_{\epsilon}
\xrightarrow{L^1} g$ as $\epsilon\rightarrow 0$. Then,  
\begin{align*}
\lim_{\epsilon \rightarrow 0} |\Lambda_{\epsilon}(g_{\epsilon}) -
\Lambda(g)| & \leq \lim_{\epsilon \rightarrow 0}  \int_D
|g_{\epsilon}(x)(\rho_{\epsilon}(x))^\alpha - g(x)(\rho(x))^\alpha|\rho(x)\,dx.
\end{align*}
Since $\rho$ is bounded, it is sufficient to prove that
\begin{align*}
\lim_{\epsilon \rightarrow 0} \int_D |g_{\epsilon}(x)(\rho_{\epsilon}(x))^\alpha -
g(x)(\rho(x))^\alpha|\, dx = 0.
\end{align*}
Writing \[ g_{\epsilon} \rho^\alpha_{\epsilon} - g \rho^\alpha =
g_{\epsilon}\rho^\alpha_{\epsilon} - g \rho^\alpha_{\epsilon} + g \rho^\alpha_{\epsilon} -
g \rho^\alpha,\] one may obtain
\begin{eqnarray*} \label{eq:gepsrhoepsineq}
&&\int_D |g_{\epsilon}(x)\rho^\alpha_{\epsilon}(x)  - g(x)\rho^\alpha(x) | \, dx  \\
&&\ \ \ \ \leq
\int_D |g_{\epsilon}(x) - g(x)| \rho^\alpha_{\epsilon}(x)\, dx
+ \|g\|_{L^\infty}\int_D |\rho^\alpha_{\epsilon}(x) - \rho^\alpha(x)|\,dx.
\end{eqnarray*}
Now, since $g_{\epsilon} \to g$ in $L^1$, and $\rho_{\epsilon}$ is bounded above and below by Lemma \ref{lemma:mollified}, we have that the first term on the right hand side vanishes in the
limit. Likewise, by Lemma \ref{lemma:mollified}, we have also $\rho_{\epsilon} \to \rho$ a.e. as $\epsilon
\rightarrow 0$. By dominated convergence, then, the second term on
the right hand side vanishes, completing the proof.
\end{proof}

In the following Subsubsection \ref{linear_sect}, the cases
$\alpha=0,1$ are considered.  Then, in Subsubsection
\ref{general_sect}, the general $\alpha\neq 0,1$ case is treated, where different techniques are used as the the functional is nonlinear.

\subsubsection{Quadratic Term: $\alpha = 0$ or $\alpha=1$}
\label{linear_sect}
The expression
(\ref{eq:glambdadef}) for $G\Lambda_n$ simplifies to give
\begin{equation*}
G\Lambda_n (u) = \begin{cases} \frac{1}{n}\sum_{i=1}^n u(X_i), & \text{when }
  \alpha = 0, \\
 \frac{1}{n(n-1)} \sum_{1\leq i\neq j\leq n} 
\eta_{\epsilon_n}(X_i - X_j) u(X_i), & \text{when }\alpha = 1.
\end{cases}
\end{equation*}

\begin{lemma} \label{lemma:quadasconvergea0}
Fix $\alpha = 0$. Let $u$ be a bounded, measurable function on the
domain $D$, and let $\{\epsilon_n\}_{n \in \mathbb{N}}$ be a
sequence converging to zero such that 
\begin{equation*} \label{eq:quadlemmaepsa0}
\lim_{n \to \infty} \frac{\log \log n}{n \epsilon_n} = 0.
\end{equation*}
Then,
\begin{equation*}
\frac{1}{\sqrt{\epsilon_n}} \Big(G\Lambda_n (u)  - \Lambda (u)\Big)
\xrightarrow{a.s.} 0.
\end{equation*}
as $n \to \infty$.
\end{lemma}

\begin{proof}
By the law of the iterated logarithm, and the boundedness of $u$, we have
\begin{equation*}
\limsup_{n \to \infty} \frac{n}{\sqrt{2 n \log \log n}} \Big|
G\Lambda_n(u) - \Lambda (u)\Big| \leq C, \hspace{0.5cm} \text{a.s.}
\end{equation*}
We may write
\begin{equation*}
\begin{aligned}
\frac{1}{\sqrt{\epsilon_n}} \Big(G\Lambda_n (u)  - \Lambda (u)\Big) &=
\frac{\sqrt{2 n \log \log n}}{n \sqrt{\epsilon_n}} \frac{n}{\sqrt{2 n
    \log \log n}}\Big(G\Lambda_n (u)  - \Lambda (u)\Big),
\end{aligned}
\end{equation*}
where by assumption, $\lim_{n \to \infty} \frac{\sqrt{n \log \log n}}{n
  \sqrt{\epsilon_n}} = 0$. Hence, $\frac{1}{\sqrt{\epsilon_n}}
\Big(G\Lambda_n (u)  - \Lambda (u)\Big) \xrightarrow{a.s.} 0$ as $n\rightarrow\infty$.
\end{proof}

\begin{lemma} \label{lemma:quadasconvergea1} 
Fix $\alpha = 1$. Let $u$ be a bounded, measurable function on the
domain $D$, and let $\{\epsilon_n\}_{n \in \mathbb{N}}$ be a
sequence converging to zero such that 
\begin{equation} \label{eq:quadlemmaepsa1}
\sum_{n=1}^{\infty} \exp(-n\epsilon_n^{(d+1)/2}) < + \infty.
\end{equation}
Then,
\begin{equation*}
\frac{1}{\sqrt{\epsilon_n}} \Big(G\Lambda_n (u)  - \Lambda (u)\Big)
\xrightarrow{a.s.} 0,
\end{equation*}
as $n \to \infty$.
\end{lemma}

\begin{proof}

We first rewrite
\begin{equation*}
\frac{1}{\sqrt{\epsilon_n}} \Big(G\Lambda_n (u)  - \Lambda (u)\Big) =
\frac{1}{\sqrt{\epsilon_n}} \Big(G\Lambda_n (u) - \Lambda_{\epsilon_n}
(u)\Big) + \frac{1}{\sqrt{\epsilon_n}} \Big(\Lambda_{\epsilon_n} (u) - \Lambda
(u)\Big).
\end{equation*}
Here, as $\alpha =1$, $\Lambda_{\epsilon_n} (u) = \int_D \int_D \eta_{\epsilon_n}(x-y)u(x)\rho(x)\rho(y)\,dx\,dy$.

By an application of inequality (\ref{eq:lambdaepsbound}), the second term on the right vanishes as $n \to
\infty$. Hence, we must show that $\lim_{n \to \infty} \frac{1}{\sqrt{\epsilon_n}} \Big(G\Lambda_n (u) - \Lambda_{\epsilon_n}
(u)\Big) = 0$ a.s.

Let $\displaystyle{f_n(x,y) = \frac{1}{\sqrt{\epsilon_n}}
\eta_{\epsilon_n}(x-y)u(x)}$. Note, for $i \neq j$, that $\mathbb{E}f_n = \mathbb{E}f_n(X_i,X_j) =
\frac{1}{\sqrt{\epsilon_n}}\Lambda_{\epsilon_n}(u)$. Then,
\begin{equation*}
\frac{1}{\sqrt{\epsilon_n}} \Big( G\Lambda_n (u) - \Lambda_{\epsilon_n} (u)\Big) =
\frac{1}{n(n-1)} \sum_{1\leq i\neq j\leq n} f_n(X_i,X_j) - \mathbb{E}f_n.
\end{equation*}

Let 
\begin{equation*}
\begin{aligned}
l_n(X_i) &= \int_D f_n(X_i,y)\rho(y)\, dy \\
m_n(X_j) &=
\int_D f_n(x,X_j)\rho(x)\, dx \\
h_n(X_i,X_j) &= f_n(X_i,X_j) - l_n(X_i) - m_n(X_j) + \mathbb{E}f_n,
\end{aligned}
\end{equation*}
so that we have
\begin{equation*} 
f_n(X_i,X_j) - \mathbb{E}f_n= h_n(X_i,X_j) + (l_n(X_i) -
\mathbb{E}f_n) +
(m_n(X_j) - \mathbb{E}f_n).
\end{equation*}
Summing this gives
\begin{eqnarray} 
&&\frac{1}{n(n-1)} \sum_{1\leq i\neq j\leq n} \Big(f_n(X_i,X_j) -
\mathbb{E}f_n\Big) \\
&&\ \ = \ \frac{1}{n(n-1)} \sum_{1\leq i\neq j\leq n}
h_n(X_i,X_j)\nonumber \\
&&\ \ \ \ \ \ \ \ \
 + \frac{1}{n}\sum_{i} (l_n(X_i) -
\mathbb{E}f_n) + \frac{1}{n} \sum_{j} (m_n(X_j) - \mathbb{E}f_n).\nonumber
\label{eq:fdecomp2}\end{eqnarray}
We handle the three terms on the right hand side of the above equation separately.

Note that $\mathbb{E}l_n = \mathbb{E}f_n$.  
An application of Bernstein's inequality (\cite{van2000asymptotic},
Lemma 19.32) yields, for $s
> 0$, that
\begin{equation*}
\mathbb{P}\Big(\Big| \frac{1}{n} \sum_{i} l_n(X_i) - \mathbb{E}f_n
\Big| > s\Big) \leq 2 \mbox{exp}\Big( -\frac{1}{4} \frac{ns^2}{\mathbb{E}l_n^2 + s \|l_n\|_{L^\infty}}\Big).
\end{equation*}
In Lemma \ref{lemma:glambdaestimates} of the Appendix, we prove the upper
bounds $\mathbb{E}l_n^2 \leq C / \epsilon_n$ and $\|l_n\|_{L^\infty}
\leq C / \epsilon_n^{1/2}$,
where $C$ is a constant independent of $n$. Hence,
we have
\begin{equation*}
\begin{aligned}
\mathbb{P}\Big(\Big| \frac{1}{n} \sum_{i} l_n(X_i) - \mathbb{E}f_n
\Big| > s\Big)   &\leq 2 \exp\Big(-\frac{C \epsilon_n n
  s^2}{1 + \epsilon_n^{1/2}s}\Big),
\end{aligned}
\end{equation*}
where $C$ is another constant independent of $n$. By the assumption (\ref{eq:quadlemmaepsa1}), this implies
\begin{equation*}
\sum_{n=1}^{\infty} \mathbb{P}\Big(\Big| \frac{1}{n} \sum_{i} l_n(X_i) - \mathbb{E}f_n
\Big| > s\Big) < \infty,
\end{equation*}
and so 
\begin{equation} \label{eq:glambdalconvergeas}
\frac{1}{n} \sum_{i=1}^n l_n(X_i) -
    \mathbb{E}f_n \xrightarrow{a.s.} 0
\end{equation}
as $n \to \infty$.
The same argument also implies that
\begin{equation} \label{eq:glambdamconvergeas}
\frac{1}{n} \sum_{j=1}^n m_n(X_j) -
    \mathbb{E}f_n \xrightarrow{a.s.} 0,
\end{equation}
as $n \to \infty$.

What remains is the double sum $\displaystyle{\frac{1}{n(n-1)}\sum_{1
    \leq i \neq j \leq n} h_n(X_i,X_j)}$. Let $Y_1,\ldots,Y_n$ be
independent copies of $X_1,\ldots,X_n$. By the decoupling inequality of de la Pe\~{n}a and Montgomery-Smith
(\cite{de1995decoupling}), there exists a constant $C$ independent of
$n$ and $h$ such that
 \begin{eqnarray} 
&&\mathbb{P}\Big( \Big|\frac{1}{n(n-1)} \sum_{1\leq i\neq j\leq n} h_n(X_i,X_j) \Big| > s \Big)\\
 \label{eq:decouple2}
 &&\ \ \ \ \ \ \ \ \ \ \leq C\, \mathbb{P}\Big( C
\Big|\frac{1}{n(n-1)} \sum_{1\leq i\neq j\leq n} h_n(X_i,Y_j) \Big| > s \Big).\nonumber
\end{eqnarray}

As in the proof of the GTV case (Lemma \ref{lemma:gtvasconverge}), the
$h_n$ here are canonical.  Recalling the U-statistics inequality of Gin\'e, Lata{\l}a and Zinn, Theorem 3.3 of \cite{gine2000exponential}, with the same application as in the proof of Lemma \ref{lemma:gtvasconverge}, we have
\begin{align} \label{eq:ustatexpbound2}
\mathbb{P}&\Big( \Big|\frac{1}{n(n-1)} \sum_{1\leq i\neq j\leq n} h_n(X_i,Y_j) \Big| >
s \Big) \\
 & \leq L \mbox{exp}\Big[ -\frac{1}{L'} \min\Big(
\begin{aligned}[t]
&\frac{n^2 s^2}{\mathbb{E}h_n^2}, \frac{ns}{\|h_n\|_{L^2 \to L^2}},\nonumber\\
&\frac{n^{2/3} s^{2/3}}{[\max (\|\mathbb{E}_X h_n^2\|_{L^\infty},
  \|\mathbb{E}_Yh_n^2\|_{L^\infty})]^{1/3}},
\frac{ns^{1/2}}{\|h_n\|_{L^\infty}^{1/2}}
\Big)\Big]\nonumber.
\end{aligned}
\end{align}

In Corollary \ref{cor:glambdaestimates} of the Appendix, we prove the upper
bounds

\begin{center}

$\mathbb{E}h_n^2 \leq C / \epsilon_n^{d+1}$,  \hspace{0.5cm}
$\|\mathbb{E}_Y h^2\|_{L^\infty} \leq C / \epsilon_n^{d+1}$, \hspace{0.5cm}
$\|\mathbb{E}_X h^2\|_{L^\infty} \leq C  / \epsilon_n^{d+1}$, \\ \vspace{0.2cm}
$\|h_n\|_{L^\infty} \leq C / \epsilon_n^{d+1/2}$, \hspace{0.5cm}
$\|h_n\|_{L^2 \to L^2} \leq C/\epsilon_n^{1/2}$. 
\end{center}
Hence, the minimum in the right hand side of (\ref{eq:ustatexpbound2})
is bounded below by
\begin{equation*}
C\min(n^2 \epsilon_n^{d+1} s^2, n \epsilon_n^{1/2} s, n \epsilon_n^{(d+1)/3}
s^{2/3}, n \epsilon_n^{(d+1/2)/2} s^{1/2}).
\end{equation*}
Note that $\epsilon_n$ vanishes, and our assumption (\ref{eq:quadlemmaepsa1}) implies that $n
\epsilon_n^{(d+1)/2} \to \infty$ as $n\rightarrow\infty$.  Therefore, we conclude, for sufficiently large $n$, that 
\begin{equation*}
C\min(n^2 \epsilon_n^{d+1} s^2, n \epsilon_n^{1/2} s, n \epsilon_n^{(d+1)/3}
s^{2/3}, n \epsilon_n^{(d+1/2)/2} s^{1/2}) \geq C n \epsilon_n^{(d+1)/2} \min(s^2,s^{1/2}).
\end{equation*}

Hence. we have
\begin{equation} \label{eq:hlambdatailbound}
\mathbb{P}\Big( \Big|\frac{1}{n(n-1)} \sum_{1 \leq i \neq j \leq n} h_n(X_i,Y_j) \Big| >
s \Big) \leq L \exp \Big( - C n \epsilon_n^{(d+1)/2} \min(s^{1/2},s^2)
\Big),
\end{equation}
where $C$ is a constant independent of $n$.
Combining (\ref{eq:decouple2}) and (\ref{eq:hlambdatailbound}), and summing over $n$, gives
\begin{equation*}
\sum_{n=1}^{\infty} \mathbb{P}\Big( \Big| \frac{1}{n(n-1)} \sum_{1 \leq i
  \neq j \leq n} h_n(X_i,X_j) \Big| > s\Big) < \infty,
\end{equation*}
for all $s>0$. Therefore, 
\begin{equation} \label{eq:glambdahconvergeas}
\frac{1}{n(n-1)} \sum_{1 \leq i
  \neq j \leq n} h_n(X_i,X_j) \xrightarrow{a.s.} 0,
\end{equation}
as $n \to \infty$.

Applying (\ref{eq:glambdalconvergeas}), (\ref{eq:glambdamconvergeas}), and
(\ref{eq:glambdahconvergeas}) to (\ref{eq:fdecomp2}) completes the
proof.
\end{proof}

\subsubsection{Quadratic Term: General $\alpha$}
\label{general_sect}

Recall, from \eqref{eq:glambdadef} and \eqref{eq:lambdadef}, the forms of $G\Lambda_n (u)$ and $\Lambda(u)$.

\begin{lemma} \label{lemma:quadasconvergea} 
Fix a bounded, measurable function $u$ on the domain $D$, and let $\{\epsilon_n\}_{n \in \mathbb{N}}$ be a
sequence converging to zero such that 
\begin{equation} \label{eq:quadlemmaepsagen}
\sum_{n=1}^{\infty} n \exp(-n\epsilon_n^{d+1}) < + \infty.
\end{equation}
Then,
\begin{equation*}
\frac{1}{\sqrt{\epsilon_n}} \Big(G\Lambda_n (u)  - \Lambda (u)\Big)
\xrightarrow{a.s.} 0,
\end{equation*}
as $n \to \infty$.
\end{lemma}

\begin{proof}
Recall the forms of $\rho_\epsilon$ and $\Lambda_\epsilon$ in \eqref{rho_defn} and \eqref{new_lambda_def} respectively.  
We now introduce the intermediate term
\begin{equation} \label{def:overlineglambda}
\overline{G\Lambda_n} (u) \coloneqq \frac{1}{n}\sum_{1 \leq i \leq n}
\rho_{\epsilon_n}(X_i)^{\alpha} u(X_i).
\end{equation}
Then, 
\begin{equation*}
G\Lambda_n (u) - \Lambda (u)  = G\Lambda_n (u) - \overline{G\Lambda_n} (u) +
\overline{G\Lambda_n} (u) - \Lambda_{\epsilon_n} (u) + \Lambda_{\epsilon_n}(u) -
\Lambda (u).
\end{equation*}
The proof proceeds in three steps: 

\vskip .1cm
{\it Step 1.}  We first attend to $G\Lambda_n (u) -
\overline{G\Lambda_n}(u)$. We claim that
\begin{equation} \label{eq:firstconvergence}
\frac{1}{\sqrt{\epsilon_n}} (G\Lambda_n(u) - \overline{G\Lambda_n}(u))
\xrightarrow{a.s.} 0,
\end{equation}
as $n \to \infty$.
Define 
\begin{equation*} 
Z_i = \frac{1}{n-1}\sum_{\stackrel{1\leq j\leq n}{j\neq i}} \eta_{\epsilon_n}(X_i - X_j),
\end{equation*}
so that $G\Lambda_n (u) = \frac{1}{n}\sum_{i=1}^n Z_i^{\alpha} u(X_i)$.
Then we have
\begin{equation*}
\mathbb{E}[Z_i|X_i] =\rho_{\epsilon_n}(X_i),
\end{equation*}
and further, an application of Bernstein's inequality (Lemma 19.32 of \cite{van2000asymptotic}) gives
\begin{equation} \label{eq:firstbernstein}
\mathbb{P}\Big( |Z_i - \rho_{\epsilon_n}(X_i) | > t \Big| X_i\Big) \leq 2
\exp\Big(-\frac{1}{4} \frac{t^2(n-1)}{a + t b}\Big),
\end{equation}
where $a = \mathbb{E}[\eta_{\epsilon_n}(X_i - X_j)^2 | X_i]$ and $b =
\sup_{x \in D} |\eta_{\epsilon_n}(X_i - x)|$. Recalling the definition
$\eta_{\epsilon}(z) = \eta(z/\epsilon)/ \epsilon^d$ and the assumptions (K1), (K4),
we have  $a \leq
C / \epsilon_n^d$ and $b \leq C / \epsilon_n^d$. Therefore, inequality
(\ref{eq:firstbernstein}) implies 
\begin{equation*} \label{eq:secondbernstein}
\mathbb{P}\Big(| Z_i- \rho_{\epsilon_n}(X_i)| > t \Big| X_i\Big) \leq 2
\exp\Big(-C\frac{t^2(n-1)\epsilon_n^d}{(t+1)}\Big)
\end{equation*}
in terms of a constant $C$ not depending on $n$.
Applying a union bound gives
\begin{equation} \label{eq:supdegreedeviation}
\begin{aligned}
\mathbb{P}\Big(\sup_{1\leq i \leq n} | Z_i -
 \rho_{\epsilon_n}(X_i) | > t\Big) &\leq n \mathbb{P}\Big(| Z_i -
 \rho_{\epsilon_n}(X_i)| > t\Big) \\
&\leq C n \exp\Big(-C \frac{t^2(n-1)\epsilon_n^d}{(t+1)}\Big).
\end{aligned}
\end{equation}

By Lemma \ref{lemma:mollified}, there exist positive constants $A,B$ such that,
for sufficiently small $\epsilon$, both $\rho$ and $\rho_{\epsilon}$
take values in the interval $[A,B]$. Let $J_n$ denote the event that
$\sup_{1 \leq i \leq n} | Z_i - \rho_{\epsilon_n}(X_i) | < A/2$.  Then,
if $J_n$ holds, the inequality $A/2 < Z_i < B + A/2$ is satisfied for
all $i$. 
Since the function $x \mapsto x^{\alpha}$ is Lipschitz on
the interval $[A/2,B+A/2]$, we obtain
\begin{equation} \label{eq:mvtdifference}
|Z_i^{\alpha} - \rho_{\epsilon_n}(X_i)^{\alpha}| \leq
C |Z_i - \rho_{\epsilon_n}(X_i)|.
\end{equation}

Hence, when $J_n$ occurs, inequality (\ref{eq:mvtdifference}) and $\|u\|_{L^\infty}<\infty$ imply, with respect to another constant $C$ independent of $n$, that
\begin{eqnarray*}
&&|G\Lambda_n(u) - \overline{G\Lambda_n}(u)|  \leq \frac{1}{n} \sum_{1 \leq
  i \leq n} |Z_i^{\alpha} - \rho_{\epsilon_n}(X_i)^{\alpha}| |u(X_i)| \\
&&\ \  \leq C  \frac{1}{n} \sum_{1 \leq i \leq n} |Z_i - \rho_{\epsilon_n}(X_i)|  \leq C \sup_{1 \leq i \leq n} |Z_i - \rho_{\epsilon_n}(X_i)|,
\end{eqnarray*}
and moreover,
\begin{equation*}
\begin{aligned}
\frac{1}{\sqrt{\epsilon_n}}|G\Lambda_n(u) - \overline{G\Lambda_n}(u)| \leq
C \frac{1}{\sqrt{\epsilon_n}} \sup_{1 \leq i \leq n} |Z_i -
\rho_{\epsilon_n}(X_i) |. \end{aligned}
\end{equation*}

It follows, by \eqref{eq:supdegreedeviation}, that
\begin{eqnarray}
&&\mathbb{P}\Big(\Big\{\frac{1}{\sqrt{\epsilon_n}} |G\Lambda_n (u) -
\overline{G\Lambda_n}(u)| > t\Big\} \cap J_n \Big) \\
&&\ \ \ \ \ \ \ \leq
 \mathbb{P}\Big( C \frac{1}{\sqrt{\epsilon_n}} \sup_{1 \leq i \leq n}
|Z_i 
-  \rho_{\epsilon_n}(X_i)| > t \Big)\nonumber\\
\label{eq:devepsbound}
&&\ \ \ \ \ \ \ \ \  \leq \  C n \exp\Big(-C \frac{t^2 (n-1)\epsilon_n^{d+1}}{t+1}\Big). \nonumber
\end{eqnarray}
On the other hand, by (\ref{eq:supdegreedeviation}) again, we have
\begin{equation} \label{eq:suplowerbound}
\mathbb{P}(J_n^c) \leq   C n \exp\Big(-C \frac{A^2(n-1)\epsilon_n^{d}}{A+1}\Big).
\end{equation}

Combining (\ref{eq:suplowerbound}) and (\ref{eq:devepsbound}) gives
\begin{eqnarray*}
&&\mathbb{P}\Big( \frac{1}{\sqrt{\epsilon_n}} |G\Lambda_n (u) -
\overline{G\Lambda_n}(u)| > t\Big) \\
&&\ \ \ \ \  \leq \mathbb{P}\Big( \Big\{\frac{1}{\sqrt{\epsilon_n}} |G\Lambda_n (u) -
\overline{G\Lambda_n}(u)| > t\Big\}\cap  J_n \Big) + \mathbb{P}(J_n^c) \\
&&\ \ \ \ \  \leq C n \exp\Big(-C (n-1)\epsilon_n^{d+1}\Big),
\end{eqnarray*}
for some constant $C$ not depending on $n$.

From our assumption (\ref{eq:quadlemmaepsagen}) on $\epsilon_n$,
it follows that
\begin{equation*}
\sum_{n=1}^{\infty} \mathbb{P}\Big( \frac{1}{\sqrt{\epsilon_n}} |G\Lambda_n (u) -
\overline{G\Lambda_n}(u)| > t\Big) < \infty,
\end{equation*}
and so 
\begin{equation*}
\frac{1}{\sqrt{\epsilon_n}} (G\Lambda_n (u) -
\overline{G\Lambda_n}(u) ) \xrightarrow{a.s.} 0.
\end{equation*}

\vskip .1cm
{\it Step 2.} Now, we argue that

\begin{equation} \label{eq:secondconvergence}
\frac{1}{\sqrt{\epsilon_n}}\Big(\overline{G\Lambda_n}(u) -
\Lambda_{\epsilon_n}(u)\Big) \xrightarrow{a.s.} 0,
\end{equation}
as $n \to \infty$.

Noting \eqref{new_lambda_def}, since
\begin{equation*} 
\mathbb{E}[\rho_{\epsilon_n}(X_i)^{\alpha} u(X_i)] = \Lambda_{\epsilon_n}(u),
\end{equation*}
by Bernstein's inequality (Lemma 19.32 of \cite{van2000asymptotic}), we have
\begin{equation*}
\mathbb{P}\Big(\frac{1}{\sqrt{\epsilon_n}} \Big|\frac{1}{n}\sum_{1 \leq i
  \leq n} \rho_{\epsilon_n}(X_i)^{\alpha} u(X_i) - \Lambda_{\epsilon_n}(u)\Big|
> t \Big) \leq 2 \exp\Big( - \frac{1}{4} \frac{t^2 \epsilon_n n}{a +
  t \sqrt{\epsilon_n} b }\Big),
\end{equation*}
where $a = \mathbb{E}[(\rho(X_i)^{\alpha} u(X_i))^2]$ and $b =
\sup_{x \in D} |\rho_{\epsilon_n}(x)^{\alpha} u(x)|$.  Both of these are
bounded by a constant $C$, and so by the assumption
(\ref{eq:quadlemmaepsagen}) on $\epsilon_n$, we obtain
\begin{equation*}
\sum_{n=1}^\infty \mathbb{P}\Big(\frac{1}{\sqrt{\epsilon_n}} |\overline{G\Lambda_n}(u) - \Lambda_{\epsilon_n}(u)|
> t \Big) <  \infty,
\end{equation*}
and therefore \eqref{eq:secondconvergence}.

\vskip .1cm

{\it Step 3.} By Lemma \ref{lemma:lambdaepslim}, we have
\begin{equation*}
|\Lambda_{\epsilon_n}(u) - \Lambda (u) | \leq C \|u\|_{L^\infty} \epsilon_n.
\end{equation*}
It follows that
\begin{equation} \label{eq:thirdconvergence}
\frac{1}{\sqrt{\epsilon_n}} \Big(\Lambda_{\epsilon_n}(u) -
    \Lambda (u)\Big) \to 0,
\end{equation}
as $n \to \infty$.
Combining (\ref{eq:firstconvergence}),
(\ref{eq:secondconvergence}), and (\ref{eq:thirdconvergence}) gives
\begin{equation*}
\frac{1}{\sqrt{\epsilon_n}}\Big(G\Lambda_n(u) - \Lambda
(u)\Big) \xrightarrow{a.s.} 0,
\end{equation*}
as $n\rightarrow \infty$, finishing the proof of the lemma.
\end{proof}

\subsection{Proof of Theorem \ref{thm:asymptotics}}
\label{asymptotic_formulas_sect}
Recall equation (\ref{eq:oneminusmod}) which decomposes the modularity $Q_n(\mathcal{U}_n)$ with respect to partitions $\mathcal{U}_n$ of $\mathcal{X}_n$ induced from a partition $\mathcal{U}= \{U_k\}_{k=1}^K$ of $D$, where each of the sets $U_k$ have finite perimeter, ${\rm Per}(U_k;\rho^2)<\infty$.

Since $S / (n(n-1)^{\alpha}) = G\Lambda_n(1)$ and
$\int_D \rho^{1+\alpha}(x)\, dx = \Lambda (1)$, by Lemmas \ref{lemma:quadasconvergea0}, \ref{lemma:quadasconvergea1},
and \ref{lemma:quadasconvergea}, which cover the cases $\alpha =0$, $\alpha=1$, and $\alpha \neq 0,1$, we have 
\begin{equation}
\label{eq:asymp0}
\frac{S}{n(n-1)^{\alpha}} \xrightarrow{a.s.} \int_D \rho^{1+\alpha}(x)\, dx
\end{equation}
as $n\rightarrow\infty$. In particular, when $\alpha = 1$, we have, as $n \to \infty$,
\begin{equation}
\label{eq:asympeq2}
\frac{2m}{n(n-1)} = \frac{1}{n(n-1)} \sum_{\substack{1 \leq i, j
    \leq n \\ i \neq j }} \eta_{\epsilon_n}(X_i- X_j) \xrightarrow{a.s}\int_D \rho^2(x)\, dx.
\end{equation}
Further, these same lemmas, applied to the indicators $\{u_k = \mathds{1}_{U_k}\}_{k=1}^K$, imply that
\begin{equation*} 
\sum_{k=1}^K (G\Lambda_n (u_k -
1/K))^2 \xrightarrow{a.s.} \sum_{k=1}^K (\Lambda (u_k -
1/K))^2
\end{equation*}
as $n\rightarrow \infty$.
Hence, combining these limits,
\begin{equation} \label{eq:asympeq1}
\frac{n^2(n-1)^{2\alpha}}{S^2}\Big[\sum_{k=1}^K (G\Lambda_n (u_k -
1/K))^2\Big] \xrightarrow{a.s.} \sum_{k=1}^K \Big( \mu(U_k) - 1/K\Big)^2,
\end{equation}
as $n\rightarrow\infty$, with $d\mu(x) = \frac{\rho^{1+\alpha}(x)}{\int_D \rho^{1+\alpha}(x)\, dx}.$

By Lemma \ref{lemma:gtvasconverge}, we have, as $n \to 
\infty$,
\begin{equation*} \label{eq:asympeq3}
\sum_{k=1}^K GTV_{n}(u_k) \xrightarrow{a.s.} \sigma_{\eta}
\sum_{k=1}^K TV(u_k; \rho^2),
\end{equation*}
where $\sigma_{\eta} = \int_{\mathbb{R}^d} \eta(x)|x_1|\, dx$.  Therefore,
\begin{equation}
\label{eq:asympeq4}
\frac{n(n-1)}{4m}\sum_{k=1}^K GTV_n(u_k) \xrightarrow{a.s.} \frac{\sigma_\eta}{2\int_D\rho^2(x)dx}\sum_{k=1}^K TV(u_k;\rho^2),
\end{equation}
as $n\rightarrow\infty$.

Since $\lim_{n\rightarrow\infty}\epsilon_n= 0$, and $TV(u_k;\rho^2)<\infty$ for $1\leq k\leq K$ as the sets in $\mathcal{U}$ have finite perimeter, noting \eqref{eq:asympeq2} and \eqref{eq:asympeq4}, we have that
$\frac{\epsilon_n n(n-1)}{4m}\sum_{k=1}^K GTV_n(u_k)$ vanishes a.s., as $n\rightarrow\infty$.  Hence, from the limit \eqref{eq:asympeq1}, we obtain the first statement (\ref{eq:asympresult1}) in Theorem \ref{thm:asymptotics}.

To prove the second statement (\ref{eq:asympresult2}), we write, dividing \eqref{eq:oneminusmod} by $\epsilon_n$, that
\begin{eqnarray*}
&&\frac{1 - 1/K - Q_n(\mathcal{U})}{\epsilon_n} \\
&&\ \ =\  \frac{n^2(n-1)^{2\alpha}}{S^2}\Big[\sum_{k=1}^K (\frac{1}{\sqrt{\epsilon_n}}G\Lambda_n (u_k -
1/K))^2\Big] + \frac{n(n-1)}{4m} \sum_{k=1}^K
GTV_{n}(u_k).
\end{eqnarray*}
By assumption, the partition $\mathcal{U}$ is balanced with respect to
$d\mu$,  and so
$\sum_{k=1}^K (\mu(U_k) - 1/K)^2 = 0.$ Equivalently, recalling the definition \eqref{eq:lambdadef} of
$\Lambda$, we have $\Lambda(u_k - 1/K) = 0$ for $1 \leq k \leq K$.

Hence, writing $g_k = u_k - 1/K$, it follows that
\begin{equation} \label{eq:balancedgllimit}
\sum_{k=1}^K \Big(\frac{1}{\sqrt{\epsilon_n}}G\Lambda_n (u_k -
1/K)\Big)^2 = \sum_{k=1}^K \Big(\frac{1}{\sqrt{\epsilon_n}}\Big(G\Lambda_n
g_k - \Lambda g_k\Big)\Big)^2 \xrightarrow{a.s.} 0,
\end{equation}
as $n \to \infty$, by Lemmas \ref{lemma:quadasconvergea0}, \ref{lemma:quadasconvergea1},
and \ref{lemma:quadasconvergea} for the various cases of $\alpha$.
Combining \eqref{eq:asympeq4} and \eqref{eq:balancedgllimit} gives
\begin{equation*}
\begin{aligned}
\frac{1 - 1/K - Q_n(\mathcal{U})}{\epsilon_n} \xrightarrow{a.s.}
\frac{\sigma_{\eta}}{2 \int_D \rho^2(x)\, dx}  \sum_{k=1}^K
TV(u_k; \rho^2),
\end{aligned}
\end{equation*}
as $n \to \infty$. This finishes the proof of Theorem \ref{thm:asymptotics}.\qed

\section{Proof of Theorem \ref{thm:mainthm}: Optimal clusterings}
\label{proof_mainthm}
The general approach to proving Theorem \ref{thm:mainthm} is to formulate
both the modularity clustering problem
(\ref{prob:kmodoriginal}) and the continuum partitioning problem
(\ref{prob:continuumk}) as optimization problems on the common metric
space $(TL^1(D))^K$, say equipped with the product metric.  Although we wish to maximize the modularity, it will be
convenient later in Subsection \ref{reformulation_min_sect} to pose an equivalent problem of minimizing a related
energy $E_n$. 

Thus, the maximum modularity clusterings of the graph $\mathcal{G}_n =
(\mathcal{X}_n, W_n)$ will be related to the solution of
\begin{equation*} 
\begin{aligned}
& \underset{\mathcal{V} \in (TL^1(D))^K}{\text{minimize}}
& & E_{n}(\mathcal{V}), \\
\end{aligned}
\end{equation*}
and likewise, the optimal partitions of Problem  (\ref{prob:continuumk}) will be related to the solution of
\begin{equation*}
\begin{aligned}
& \underset{\mathcal{V} \in (TL^1(D))^K}{\text{minimize}}
& & E(\mathcal{V}).
\end{aligned}
\end{equation*}

We then argue in Subsection \ref{proof_gamma_conv_sect} that the random functionals $E_{n}$ $\Gamma\big( (TL^1)^K\big)$-converge to
$E$, in the sense of Definition \ref{def:nondeterministicgamma}, which when taken with an additional compactness property given in Subsection \ref{compactness_sect} will be
sufficient to imply that the minimizers of $E_n$ converge subsequentially in $(TL^1)^K$ to a minimizer of $E$, and thereby prove Theorem \ref{thm:mainthm} at the end of the section.

\subsection{Reformulation as a mimimization problem}
\label{reformulation_min_sect}

Recall the identity (\ref{eq:oneminusmod}),
\begin{eqnarray*} 
&&1 - 1/K - Q_n(\mathcal{U}_n) \\
&&\ \ \ = \ \frac{n^2(n-1)^{2\alpha}}{S^2}\Big[\sum_{k=1}^K \big(G\Lambda_n (u_{n,k} -
1/K)\big)^2\Big] + \epsilon_n \frac{n(n-1)}{4m} \sum_{k=1}^K
GTV_{n}(u_{n,k}),
\end{eqnarray*}
where $\mathcal{U}_n = \{ U_{n,k}\}_{k=1}^K$ is a partition of the data points $\mathcal{X}_n$ and $u_{n,k} = \mathds{1}_{U_{n,k}}\in I_n(D)$ for $1\leq k\leq K$.
As is our convention, we note that some of the $\{U_{n,k}\}_{k=1}^K$
may be empty sets, and so generally we have $|\mathcal{U}| \leq K$.

We define
\begin{equation} \label{eq:gfdefcluster}
F_n(\mathcal{U}_n) \coloneqq
\frac{1}{\epsilon_n}\frac{n^2(n-1)^{2\alpha}}{S^2}\Big[\sum_{k=1}^K \big(G\Lambda_n (u_{n,k}-
1/K)\big)^2\Big]
\end{equation}
and
\begin{equation} \label{eq:gtvdefcluster}
TV_n(\mathcal{U}_n) \coloneqq  \frac{n(n-1)}{4m} \sum_{k=1}^K GTV_{n}(u_{n,k}),
\end{equation}
so that the problem of \textit{maximizing}
$Q_n(\mathcal{U}_n)$ over clusterings $\mathcal{U}_n$ of $\mathcal{X}_n$ with
$|\mathcal{U}_n| \leq K$ is
equivalent to that of \textit{minimizing} $F_n(\mathcal{U}_n) + TV_n(\mathcal{U}_n)$.

We now place the
modularity optimization problem on the space
$(TL^1(D))^K$.  Recalling that $\nu_n$ denotes the empirical
measure, we define $M_n(D) \subset (TL^1(D))^K$ by
\begin{equation*}
M_n(D) \coloneqq \Big\{ ((\nu_n, u_{n,k}))_{k=1}^K : u_{n,k} \in I_n(D), \sum_{k=1}^K u_{n,k} = \mathds{1}_{\mathcal{X}_n} \Big\}.
\end{equation*}
As a notational convenience, we often denote elements of $M_n(D)$ by
\[ (\nu_n, \mathcal{U}_n) \coloneqq \{(\nu_n, u_{n,k})\}_{k=1}^K, \]
where $u_{n,k} = \mathds{1}_{U_{n,k}}$ and $\mathcal{U}_n = \{U_{n,k}\}_{k=1}^K$.

Define
 $E_{n} : (TL^1(D))^K \times \Omega \to
[0,\infty]$ by
\begin{equation} \label{eq:enenergydef}
E_{n}(\mathcal{V}_n) = \begin{cases}
F_n(\mathcal{U}_n) + TV_n(\mathcal{U}_n), & \text{if } \mathcal{V}_n = (\nu_n, \mathcal{U}_n)\in M_n(D), \\
\infty, & \text{otherwise.}
\end{cases}
\end{equation}
The energy minimization problem 
\begin{equation} \label{prob:binclusternew}
\begin{aligned}
& \underset{\mathcal{V}_n \in (TL^1(D))^K} {\text{minimize}}
& & E_{n}(\mathcal{V}_n), \\
\end{aligned}
\end{equation}
is equivalent to the $K$-class modularity clustering problem
(\ref{prob:kmodoriginal}), in the sense that $\mathcal{U}_n$ is a
solution to \eqref{prob:kmodoriginal} iff $\mathcal{V}_n =
(\nu_n,\mathcal{U}_n)$ is a solution to \eqref{prob:binclusternew}.

Similarly, we define continuum functionals on partitions $\mathcal{U} = \{U_k\}_{k=1}^K$ of $D$, via their indicators $\{u_k\}_{k=1}^K\subset I(D)$, by
\begin{equation} \label{eq:fdef}
F(\mathcal{U}) = \begin{cases} 0, & \text{if }\sum_{k=1}^K
(\mu(U^{(k)}) - 1/K)^2 = 0, \\
\infty, & \text{otherwise,}
\end{cases}
\end{equation}
and
\begin{equation} \label{eq:tvdef}
TV(\mathcal{U}) = C_{\eta,\rho} \sum_{k=1}^K TV(u_k ; \rho^2),
\end{equation}
where $d\mu = \rho^{1+\alpha} dx / \int_D \rho^{1+\alpha}(x)\, dx$ and
$$C_{\eta,\rho} = \frac{\sigma_{\eta}}{ 2\int_D \rho^2(x)\, dx} = \frac{ \int_D \eta(x)|x_1|dx }{2\int_D \rho^2(x)\, dx}.$$

Define $M(D) \subset (TL^1(D))^K$ by
\begin{equation*}
M(D) \coloneqq \Big\{ ((\nu, u_k))_{k=1}^K : u_k \in I(D), \sum_{k=1}^K u_k = \mathds{1}_D, \sum_{k=1}^K TV(u_k, \rho^2)<\infty\Big\}.
\end{equation*}
As before, we denote elements of $M(D)$ by
$$(\nu,\mathcal{U}) \coloneqq \{(\nu, u_k)\}_{k=1}^K,$$
where $\mathcal{U} = \{U_k\}_{k=1}^K$ and $u_k = \mathds{1}_{U_k}$ for $1\leq k\leq K$.

Define the energy $E : (TL^1(D))^K \to [0,\infty]$ as
\begin{equation} \label{eq:energydef}
E(\mathcal{V}) \coloneqq  
\begin{cases} F(\mathcal{U}) + 
TV(\mathcal{U}), & \text{if } \mathcal{V}= (\nu, \mathcal{U})\in M(D), \\
\infty, & \text{otherwise.}
\end{cases}
\end{equation}

Then, with $\mu$ as above and $\phi = \rho^2$, the continuum partitioning problem (\ref{prob:continuumk}), which does not include the prefactor $C_{\eta,\rho}$, is equivalent to
\begin{equation} \label{prob:continuumE}
\begin{aligned}
& \underset{\mathcal{V}\in (TL^1(D))^K}{\text{minimize}}
& & E(\mathcal{V}), 
\end{aligned}
\end{equation}
in the sense that $\mathcal{U}$ is a solution to
\eqref{prob:continuumk} iff  $\mathcal{V} = (\nu,\mathcal{U})$ is a solution to
\eqref{prob:continuumE}.

As noted in Subsection \ref{continuum problem}, since there is a solution to \eqref{prob:continuumk}, the problem \eqref{prob:continuumE} also possesses a solution.  In particular, the energy $E$ is not identically infinite.

\subsection{Gamma convergence}
\label{proof_gamma_conv_sect}

We now state the Gamma convergence used later in the proof of Theorem \ref{thm:mainthm}.

\begin{theorem} \label{thm:maingamma}
Under the assumptions of Theorem \ref{thm:mainthm}, 
the random functionals $E_n : (TL^1(D))^K \times \Omega \to [0,\infty]$, given in
  (\ref{eq:enenergydef}), $\Gamma$-converge in $(TL^1)^K$ to $E : (TL^1(D))^K \to [0,\infty]$, given in (\ref{eq:energydef}):
\begin{equation*}
E_{n} \xrightarrow{\Gamma((TL^1))^K} E,
\end{equation*}
as $n \to \infty$, in the sense of Definition \ref{def:nondeterministicgamma}.
\end{theorem}

\begin{proof}
The proof of Theorem \ref{thm:maingamma} is in two steps.  In Subsection \ref{liminf_subsect}, via Lemma \ref{lemma:eliminf}, we give the `liminf' estimate.  In Subsection \ref{recovery_subsect}, through Lemma \ref{lemma:erecovery}, we prove the `recovery sequence' property.
\end{proof}

\subsubsection{Liminf Inequality}
\label{liminf_subsect}

We now argue the liminf inequality for the
$\Gamma$-convergence in Theorem \ref{thm:maingamma}, according to Definition \ref{def:nondeterministicgamma}.  Recall that $\Omega_0$ denotes the probability $1$ set of realizations $\{X_i\}_{i\in \mathbb{N}}$, under which Proposition \ref{thm:tmaps} holds.
  
We first show a closure property of $M_n(D)$ and $M(D)$.

\begin{lemma} \label{lemma:mlemma}
On the probability $1$ set $\Omega_0$, the following statement holds:  If $\{ \mathcal{V}_n\}_{n \in \mathbb{N}}$ is a sequence in $M_n(D)$ and $\mathcal{V}_n
\xrightarrow{(TL^1)^K} \mathcal{V}$, then $\mathcal{V} \in M(D)$.
\end{lemma}

\begin{proof}  Fix a realization in the probability $1$ set $\Omega_0$.
Let $\mathcal{V}_n = ((\nu_n, u_{n,k}))_{k=1}^K$ and $\mathcal{V} =
((\mu_k, u_k))_{k=1}^K$.  By the characterization of $TL^1$
convergence, Lemma \ref{lemma:tl1convergence}, for each $1\leq k\leq K$  we have
$\nu_n \xrightarrow{w} \mu_k$, and so by Corollary \ref{empconverge} it follows that $\mu_k = \nu$.
Further, we have
\begin{equation*}
\lim_{n \to \infty} \int_D | u_k(x) - u_{n,k}(T_n x) | \rho(x) dx = 0,
\end{equation*}
where $\{T_n\}_{n\in \mathbb{N}}$ is the sequence of transportation maps given in Proposition \ref{thm:tmaps}.

  Hence,
as $\rho$ is bounded above and below on $D$, 
it follows that $u_k$ is the $L^1$ limit of a sequence of
indicator functions $\tilde{u}_{n,k}(x) \coloneqq u_{n,k}(T_n x) \in I(D)$. It follows, by subsequential Lebesgue a.e. convergence,
that $u_k \in I(D)$.  Similarly, the relation $\sum_{k=1}^K u_k = \mathds{1}_D$
follows from the corresponding relations for
$\{u_{n,k}\}_{k=1}^K$. Thus, $\mathcal{V} = ((\nu,u_k))_{k=1}^K \in M(D)$.
\end{proof}

We now establish the following technical lemma, which adapts a technique
from the proof of Theorem 1.1 in \cite{trillos2014continuum} to
relate graph functionals with their continuum nonlocal analogues.

\begin{lemma} \label{lemma:glambdalimit}
On the probability $1$ set $\Omega_0$, the following statement holds: Given any sequence of uniformly bounded, nonnegative functions
$\{g_n\}_{n \in \mathbb{N}}$, and a function $g$, if 
$(\nu_n, g_n) \xrightarrow{TL^1} (\nu, g)$, then 
\begin{equation*} \label{eq:glambdalimit}
\lim_{n \to \infty} G\Lambda_n (g_n) = \Lambda (g).
\end{equation*}

\end{lemma}

\begin{proof}
Fix a realization in the probability $1$ set $\Omega_0$.
Recall, from \eqref{eq:glambdadef} and \eqref{eq:lambdadef}, that 
\begin{eqnarray*}
G\Lambda_n (g_n) &=& \frac{1}{n} \sum_{i=1}^n \Big(
\frac{1}{n-1}\sum_{\stackrel{1\leq j\leq n}{j\neq i}} \eta_{\epsilon_n}(X_i - X_j)\Big)^{\alpha}
g_n(X_i)\\
&=& \frac{n^\alpha}{(n-1)^\alpha} \frac{1}{n} \sum_{i=1}^n \Big(
\frac{1}{n}\sum_{j=1}^n \eta_{\epsilon_n}(X_i - X_j) - \frac{\eta_{\epsilon_n}(0)}{n}\Big)^{\alpha}
g_n(X_i),
\end{eqnarray*}
and $\Lambda (g) = \int_D\rho^{1+\alpha}(x)g(x)dx$. 
Let $\{T_n\}_{n\in \mathbb{N}}$ be the transport maps in Proposition \ref{thm:tmaps}.  Since $T_{n\sharp} \nu = \nu_n$,
  by a change of variables, we have
\begin{equation*}
G\Lambda_n (g_n) = \frac{n^\alpha}{(n-1)^\alpha} \int_D \Big( \int_D \eta_{\epsilon_n}(T_nx - T_n
y)\rho(y)\, dy - \frac{\eta_{\epsilon_n}(0)}{n}\Big)^{\alpha} g_n(T_n x)\rho(x)\, dx.
\end{equation*}
For what follows, let $R_n(x) := \int_D \eta_{\epsilon_n}(T_n x -
T_n y) \rho(y)\, dy$. Recall also, from \eqref{rho_defn}, that $\rho_\epsilon(x) = \int_D \eta_\epsilon(x-y)\rho(y)dy$.

\vskip .1cm
{\it Step 1.}
 First, suppose that $\eta$ is of the form $\eta(x) = a$ for $|x| < b$ and
$\eta(x) = 0$ for $|x| > b$. 
Define
\begin{equation}\label{eq:overlineepsdef}
\overline{\epsilon}_n \coloneqq  \epsilon_n +
2 \frac{\|Id-T_n\|_{L^\infty}}{b},
\end{equation}
and note that, for Lebesgue a.e. $(x,y) \in D \times D$, 
\begin{equation*}
|x - y|\, > b \overline{\epsilon}_n {\rm \ \ implies \ \ } |T_n x-T_ny|\, > b\epsilon_n.
\end{equation*}

By the form of $\eta$, we have the bound
\begin{align*}
\eta\Big(\frac{T_nx-T_ny}{\epsilon_n}\Big) \leq \eta\Big(\frac{x-y}{\overline{\epsilon}_n}\Big).
\end{align*}
Integrating with respect to $\rho(y)\,dy$, and scaling appropriately, we obtain
\begin{equation*}
R_n(x) \leq (\overline{\epsilon}_n / \epsilon_n)^d \rho_{\overline{\epsilon}_n}(x),
\end{equation*}
 for Lebesgue a.e. $x \in D$.

By the assumption (I2)
on $\epsilon_n$, together with the estimates in Proposition \ref{thm:tmaps} on $\{T_n\}_{n\in \mathbb{N}}$,
it follows that $\epsilon_n$ vanishes slower than $\|Id-
T_n\|_{L^\infty}$, and so for large $n$ we have
\begin{equation*} 
\tilde{\epsilon}_n \coloneqq  \epsilon_n -
2\frac{\|Id-T_n\|_{L^\infty}}{b} > 0.
\end{equation*}
In particular, for Lebesgue a.e. $(x,y)\in D\times D$, 
\begin{equation*}
|T_n x- T_n y|\, > b\epsilon_n \ \ {\rm implies} \ \ |x-y|\, > b\tilde\epsilon_n
\end{equation*}
and so
\begin{equation} \label{eq:etatransportineq}
\eta\Big(\frac{x-y}{\tilde{\epsilon}_n}\Big) \leq \eta\Big(\frac{T_nx-T_ny}{\epsilon_n}\Big).
\end{equation}
Again, integrating with respect to $\rho(y)\, dy$ and scaling
appropriately, we obtain a lower bound of $R_n(x)$, and can write, for Lebesgue a.e. x,
\begin{equation} \label{eq:squeeze}
(\tilde{\epsilon}_n / \epsilon_n)^d \rho_{\tilde{\epsilon}_n}(x) \leq
R_n(x) \leq (\overline{\epsilon}_n / \epsilon_n)^d \rho_{\overline{\epsilon}_n}(x).
\end{equation}
By the assumption (I2) on the rate $\epsilon_n$, we observe that 
\begin{equation*}
\lim_{n \to \infty} \frac{\tilde{\epsilon}_n}{\epsilon_n} = \lim_{n
  \to \infty} 1 + 2\frac{\|Id - T_n\|_{L^\infty}}{\epsilon_n} = 1.
  \hspace{0.5cm} 
\end{equation*}
Similarly, we have $\lim_{n \to \infty}
\frac{\overline{\epsilon}_n}{\epsilon_n} = 1$.

In light of (\ref{eq:squeeze}), and Lemma \ref{lemma:mollified} in the Appendix, which shows $\lim_{\epsilon\rightarrow 0}\rho_\epsilon = \rho$ and bounds $\rho_\epsilon$ from above and below, we make two observations:
\begin{enumerate}
\item[i)] For Lebesgue a.e. $x$, we have $R_n(x) \to \rho(x)$ as $n\rightarrow\infty$. 
\item[ii)] There exist constants $A',B' > 0$ such that, for all large $n$, $A' \leq R_n(x) \leq B'$ for Lebesgue a.e. $x$.
\end{enumerate}

\vskip .1cm
{\it Step 2.} Now let $\eta$ be a simple function satisfying
assumptions $(K1)-(K4)$, which implies that we may write $\eta$ as a convex combination
$\eta = \sum_{l=1}^L \lambda_l \eta^l$ for functions $\eta^{(l)}$ satisfying the
assumptions of Step 1. We let $R_n^{(l)}(x) \coloneqq
\int_D  \eta_{\epsilon_n}^{(l)}(T_nx - T_ny) \rho(y)\,dy$
so that $R_n(x) = \sum_{l=1}^{(l)} \lambda_l R_n^{(l)}(x)$.

Hence,
\begin{enumerate}
\item[i)] For Lebesgue a.e. $x$, each $R_n^{(l)}(x)\rightarrow \rho(x)$ as $n\rightarrow\infty$. The same holds for the convex combination $R_n$.
\item[ii)] There exist constants $A',B' > 0$ such that, for all large $n$, $A' \leq R_n^{(l)}(x) \leq B'$ for Lebesgue a.e. $x$.
 Therefore, the
  same holds for $R_n$.
\end{enumerate}

Note also, since $\lim_{n\rightarrow\infty} n\epsilon_n^d =\infty$ by the assumption (I2), we have $\eta_{\epsilon_n}(0)/n \leq \|\eta\|_{L^\infty} (n\epsilon_n^d)^{-1}$ vanishes as $n\rightarrow\infty$.  We have then, by bounded convergence, that
\begin{equation*}
\begin{aligned}
\lim_{n\rightarrow\infty}
\Big| G\Lambda_n (g_n) - \Lambda (g) \Big| &= \lim_{n\rightarrow\infty} \Big| \int_D
R_n(x)^{\alpha} g_n(T_nx) \rho(x)\, dx - \int_D
g(x)\rho(x)^{1+{\alpha}}\, dx\Big| \\
& \leq \lim_{n\rightarrow\infty}
\int_D \Big| R_n(x)^{\alpha} g_n(T_nx) -
\rho(x)^{\alpha} g(x)\Big| \rho(x)\, dx.
\end{aligned} 
\end{equation*}

Since $\rho$ is bounded, we now argue that
\begin{equation*}
\lim_{n \to \infty} \int_D \Big|  R_n(x)^{\alpha} g_n(T_nx) -
\rho(x)^{\alpha} g(x)\Big|\, dx = 0. \hspace{0.5cm} 
\end{equation*}
By adding and subtracting  $\rho(x)^{\alpha}g_n(T_nx)$, we obtain
\begin{eqnarray*}
&& \int_D \Big|  R_n(x)^{\alpha} g_n(T_nx) -
\rho(x)^{\alpha} g(x)\Big|\, dx \\
&&\ \ \leq
 \int_D \Big| R_n (x)^{\alpha} -
\rho(x)^{\alpha} \Big| g_n(T_nx)\, dx + \int_D \Big| g_n(T_n x) - g(x) \Big|\rho^{\alpha}(x) \, dx.
\end{eqnarray*}

With respect to the first term on the right side, by assumption, $g_n(T_nx)$ is uniformly
bounded. Also, the sequence $R_n$ is bounded above and below, and converges Lebesgue
a.e. to $\rho$, so by dominated convergence the integral
vanishes in the limit.  With respect to the second term on the right side, suppose $(\nu_n, g_n )\xrightarrow{TL^1} (\nu, g)$ as $n\rightarrow\infty$.  Since $\rho$ is bounded above and below, we have $\rho^{\alpha}$ is
bounded, and the corresponding integral, by the characterization of $TL^1$ convergence in Lemma \ref{lemma:tl1convergence}, also vanishes in the limit. 

Hence, when the kernel $\eta$ is a simple function, we have that
\begin{equation}
\label{eq:simplelimit}
\lim_{n \to \infty} G\Lambda_n (g_n) = \Lambda (g). 
\end{equation}
\vskip .1cm

{\it Step 3.}
Now, we consider general $\eta$ satisfying properties $(K1)-(K4)$.  

We first approximate $\eta$ by simple functions $\eta^{(k)}$, satisfying
$(K2)-(K4)$, with $\eta^{(k)} \leq \eta$ and $\eta^{(k)} \rightarrow \eta$ pointwise.  Let $G\Lambda^{(k)}_n (g_n) = \frac{1}{n}\sum_{i=1}^n
\Big(\frac{1}{n-1}\sum_{j=1}^n
\eta^{(k)}_{\epsilon_n}(X_i-X_j)\Big)^{\alpha}g_n(X_i)$, and $\lambda_k = \int_D \eta^{(k)}(x)dx$.
Then, by
(\ref{eq:simplelimit}), we have
\begin{equation*}
\lim_{n \to \infty} \frac{1}{\lambda_k^{\alpha} } G\Lambda_n^{(k)} (g_n) =
  \Lambda (g).
\end{equation*}
Because $\eta^{(k)} \leq \eta$, and the sequence
$\{g_n\}_{n\in\mathbb{N}}$ is assumed nonnegative, we have that
\begin{equation*}
\liminf_{n \to \infty} \frac{1}{\lambda_k^{\alpha} } G\Lambda_n (g_n) \geq
  \Lambda (g),
\end{equation*}
and taking the limit $\lambda_k \rightarrow 1$ as $k\rightarrow\infty$ gives
\begin{equation} \label{eq:glineq1}
\liminf_{n \to \infty} G\Lambda_n (g_n) \geq
  \Lambda (g).
\end{equation}

Likewise, consider approximating $\eta$ by simple functions $\eta^{(k)}$
satisfying $(K2)-(K4)$, with $\eta^{(k)} \geq \eta$ and $\eta^{(k)} \rightarrow \eta$
pointwise. Then, similarly, we obtain that
\begin{equation} \label{eq:glineq2}
\limsup_{n \to \infty} G\Lambda_n (g_n) \leq
  \Lambda (g).
\end{equation}

Combining inequalities (\ref{eq:glineq1}) and (\ref{eq:glineq2}) gives that
\begin{equation*} \label{eq:limnonneg}
\lim_{n \to \infty} G\Lambda_n (g_n) = \Lambda (g),
\end{equation*}
completing the proof.
\end{proof}

\begin{lemma} \label{lemma:flim}
On the probability $1$ set $\Omega_0$, the following statement holds: Given any sequence
$\{\mathcal{V}_n \}_{n \in \mathbb{N}}$ such that
$\mathcal{V}_n \xrightarrow{(TL^1)^K} \mathcal{V}$, where
$\mathcal{V}_n = (\nu_n,\mathcal{U}_n) \in M_n(D)$ and $\mathcal{V}
= (\nu,\mathcal{U}) \in M(D)$, then
\begin{equation} \label{eq:gfliminf}
F(\mathcal{U}) \leq \liminf_{n \to \infty} F_n (\mathcal{U}_n).
\end{equation}
\end{lemma}

\begin{proof}Fix a realization in the probability $1$ set $\Omega_0$.
If $F(\mathcal{U}) = 0$, the above inequality holds trivially.  
We now consider the other case when
$F(\mathcal{U}) = \infty$.  Recalling the definitions (\ref{eq:fdef}) and \eqref{eq:lambdadef}
of $F$ and $\Lambda$, this means that there is a $1\leq k\leq K$ such that $u_k =
\mathds{1}_{U_k}$ satisfies $\Lambda(u_k-1/K) \neq 0$. Let
$\delta = \Lambda(u_k-1/K)^2 > 0$. 

Now, by Corollary \ref{empconverge}, $\nu_n \xrightarrow{w}\nu$, and so $(\nu_n, 1/K) \xrightarrow{TL^1}(\nu, 1/K)$ by Lemma \ref{lemma:tl1convergence}, as $n\rightarrow\infty$.  Therefore, by Lemma \ref{lemma:glambdalimit}, 
 as $n\rightarrow\infty$ we have
$$G\Lambda_n(1/K) =\frac{1}{K}\frac{S}{n(n-1)^\alpha}\rightarrow \frac{1}{K}\Lambda(1) = \frac{1}{K}\int_D\rho^{1+\alpha}(x)dx.$$
By decomposing $G\Lambda_n(u_{n,k} - 1/K) = G\Lambda_n(u_{n,k}) - G\Lambda_n(1/K)$ and noting Lemma \ref{lemma:glambdalimit} again, it follows that, if $\mathcal{V}_n \xrightarrow{(TL^1)^K}\mathcal{V}$, then
\begin{align*}
\lim_{n \to \infty} \sum_{k=1}^K \big(G\Lambda_{n}(u_{n,k}-1/K) \big)^2=
\sum_{k=1}^K\big(\Lambda(u_k-1/K) \big)^2.
\end{align*} 
In particular, there is an $N > 0$, depending on the realization, such that, for $n > N$, we have
\begin{align*}
\sum_{k=1}^K \big (G\Lambda_n(u_k - 1/K)\big )^2 \geq \delta/2.
\end{align*} 
Since 
$$F_n(\mathcal{U}_n) = \frac{1}{\epsilon_n}\frac{n^2(n-1)^{2\alpha}}{S^2}\sum_{k=1}^K \big (G\Lambda_n (u_{n,k} -
1/K)\big )^2,$$ 
and $\frac{n^2(n-1)^{2\alpha}}{S^2}
\xrightarrow{} (\int_D \rho^{1+\alpha}(x)\,dx)^{-2}$, it follows
that
\begin{equation*}
\begin{aligned}
\liminf_{n \to \infty} F_n(\mathcal{U}_n) &\geq \liminf_{n \to \infty}
\frac{C}{\epsilon_n}\sum_{k=1}^K \big (G\Lambda_n(u_{k} - 1/K)\big )^2 \\
&\geq \liminf_{n \to \infty} \frac{C\delta}{2\epsilon_n} = \infty.
\end{aligned}
\end{equation*}
Hence, in this case also, inequality (\ref{eq:gfliminf}) holds.
\end{proof}

\begin{lemma} \label{lemma:gtvliminf}
On the probability $1$ set $\Omega_0$, the following statement holds: Given any
sequence $\{u_n\}_{n \in \mathbb{N}}$ such that $(\nu_n,u_n) \xrightarrow{TL^1} (\nu,u)$, then 
\begin{equation*} \label{eq:gtvliminf}
\sigma_{\eta} TV(u;\rho^2) \leq \liminf_{n \to \infty} GTV_{n}(u_n),
\end{equation*}
where $\sigma_{\eta} = \int_{\mathbb{R}^d} \eta(x)|x_1|\, dx$.
\end{lemma}
\begin{proof}
The desired statement follows the same argument given for the liminf inequality for the Gamma convergence stated in Theorem 1.1 in
\cite{trillos2014continuum}--see Step 3 of Section 5.1
of \cite{trillos2014continuum}.  There, the probability $1$ set is $\Omega_0$.  We note this proof, although stated for $d\geq 2$, also holds in $d=1$ with the same notation. 

For a
sense of what is involved, note that in \cite{trillos2014continuum},
Section 4.1, it is proved,
for $\{u_{\epsilon}\}_{\epsilon > 0}$ with $u_{\epsilon} \to u$ in
$L^1(D)$, that the following liminf inequality holds:
\begin{equation} \label{eq:nonlocalliminf}
\sigma_{\eta} TV(u;\rho^2) \leq \liminf_{\epsilon \to 0} TV_{\epsilon}(u_\epsilon; \rho).
\end{equation}

To apply (\ref{eq:nonlocalliminf}) to the graph total variation, one
may proceed in three steps, analogous to the proof of Lemma
\ref{lemma:glambdalimit}, by first considering $\eta$ that are
normalized indicator
functions, then simple functions, and finally the general kernels
satisfying assumptions $(K1)-(K4)$, to show
\begin{equation} \label{eq:gtvnonlocalliminf}
\liminf_{n \to \infty} GTV_{n}(u_n) \geq \liminf_{n \to
  \infty} TV_{\epsilon_n}(u_{n}; \rho).
\end{equation}
Combining (\ref{eq:nonlocalliminf}) and (\ref{eq:gtvnonlocalliminf})
would then give the proof.
\end{proof}

\begin{lemma} \label{lemma:gtvlim} 
On the probability $1$ set $\Omega_0$, the following statement holds:  Given any sequence
$\{\mathcal{V}_n \}_{n \in \mathbb{N}}$ such that
$\mathcal{V}_n \xrightarrow{(TL^1)^K} \mathcal{V}$, where
$\mathcal{V}_n = (\nu_n,\mathcal{U}_n) \in M_n(D)$ and $\mathcal{V}
= (\nu,\mathcal{U}) \in M(D)$, then
\begin{equation*}
TV(\mathcal{U}) \leq \liminf_{n \to \infty} TV_{n}(\mathcal{U}_n).
\end{equation*}
\end{lemma}

\begin{proof}
Fix a realization in the probability $1$ set $\Omega_0$.  Note that $\nu_n \xrightarrow{w}\nu$ by Corollary \ref{empconverge},
and so $(\nu,1)\xrightarrow{TL^1} (\nu,1)$ by Lemma \ref{lemma:tl1convergence}, as $n\rightarrow\infty$.   Hence, by Lemma \ref{lemma:glambdalimit}, applied with $\alpha = 1$ and $g_n\equiv 1$, we have that 
$G\Lambda_n(1) = 2m/(n(n-1)) \xrightarrow{} \int_D \rho^2(x)\, dx$, as $n\rightarrow\infty$.

Recall now that
$TV_{n}(\mathcal{U}_n) =  ({n(n-1)}/{4m})\sum_{k=1}^K GTV_{n}(u_{n,k})$. 
If $\mathcal{V}_n \xrightarrow{(TL^1)^K} \mathcal{V}$, by Lemma \ref{lemma:gtvliminf}, we have
\begin{equation*}
\sigma_{\eta} TV(u_k; \rho) \leq \liminf_{n \to \infty} GTV_n(u_{n,k})
\end{equation*}
for $1\leq k\leq K$.  It follows that
\begin{equation*}
TV(\mathcal{U}) = C_{\eta,\rho} \sum_{k=1}^K TV(u_k;\rho^2) \leq \liminf_{n \to \infty}
\frac{n(n-1)}{4m} \sum_{k=1}^K  GTV_{n}(u_{n,k}),
\end{equation*}
where $C_{\eta,\rho} = \sigma_\eta/\big(2\int_D\rho^2(x)dx\big)$.
\end{proof}

\begin{lemma} \label{lemma:eliminf}
On the probability $1$ set $\Omega_0$, the following statement holds:  Given
any sequence $\{\mathcal{V}_n\}_{n\in \mathbb{N}}$ in $(TL^1(D))^K$ and $\mathcal{V}\in (TL^1)^K$ such that
$\mathcal{V}_n \xrightarrow{(TL^1)^K} \mathcal{V}$ as
$n\rightarrow\infty$, then
\begin{equation*} 
E(\mathcal{V}) \leq \liminf_{n \to \infty} E_n(\mathcal{V}_n).
\end{equation*}
\end{lemma}

\begin{proof}Fix a realization in the probability $1$ set $\Omega_0$.
Without loss of generality, we may assume that $\mathcal{V}_n =
(\nu_n,\mathcal{U}_n) \in M_n(D)$, as $E_n(\mathcal{V}_n)$ diverges otherwise.   If $\mathcal{V}_n \xrightarrow{(TL^1)^K}\mathcal{V}$, by Lemma \ref{lemma:mlemma}, it
follows that $\mathcal{V} = (\nu, \mathcal{U}) \in M(D)$. 
Also, by Lemmas \ref{lemma:flim} and \ref{lemma:gtvlim}, we have 
$$TV(\mathcal{U}) \leq \liminf_{n
  \to \infty} TV_n(\mathcal{U}_n) \ \ {\rm  and \ \ } F(\mathcal{U}) \leq \liminf_{n \to
  \infty} F_n(\mathcal{U}_n).$$
Adding these two liminf inequalities gives $\displaystyle{E(\mathcal{V}) \leq \liminf_{n \to \infty} E_n(\mathcal{V}_n)}$.
\end{proof}

\subsubsection{Existence of Recovery Sequence}
\label{recovery_subsect}
 The a.s. recovery sequence for $(\nu,\mathcal{U})$ in $M(D)$ will be $\{(\nu_n, \mathcal{U}_n)\}_{n\in \mathbb{N}}\subset M_n(D)$,
 where $\mathcal{U}_n$ is the partition of $\mathcal{X}_n$ induced
by $\mathcal{U}$. However, before proving this in Lemma \ref{lemma:erecovery}, we first establish
 a preliminary result.

\begin{lemma} \label{lemma:restrictconverge}
Fix $u \in L^1(D)$, and let $\{T_n\}_{n \in \mathbb{N}}$ be the
transport maps given in Proposition \ref{thm:tmaps}.  Then, a.s.,
\begin{equation*}
u \circ T_n \xrightarrow{L^1} u \, \text{ as } n\rightarrow\infty.
\end{equation*}
\end{lemma}

\begin{proof}
Let $u_{\epsilon}$ be a Lipschitz function such that $\int_D |u(x) -
u_{\epsilon}(x)| \, dx < \epsilon$.   Let $A>0$ be a lower bound for $\rho$ on $D$.  It follows that
\begin{eqnarray} 
&&A\int_D |u(T_n x) - u(x)|\, dx \leq \int_D |u(T_nx) -
u_{\epsilon}(T_nx)|\rho(x)\, dx \nonumber\\
&&\ \ \ \  + \int_D |u_{\epsilon}(T_nx) -
u_{\epsilon}(x)|\rho(x)\, dx + \int_D |u_{\epsilon}(x) - u(x)|\rho(x)\, dx. 
\label{eq:recovu}
\end{eqnarray}

We rewrite the first term in the right side of (\ref{eq:recovu}) in terms of the data set $\mathcal{X}_n$:
\begin{align*}
\int_D |u(T_nx) -
u_{\epsilon}(T_nx)|\rho(x)\, dx = \frac{1}{n}\sum_{i=1}^n |u(X_i) - u_{\epsilon}(X_i)|.
\end{align*}
By the strong law of large numbers, we have a.s. that $\lim_{n \to \infty}
\frac{1}{n}\sum_{i=1}^n |u(X_i) - u_{\epsilon}(X_i)| = \int_D
|u(x)-u_{\epsilon}(x)|\rho(x)\, dx < \epsilon$.

For the second term in (\ref{eq:recovu}), let $C$ be the Lipschitz constant for
$u_{\epsilon}$. Then a.s., by Proposition \ref{thm:tmaps},
\begin{align*}
\limsup_{n \to \infty} \int_D |u_{\epsilon}(T_nx) -
u_{\epsilon}(x)|\rho(x)\, dx \leq \limsup_{n\to\infty} C \|\rho\|_{L^{\infty}}\|T_n - Id\|_{L^\infty} = 0.
\end{align*}

Taking limits in (\ref{eq:recovu}) therefore gives a.s. that
\begin{align*}
\limsup_{n \to \infty} \int_D |u(T_n x) - u(x)|\, dx \leq
2A^{-1}\epsilon.
\end{align*}
Letting $\epsilon$ go to zero along a countable sequence
establishes the lemma.
\end{proof}

\begin{lemma} \label{lemma:erecovery}
Let $\mathcal{V} \in (TL^1(D))^K$. If $\mathcal{V} = (\nu, \mathcal{U}) \in M(D)$, let $\mathcal{V}_n =(\nu_n, \mathcal{U}_n) \in M_n(D)$, where
$\mathcal{U}_{n,k} = \mathcal{U}_k\cap \mathcal{X}_n$ for $1\leq
k\leq K$ and $n\geq 1$.  On the other hand, if $\mathcal{V} \notin M(D)$, let $\mathcal{V}_n = \mathcal{V}$ for $n\geq 1$.

Then, a.s., as $n\rightarrow\infty$,
$$\mathcal{V}_n \xrightarrow{(TL^1)^K}  \mathcal{V}  \ \ {\rm and \ \ }
E_n(\mathcal{V}_n) \rightarrow E(\mathcal{V}).$$
\end{lemma}

\begin{proof}
In the case that $\mathcal{V}\notin M(D)$, since the sequence $\{\nu_n\}_{n\in\mathbb{N}}$ is composed of distinct elements, for all large $n$, $\mathcal{V}\not\in M_n(D)$ and hence $E_n(\mathcal{V})=E(\mathcal{V})=\infty$.

Suppose now that $\mathcal{V}=(\nu,\mathcal{U})\in M(D)$.  In the following, we will use the fact that $u_{n,k}(x) = \mathds{1}_{U_{n,k}}(x) = \mathds{1}_{U_k}(x) = u_k(x)$ when $x\in \mathcal{X}_n$.  To show a.s. that $\mathcal{V}_n \xrightarrow{(TL^1)^K} \mathcal{V}$ as $n\rightarrow\infty$, by Lemma \ref{lemma:tl1convergence}, it is enough to show a.s. that $\nu_n \xrightarrow{w}\nu$ and, for $1\leq k\leq K$, that 
\begin{equation} \label{eq:utlimit}
\int_D |u_{n,k}(T_nx) - u_k(x)|d\nu(x) = \int_D 
  |u_k(T_nx) - u_k(x)|\rho(x)dx \rightarrow 0,
\end{equation}
as $n\rightarrow\infty$, since $T_nx\in \mathcal{X}_n$ implies
$u_{n,k}(T_n x) =u_k(T_n x)$.
  
The a.s. convergence $\nu_n \xrightarrow{w} \nu$ follows, for instance, by Corollary \ref{empconverge}.
On the other hand, the limit
(\ref{eq:utlimit}) follows by Lemma \ref{lemma:restrictconverge}. 

To show that a.s. $E_n(\mathcal{V}_n) \rightarrow E(\mathcal{V})$, we need to show that 
 \begin{equation}
\label{recovery_part}
TV_n(\mathcal{U}_n) \xrightarrow{a.s.} TV(\mathcal{U}) \ \ \ {\rm and \ \ \ }F_n(\mathcal{U}_n) \xrightarrow{a.s.} F(\mathcal{U}),
\end{equation}
as $n\rightarrow\infty$. Since condition (I2) on
$\{\epsilon_n\}_{n\in \mathbb{N}}$ implies condition (I1), we shall
see that these
limits in fact follow from the proof of
Theorem \ref{thm:asymptotics}. 

In particular, recall the
definitions \eqref{eq:gtvdefcluster} and \eqref{eq:tvdef} of $TV_{n}$
and $TV$ respectively. Then, since $u_{n,k} = u_k$ on
$\mathcal{X}_n$, the limit $TV_n(\mathcal{U}_n) \xrightarrow{a.s.} TV(\mathcal{U})$, as $n\rightarrow\infty$,
follows from \eqref{eq:asympeq4}.

With regards to the $F_n$ convergence, we consider two
possibilities. First, suppose that $\mathcal{U}$ is balanced. Then,
recalling the definition \eqref{eq:gfdefcluster}, and again noting
that $u_{n,k} = u_k$ on $\mathcal{X}_n$, it follows from
\eqref{eq:balancedgllimit} that $F_n(\mathcal{U}_n) \xrightarrow{a.s.}
F(\mathcal{U}) = 0$ as $n \to \infty$. 

Suppose now that $\mathcal{U}$ is not balanced, so that $\sum_{k=1}^K
\Big(\mu(U_k) - 1/K\Big)^2 \neq 0$. Then \eqref{eq:asympeq1} implies, as $n\rightarrow\infty$, 
that $F_n(\mathcal{U}_n) \xrightarrow{a.s.} F(\mathcal{U}) =
\infty$. Having considered all cases, \eqref{recovery_part} is established.
\end{proof}

\subsection{Compactness and Proof of Theorem \ref{thm:mainthm}}
\label{compactness_sect}
After a few preliminary estimates, we  supply the needed compactness property for the graph energies $\{E_n\}_{n\in \mathbb{N}}$ in Theorem \ref{thm:compactness}.  Then, we prove Theorem \ref{thm:mainthm} at the end of the section.

\begin{lemma}
\label{lemma:D_compactness}
Let $\{u_n\}_{n \in \mathbb{N}}$ be a sequence of indicator functions on $D$,
$u_n \in I(D)$, and $\{\epsilon_n\}_{n \in \mathbb{N}}$ be a sequence
of positive numbers with $\lim_{n\rightarrow\infty}\epsilon_n = 0$. If
\begin{equation*}
\sup_{n \in \mathbb{N}} TV_{\epsilon_n}(u_n;\mathds{1}_D) < \infty,
\end{equation*}
then $\{u_n\}_{n \in \mathbb{N}}$ is relatively compact with respect to the $L^1$ topology.
\end{lemma}
\begin{proof}

This result is a special case of Proposition 4.6 of
\cite{trillos2014continuum} and Theorem 3.1 of
\cite{alberti1998non}, which treat more involved settings.  However, for the convenience of the reader, we present
a streamlined argument in our situation, which makes use of the assumptions that
the functions $\{u_n\}_{n \in \mathbb{N}}$ are $\{0,1\}$-valued, and
that the kernel $\eta$ is compactly supported.

For what follows, we extend $u_n$ to all of $\mathbb{R}^d$ by setting
$u_n(x) = 0$ for $x \in \mathbb{R}^n \backslash D$. 
Note that
\begin{eqnarray*}
&&\int_{\mathbb{R}^d \times \mathbb{R}^d} \eta_{\epsilon_n}(x' -
x)|u_n(x') - u_n(x)|\, dx\, dx' \\
&&\ \ \ \ \ \  \ \ \ \ \ \ \ \ = \ \int_{D \times D} \eta_{\epsilon_n}(x' -
x)|u_n(x') - u_n(x)|\, dx\, dx'  \\
&&\ \ \ \ \ \ \ \ \ \ \ \ \ \ \ \ \ \ \  + 2 \int_{D \times (\mathbb{R}^d
\backslash D)} \eta_{\epsilon_n}(x' -
x)|u_n(x') - u_n(x)|\, dx\, dx',
\end{eqnarray*}
since $|u_n(x') - u_n(x)|$ vanishes on $(\mathbb{R}^d \backslash D) \times (\mathbb{R}^d
\backslash D)$. Now, as $\sup_{n \in \mathbb{N}}
TV_{\epsilon_n}(u_n;\mathds{1}_D) < \infty$, it follows that there is a constant
$C$ such that
\begin{equation} \label{eq:upperboundeps1}
\int_{D \times D} \eta_{\epsilon_n}(x' -
x)|u_n(x') - u_n(x)|\, dx\, dx'  \leq C \epsilon_n.
\end{equation}

Let $R$ be such that $\eta(z) = 0$ when $|z| \geq R$, and let $\partial_n
D = \{ x \in \mathbb{R}^d | d(x,\partial D) < R\epsilon_n
\}$ be the $R\epsilon_n$-neighborhood of $\partial D$.  Note that the volume of $\partial_n D$ is bounded by $C\epsilon_n$ for some constant $C$.  Then, by boundedness of $\{u_n\}_{n\in \mathbb{N}}$ and property (K1) of the kernel $\eta$, we have
\begin{eqnarray} \label{eq:upperboundeps2}
&&
\int_{D \times (\mathbb{R}^d
\backslash D)} \eta_{\epsilon_n}(x' -
x)|u_n(x') - u_n(x)|\, dx\, dx' \leq  \int_{D \times (\mathbb{R}^d
\backslash D)}  \eta_{\epsilon_n}(x' - x)\, dx\, dx' \nonumber\\
&&\ \ \leq \int_{D \times \partial_n D} \eta_{\epsilon_n}(x'-x)\,
dx\, dx' 
\leq \int_{\partial_n D} dx'  \leq C
\epsilon_n.
\end{eqnarray}
Combining (\ref{eq:upperboundeps1}) and
(\ref{eq:upperboundeps2}) thus gives, in terms of another constant $C$, that
\begin{eqnarray*}
C\epsilon_n &\geq& \int_{\mathbb{R}^d \times \mathbb{R}^d} \eta_{\epsilon_n}(x' -
x)|u_n(x') - u_n(x)|\, dx\, dx'\\
&=& \int_{\mathbb{R}^d \times \mathbb{R}^d} \eta_{\epsilon_n}(y)|u_n(x+y) - u_n(x)|\, dx\, dy.
\end{eqnarray*}

Now, let $\phi : \mathbb{R}^d \to \mathbb{R}$ be a non-negative, smooth, compactly supported
radial function, with $\phi(0) > 0$, $\phi \leq \eta$, and $|\nabla
\phi| \leq \eta$. Indeed, by property (K3), since $\eta(0) > 0$ and $\eta$ is continuous
at zero, such a $\phi$ exists.  Let $c = \int_{\mathbb{R}^d}
\phi(x)\, dx\leq 1$ and $\phi_n(x) = (c \epsilon_n^d)^{-1}\phi(x /
\epsilon_n)$. Define the sequence $\{w_n\}_{n \in \mathbb{N}}$ by
\begin{equation*}
w_n(x) = \int_{\mathbb{R}^d} \phi_n(y)u_n(x+y)\, dy.
\end{equation*}
We make the following observations.
\begin{enumerate}
\item The sequence $\{w_n\}_{n \in \mathbb{R}^d}$ is bounded in
  $L^1(D)$, 
\begin{equation*}
\int_D |w_n(x)|\, dx \leq \int_{\mathbb{R}^n} |w_n(x)|\, dx \leq
\int_{\mathbb{R}^d} |u_n(x)|\, dx \leq vol(D),
\end{equation*}
as $u_n$ is an
indicator function on $D$, and extended by zero outside of $D$.

\item We have the limit $\int_{\mathbb{R}^d}
  |w_n(x) - u_n(x)|\, dx \to 0$ as $n\rightarrow\infty$. Indeed, one may verify the inequality
\begin{equation*}
\begin{aligned}
\int_{\mathbb{R}^d}
  |w_n(x) - u_n(x)|\, dx &= \int_{\mathbb{R}^d} \Big|\int_{\mathbb{R}^d}
  \phi_n(y) u_n(x+y)\,dy - u_n(x)\Big|\, dx \\
& \leq \int_{\mathbb{R}^d \times \mathbb{R}^d} \phi_n(y) |u_n(x+y) -
u_n(x)|\, dx \, dy \\
&\leq \frac{1}{c}\int_{\mathbb{R}^d \times \mathbb{R}^d} \eta_{\epsilon_n}(y) |u_n(x+y) -
u_n(x)|\, dx \, dy  \leq C \epsilon_n,
\end{aligned}
\end{equation*}
since $\phi_n \leq c^{-1}\eta_{\epsilon_n}$ and \eqref{eq:upperboundeps1} holds.
\item The gradient $|\nabla w_n(x)|$ is uniformly bounded in $L^1(D)$. To see
  this, note that
\begin{equation*}
\begin{aligned}
\int_{\mathbb{R}^d} |\nabla w_n(x)|\, dx &= \int_{\mathbb{R}^d} \Big|
\int_{\mathbb{R}^d} \nabla \phi_n(y)u_n(x+y)\,dy\Big|\, dx \\
&= \int_{\mathbb{R}^d} \Big|
\int_{\mathbb{R}^d} \nabla \phi_n(y)(u_n(x+y) - u_n(x))\,dy\Big|\,
dx\\
&\leq \int_{\mathbb{R}^d \times \mathbb{R}^d} |\nabla
\phi_n(y)\|u_n(x+y) - u_n(x)|\, dy\,dx \\
& \leq \frac{1}{c \epsilon_n} \int_{\mathbb{R}^d \times \mathbb{R}^d}
\eta_{\epsilon_n}(y) |u_n(x+y) - u_n(x)|\, dy\,dx \\
& \leq \frac{1}{c} TV_{\epsilon_n}(u_n;\mathds{1}_D).
\end{aligned}
\end{equation*}
Here, the second equality follows as
$\int_{\mathbb{R}^d} \nabla \phi_n(y)\, dy = 0$ because $\phi$ is
compactly supported,.  The third inequality follows as $|\nabla \phi_n|\leq (c\epsilon_n)^{-1}\eta_{\epsilon_n}$.
\end{enumerate}

Now, because $\sup_n
\int_D |w_n(x)|\, dx < \infty$ and $\sup_n \int_D |\nabla w_n(x)|\,
dx < \infty$, it follows  by Theorem 3.23 of \cite{ambrosio2000functions} that $\{w_n \}_{n \in
  \mathbb{N}}$ is relatively compact in $L^1(D)$. Since $\int_{D}
  |w_n(x) - u_n(x)|\, dx \to 0$, as $n\rightarrow\infty$, it follows that the sequence
  $\{u_n\}_{n \in \mathbb{N}}$ is also relatively compact in $L^1(D)$,
  with the same cluster points as $\{ w_n \}_{n \in
  \mathbb{N}}$.
\end{proof}

Recall that $\Omega_0$ denotes the probability $1$ set of realizations of $\{X_i\}_{i\in mathbb{N}}$ under which Proposition \ref{thm:tmaps} holds.

\begin{lemma} \label{lemma:gtvcompact}
 Suppose $\{\epsilon_n\}_{n\in \mathbb{N}}$ satifies condition (I2).  On the probability $1$ set $\Omega_0$, the following holds: Given any sequence $\{u_n\}_{n \in \mathbb{N}}$ of indicator functions on the data points,
$u_n \in I_n(D)$, if
\begin{equation*}
\sup_{n \in \mathbb{N}} GTV_n\big(u_n\big) < \infty,
\end{equation*}
then $\{(\nu_n,u_n)\}_{n \in \mathbb{N}}$ is relatively compact
with respect to the $TL^1$ topology. 
\end{lemma}

\begin{proof}
We begin as in the proof of Lemma \ref{lemma:glambdalimit}. Fix a realization in the probability $1$ set $\Omega_0$.  Suppose that $\eta$ is of the form $\eta(x) = a$ for $|x| < b$ and
$\eta(x) = 0$ for $|x| > b$.  Let $\tilde{\epsilon}_n \coloneqq  \epsilon_n -
2 \frac{\|Id-T_n\|_{L^\infty}}{b}$, with respect to the transport maps $\{T_n\}_{n\in \mathbb{N}}$.  Then, for all large $n$, $\tilde{\epsilon}_n>0$, and we have the inequality (\ref{eq:etatransportineq}),
\begin{align*}
\eta\Big(\frac{x-y}{\tilde{\epsilon}_n}\Big) \leq
  \eta\Big(\frac{T_nx-T_ny}{\epsilon_n}\Big) \ \ \ {\rm  Lebesgue \ a.e. }\,\,
  (x,y) \in D \times D
\end{align*}
Let $A>0$ be a lower bound for $\rho$ on $D$.  Then,
\begin{eqnarray*}
&&A^2\int_D \eta\Big(\frac{x-y}{\tilde{\epsilon}_n}\Big)|u_n(T_nx) -
u_n(T_ny)|\,dx\,dy\\
&&\ \ \ \ \ \ \ \ \ \ \ \ \  \leq \int_D \eta\Big(\frac{x-y}{\tilde{\epsilon}_n}\Big)|u_n(T_nx) -
u_n(T_ny)|\rho(x)\rho(y)\,dx\,dy\\
&&\ \ \ \ \ \ \ \ \ \ \ \ \  \leq \int_D
\eta\Big(\frac{T_nx-T_ny}{\epsilon_n}\Big)|u_n(T_nx) - u_n(T_n
y)|\rho(x)\rho(y)\, dx\, dy\\
&&\ \ \ \ \ \ \ \ \ \ \ \ \ = \epsilon_n^{d+1} GTV_n(u_n).
\end{eqnarray*}
The above inequality is equivalent to
\begin{equation*}
(\tilde{\epsilon}_n / \epsilon_n)^{d+1} TV_{\tilde{\epsilon}_n}(u_n
\circ T_n;\mathds{1}_D) \leq GTV_n(u_n).
\end{equation*}

Since $\lim_{n\rightarrow\infty}\tilde{\epsilon}_n / \epsilon_n =1$, the bound $\sup_n
GTV_n(u_n) < \infty$ implies that $$\sup_n TV_{\tilde{\epsilon}_n}(u_n
\circ T_n; \mathds{1}_D ) < \infty.$$
It follows, by Lemma \ref{lemma:D_compactness}, that the family $\{
u_n \circ T_n \}_{n \in \mathbb{N}}$ is relatively compact with
respect to the
$L^1$ topology.
Since, by Corollary \ref{empconverge}, $\nu_n \xrightarrow{w} \nu$, we conclude by Lemma \ref{lemma:tl1convergence} that $\{ (\nu_n,u_n) \}_{n \in \mathbb{N}}$ is relatively
compact in $TL^1$.

Suppose now $\eta$ is an arbitrary kernel satisfying assumptions
$(K1)$-$(K4)$. Since $\eta$ is continuous at zero, and $\eta(0) > 0$,
there is some radius $R$ such that $\tilde{\eta} = \frac{\eta(0)}{2} \mathds{1}_{|x|
  < R}$ satisfies $\tilde{\eta} \leq \eta$. Let $c =
\int_{\mathbb{R}^d} \tilde{\eta}(x)\, dx$.  Then, if $\widetilde{GTV}_n$
denotes the graph total variation associated to the kernel
$\tilde{\eta}/c$ (instead of $\eta$), we have
\begin{equation*}
GTV_n(u_n) \geq c \,\widetilde{GTV}_n(u_n). 
\end{equation*}
Since $\sup_n GTV_n(u_n) < \infty$ implies $\sup_n \widetilde{GTV}_n(u_n)
< \infty$, it follows from our previous discussion that the sequence
$\{ (\nu_n, u_n) \}_{n \in \mathbb{N}}$ is relatively compact in $TL^1(D)$.
\end{proof}

\begin{theorem} \label{thm:compactness}
Suppose $\{\epsilon_n\}_{n\in\mathbb{N}}$ satisfies condition (I2).  On the probability $1$ set $\Omega_0$, the following holds:
Given any sequence
$\{\mathcal{V}_n\}_{n \in \mathbb{N}} \subset (TL^1)^K$ ,
if 
\begin{equation*}
 \sup_{n \in \mathbb{N}} E_n(\mathcal{V}_n) < \infty, 
\end{equation*}
then $\{\mathcal{V}_n\}_{n \in \mathbb{N}}$ is relatively compact with
respect to the $(TL^1)^K$ topology.
\end{theorem}

\begin{proof}  Fix a realization in the probability $1$ set $\Omega_0$.  Suppose $\sup_n E_n(\mathcal{V}_n) < \infty$.  By definition of $E_n$, it follows that $\mathcal{V}_n = (\nu_n,
\mathcal{U}_n) \in M_n(D)$ where $\mathcal{U}_n = \{U_{n,k}\}_{k=1}^K$ for $n\in \mathbb{N}$.
By Corollary \ref{empconverge}, we have $\nu_n \xrightarrow{w}\nu$.  Since $(\nu_n,\mathds{1}_D) \xrightarrow{TL^1} (\nu,\mathds{1}_D)$ by Lemma \ref{lemma:tl1convergence},
we have, by Lemma \ref{lemma:glambdalimit}, that $G\Lambda_n(1) = 2m/(n(n-1)) \rightarrow \int_D \rho^2(x)dx$.
Recall now that $E_n(\mathcal{V}_n) = TV_n(\mathcal{U}_n) + F_n(\mathcal{U}_n)$ and
\begin{equation*}
TV_n(\mathcal{U}_n) =  \frac{n(n-1)}{4m} \sum_{k=1}^K
GTV_{n}(u_{n,k}),
\end{equation*}
where $u_{n,k} = \mathds{1}_{U_{n,k}}$ for $1\leq k\leq K$.
Hence, given that
$\sup_{n \in \mathbb{N}} E_n(\mathcal{V}_n) < \infty$,
we have
\begin{equation*}
\sup_{n \in \mathbb{N}} GTV_n(u_{n,k}) < \infty,
\end{equation*}
 for $1 \leq k \leq K$. Thus, by Lemma \ref{lemma:gtvcompact}, the collection $\{ (\nu_n, u_{n,k}) \}_{n \in \mathbb{N}}$ is
relatively compact in $TL^1$ for $1\leq k \leq K$. Thus,
$\{\mathcal{V}_n = ((\nu_n, u_{n,k}))_{k=1}^K=(\nu_n,\mathcal{U}_n)\}_{n \in \mathbb{N}}$ is relatively compact in $(TL^1)^K$.
\end{proof}

\noindent{\bf Proof of Theorem \ref{thm:mainthm}.}
We have seen in Theorem \ref{thm:maingamma} that
\begin{equation*} E_{n} \xrightarrow{\Gamma((TL^1)^K)} E,
\end{equation*}
in the sense of Definition \ref{def:nondeterministicgamma}.
By Theorem \ref{thm:compactness}, 
the graph
energies $E_{n}$ have the compactness property according to Definition \ref{def:random_compactness}.  Also, as noted in Subsection \ref{reformulation_min_sect}, the energy $E$ is not identically infinite.

For each realization $\{X_i\}_{i\in \mathbb{N}}$, consider a partition $\mathcal{U}_n^* \in \argmax_{|U_n| \leq K} Q_n(\mathcal{U}_n)$.  Then, by the discussion in Subsection \ref{reformulation_min_sect}, $\mathcal{U}_n^*$ is a minimizer of $F_n+TV_n$, and so $\mathcal{V}_n = (\nu_n, \mathcal{U}_n^*) \in M_n(D)$ is a minimizer of $E_n$.  The sequence $\{\mathcal{V}_n\}_{n \in \mathbb{N}}$ is also bounded in $(TL^1)^K$:  Indeed, we have $\int_D |x|d\nu_n(x) \leq \sup_{x\in D}|x|$ and, for $u_{n,k}^* = \mathds{1}_{U_{n,k}^*}$, $\|u_{n,k}^*\|_{L^1} \leq \mbox{vol}(D) \|u_{n,k}^*\|_{L^\infty} \leq \mbox{vol}(D)$.

Hence, on the full set of realizations $\{X_i\}_{i\in\mathbb{N}}$, denoted as $\Omega$, the sequence $x_n = \mathcal{V}_n$, on the metric space $(TL^1)^K$, satisfies the hypotheses of Theorem \ref{thm:nondeterministicgamma}.  Therefore, with respect to realizations $\{X_i\}_{i\in\mathbb{N}}$ on a probability $1$ set $\Omega^*$, $\mathcal{V}_n$ converges in $(TL^1)^K$, perhaps along a subsequence, to a limit $\mathcal{V}$, which is a minimizer of $E$, and is therefore of the form $\mathcal{V} =(\nu, \mathcal{U}^*)$. Since the `liminf' inequality, Lemma  \ref{lemma:eliminf}, holds on $\Omega_0$, we note that $\Omega^* \subset \Omega_0$. Moreover, by Corollary \ref{empconverge}, $\nu_n\xrightarrow{w} \nu$ on $\Omega^*$.  Therefore, by \eqref{equivalent_convergence}, on $\Omega^*$, $\mathcal{U}_n^*$ converges weakly, perhaps along a subsequence, to the limit $\mathcal{U}^*$, which is an optimal partition of the continuum problem \eqref{prob:continuumk} with $\phi=\rho^2$ and $d\mu = \rho^{1+\alpha}/\int_D \rho^{1+\alpha}(x)dx$.  

In fact, we may conclude that, on $\Omega^*$, the distances $\beta_n \coloneqq \inf \{ d_{(TL^1)^K}(\mathcal{V}_n,\mathcal{V}) : \mathcal{V} \in \argmin E\}$, where $d_{(TL^1)^K}$ is the product metric for $(TL^1)^K$ convergence, satisfy $\beta_n \to 0$ as $n \to \infty$. For if not, there is a subsequence $\{n_m\}_{m \in \mathbb{N}}$ with $\beta_{n_m} \to \beta > 0$. However, by the above discussion one may find a further subsequence $\{n_m' \}_{m \in \mathbb{N}}$ with $\beta_{n_m'} \to 0$, a contradiction.

Moreover, if problem \eqref{prob:continuumk} has a unique solution $\mathcal{U}^* = \{U_k^*\}_{k=1}^K$, modulo permutations, then  $\argmin E = \{ ((\nu, u^*_{\pi(k)}))_{k=1}^K : \pi \in \mbox{Sym}(K) \}$, where $\mbox{Sym}(K)$ denotes the permutations of $\{1,\ldots,K\}$ and $u^*_k = \mathds{1}_{U^*_k}$ for $1 \leq k \leq K$. Thus, on the probability 1 set $\Omega^*$, since $\beta_n \to 0$, one may construct a sequence $\{ \pi_n \}_{n \in \mathbb{N}}$ of permutations such that, as $n \to \infty$, $((\nu_n,u_{n,\pi_n(k)}^*))_{k=1}^K$ converges in $(TL^1)^K$ to $(\nu,\mathcal{U}^*) = ((\nu,u_k^*))_{k=1}^K$. Hence, by \eqref{equivalent_convergence}, $\mathcal{U}_n^*$ converges to $\mathcal{U}^*$ weakly, in the sense of \eqref{eq:convergeeq}. \qed

\section{Appendix}
\label{appendix}

\subsection{Approximation Lemma}
Recall that we have defined $\rho_\epsilon(x) = \int_D \eta_\epsilon(x-y)\rho(y)dy$.

\begin{lemma} \label{lemma:mollified}  Under the standing assumptions on $\rho$, $D$ and $\eta$ in Subsection \ref{results},
we have the following:
\begin{enumerate}
\item[(i)] $\rho_{\epsilon}$ converges pointwise to $\rho$ as $\epsilon\downarrow 0$.
\item[(ii)] There exists a constant $C$ such that, for sufficiently small
  $\epsilon$, 
\begin{equation} \label{eq:mollifiedrate}
\int_D |\rho_{\epsilon}(x) - \rho(x)|\, dx \leq C
  \epsilon. 
\end{equation}
\item[(iii)] There exist constants $a,b$ such that, for sufficiently small
  $\epsilon$,
\begin{equation*}
0 < a \leq \rho_{\epsilon}(x) \leq b \hspace{0.5cm} \text{for all $x \in D$}.
\end{equation*}
\end{enumerate}
\end{lemma}

\begin{proof}
The pointwise convergence in item (i) follows
from continuity of $\rho$. 

We now
focus attention on item (ii), inequality (\ref{eq:mollifiedrate}).
For the moment, fix $x \in D$. Since $\eta$ is compactly
supported, we take $R$ such that $\eta(z) = 0$ for $|z| > R$. Then
for $0 < \epsilon < \mbox{dist}(x,\partial D)/R$, we
have, since $\int_{\mathbb{R}} \eta(x) dx = 1$, that
\begin{align*}
\rho(x) = \int_{x+z \in D} \eta_{\epsilon}(z) \rho(x) \, dz,
\end{align*}
and so,
\begin{align*}
\rho_{\epsilon}(x) - \rho(x) &= \int_{x+z \in D} \eta_{\epsilon}(z)
\big(\rho(x+z) - \rho(x) \big) \, dz.
\end{align*}

Let $L$ be a Lipschitz constant for $\rho$. Then,
\begin{equation*} \label{eq:mollifierinequalityint}
\begin{aligned}
|\rho_{\epsilon}(x) - \rho(x)| & \leq L \int_{x+z \in D}
\eta_{\epsilon}(z) |z| \, dz \\
& \leq L  \int_{|z| \leq \mbox{diam}(D)} \eta_{\epsilon}(z) |z| \, dz
\\
& = L \epsilon \int_{|z| \leq \mbox{diam}(D) / \epsilon} \eta(z) |z|
\, dz \\
& \leq L \epsilon \int_{\mathbb{R}^d} \eta(z) |z| \, dz.
\end{aligned}
\end{equation*}

Because $\int_{\mathbb{R}} \eta(x) dx = 1$ and
$\eta(z) = 0$ for $|z| > R$, the above implies
\begin{equation} \label{eq:rhoepsestimate}
|\rho_{\epsilon}(x) - \rho(x)| \leq LR \epsilon.
\end{equation}

Let $D_{R\epsilon} = \{ x \in D\ |\
\mbox{dist}(x, \partial D) < R\epsilon \}$ and $\partial_{R\epsilon}D
= D \setminus D_{R\epsilon}$. 
 We split the integral,
\begin{align*}
\int_D |\rho_{\epsilon}(x) - \rho(x)| \,dx = \int_{D_{R\epsilon}}
  |\rho_{\epsilon}(x) - \rho(x)| \, dx + \int_{\partial_{R\epsilon} D}
  |\rho_{\epsilon}(x) - \rho(x)| \,dx,
\end{align*}
and consider the two terms separately.

On $D_{R\epsilon}$, applying inequality (\ref{eq:rhoepsestimate}) yields
\begin{equation} \label{eq:rhointeriorint}
\int_{D_{R\epsilon}} |\rho_{\epsilon}(x) - \rho(x)| \,dx \leq \mbox{vol}(D)LR\epsilon.
\end{equation}

For the second integral, note that because the boundary is Lipschitz, there is a
$\epsilon_0 > 0$ and constant $C$ such that
$\mbox{vol}(\partial D_{R\epsilon}) \leq C R \epsilon$ for $0 <
\epsilon < \epsilon_0$. It follows that
\begin{equation} \label{eq:rhoboundaryint}
\int_{\partial_{R\epsilon} D} |\rho_{\epsilon}(x) - \rho(x)| \, dx
\leq  2C\|\rho\|_{L^\infty}R\epsilon.
 \end{equation}

Combining (\ref{eq:rhointeriorint}) and (\ref{eq:rhoboundaryint}) gives, for sufficiently small $\epsilon > 0$,
\begin{equation*} \label{eq:rhol1bound}
\int_D |\rho_{\epsilon}(x) - \rho(x)| \,dx \leq C R\epsilon,
\end{equation*}
where $C$ is a constant independent of $\epsilon$.

For item (iii) of the lemma, note that because the boundary of $D$ is
 Lipschitz, there exists constants $r_0,r_1$ such that, for any $x
\in D$ and $0 < r < r_1$, 
\begin{equation*} \label{eq:appendixineq}
0 < r_0 < \frac{\mbox{vol}(D \cap B(x,r))}{\mbox{vol}(B(x,r))},
\end{equation*}
where $B(x,r)$ denotes the ball of radius $r$ centered at $x$ (cf. the discussion about cone conditions in
Section 4.11 of \cite{adams2003sobolev}). 

By assumption, $\rho$ is bounded above and below:  $0<A\leq \rho(\cdot) \leq B$.  Also, by assumption (K3), $\eta$ is continuous at zero and $\eta(0) > 0$, so we
may take $r$ so that $0 < \eta(0) / 2 \leq \eta(z)$ for $|z| < r$.
Letting $\tilde{\eta}(x) = \eta(0)\mathds{1}_{B(x,r)}/{2\mbox{vol}(B(x,r))}$, we
thus have
\begin{equation*}
0 < A r_0 \eta(0) / 2 \leq A \int_D \tilde{\eta}(x-y)\,dy  \leq \int_D \eta_{\epsilon}(x-y)\rho(y)\, dy = \rho_\epsilon(x).
\end{equation*}
Since $\int_{\mathbb{R}}\eta_\epsilon(x) dx = 1$, the upper bound $\rho_\epsilon(x) \leq B$ holds.  This completes the lemma.
\end{proof}

\subsection{Estimates for GTV Limsup}

Recall, from the proof of Lemma \ref{lemma:gtvasconverge}, that $f_n(X_i,X_j) = \frac{1}{\epsilon_n}
\eta_{\epsilon_n}(X_i-X_j)|u(X_i) - u(X_j)|$, $l_n(X_i) = \frac{1}{\epsilon_n} \int_D
\eta_{\epsilon_n}(X_i - y)|u(X_i) - u(y)|\rho(y)\, dy$, and $m_n(X_j) =
\frac{1}{\epsilon_n} \int_D \eta_{\epsilon_n}(x-X_j)|u(x)-u(X_j)|\rho(x)\,
dx$.  Here, $u\in I(D)$ is such that $TV(u;\rho^2)<\infty$.

\begin{lemma} \label{lemma:gtvestimates}
There exists a constant $C$ such that, for sufficiently large $n$,
we have the following bounds:

\setlength{\tabcolsep}{10pt}
\renewcommand{\arraystretch}{1.5}
\begin{center}
\begin{tabular}{l  l }

$|\mathbb{E}f_n| \leq C$, & $\|f_n\|_{L^2 \to L^2} \leq
                           C/\epsilon_n $, \\

$\mathbb{E}f_n^2 \leq C / \epsilon_n^{d+1}$, & $\mathbb{E}l_n^2
                                                    \leq
                                                    C / \epsilon_n$, \\
$\|\mathbb{E}_Y f_n^2\|_{L^\infty} \leq C / \epsilon_n^{d+2}$, & $ \|l_n\|_{L^\infty} \leq C /
                                                  \epsilon_n$, \\
$\|\mathbb{E}_X f_n^2\|_{L^\infty} \leq C / \epsilon_n^{d+2}$, & $ \mathbb{E}m_n^2
                                                  \leq C / \epsilon_n$,
  \\
$\|f_n\|_{L^\infty} \leq C/\epsilon_n^{d+1}$, & $\|m_n\|_{L^\infty} \leq C / \epsilon_n.$
\end{tabular}
\end{center}

\end{lemma}

\begin{proof}

Note, by
Lemma \ref{lemma:tvlim}, that $\mathbb{E}f_n = TV_{\epsilon_n}(u;\rho)$ converges to $TV(u;\rho^2)<\infty$ as $n \to \infty$. 
Hence
$|\mathbb{E}f_n| \leq C$. 

For $\mathbb{E}f_n^2$, we have
\begin{equation*}
\begin{aligned}
\mathbb{E}f_n^2 = \int_D \int_D \frac{1}{\epsilon_n^2} (\eta_{\epsilon_n}(x-y))^2|u(x)-u(y)|^2\rho(x)\rho(y)\,dx\,dy.
\end{aligned}
\end{equation*}
Since $\eta$ is bounded above, we have $(\eta_{\epsilon_n}(x-y))^2 \leq C
\eta_{\epsilon_n}(x-y) / \epsilon_n^d$. Because $u \in I(D)$, we have
$|u(x) - u(y)|^2 = |u(x)-u(y)|$. Hence, 
\begin{equation*}
\begin{aligned}
\mathbb{E}f_n^2 &\leq C \int_D \int_D
\frac{1}{\epsilon_n^2}\frac{1}{\epsilon_n^d}
\eta_{\epsilon_n}(x-y)|u(x)-u(y)|\rho(x)\rho(y)\,dx\,dy \\
&\leq C\,TV_{\epsilon_n}(u;\rho)  / \epsilon_n^{d+1} \leq C' / \epsilon_n^{d+1}
\end{aligned}
\end{equation*}
Likewise, one may get the bound
\begin{equation*}
\begin{aligned}
\mathbb{E}_Y f_n^2 &= \int_D \frac{1}{\epsilon_n^2}
(\eta_{\epsilon_n}(x-y))^2 |u(x)-u(y)|^2 \rho(y)\,dy \\
& \leq C \int_D \eta_{\epsilon_n}(x-y) \rho(y)\, dy / \epsilon_n^{d+2}  \leq
C' / \epsilon_n^{d+2},
\end{aligned}
\end{equation*}
with the last inequality following from our assumption that $\rho$ is bounded.
By the symmetry of $f_n$, this also gives
\begin{equation*}
\mathbb{E}_X f_n^2 \leq C' / \epsilon_n^{d+2}.
\end{equation*}

Recall that $\|f_n\|_{L^2 \to L^2}$ is given by
\begin{eqnarray*}
&&\|f_n\|_{L^2 \to L^2}\\
&&\ \ \   = \sup \{ \int_{D \times D} f_n(x,y) h(x)g(y)d\nu(x) d\nu(y) : \|h\|_{L^2(D,\nu)} \leq 1, \|g\|_{L^2(D,\nu)} \leq 1 \}.\end{eqnarray*}
It is straightforward to show $\int_D |f_n(x,y)|\rho(x)\, dx \leq C/\epsilon_n$ and  
$\int_D |f_n(x,y)|\rho(y)\, dy \leq C/\epsilon_n$, for some constant $C$ independent of $n$.
Thus, with respect to the map $Jg(x) = \int_D f_n(x,y)g(y)\rho(y)dy$, by Theorem 6.18 of \cite{folland2013real}, we have
\begin{equation*}
\|Jg\|_{L^2(D,\nu)} \leq C \|g\|_{L^2(D,\nu)} / \epsilon_n,
\end{equation*}
which implies $\|f_n\|_{L^2 \to L^2} \leq C / \epsilon_n$.

Now considering $l_n$, we have
\begin{equation*}
\mathbb{E}l_n^2 = \int_D \Big( \frac{1}{\epsilon_n} \int_D
\eta_{\epsilon_n}(x - y) |u(x) - u(y)|\rho(y)\,dy\Big)^2\rho(x)\,dx.
\end{equation*}
By Jensen's inequality, it follows that
\begin{equation*}
\begin{aligned}
\mathbb{E}l_n^2 &\leq \frac{1}{\epsilon_n^2} \int_{D \times D}
\eta_{\epsilon_n}(x-y)|u(x) - u(y)|^2 \rho(y)\rho(x)\,
dy\, dx \\ 
& = \frac{1}{\epsilon_n^2} \int_{D \times D} \eta_{\epsilon_n}(x-y)
|u(x)-u(y)|\rho(y)\rho(x)\,dy \,dx \\
& \leq \frac{1}{\epsilon_n} \mathbb{E}f_n
 \leq C / \epsilon_n.
\end{aligned}
\end{equation*}

Similarly, we have
\begin{equation*}
\begin{aligned}
|l_n(x)| &= \frac{1}{\epsilon_n}\int_D
\eta_{\epsilon_n}(x-y)|u(x)-u(y)|\rho(y)\, dy \\
& \leq \frac{1}{\epsilon_n}\int_D
\eta_{\epsilon_n}(x-y)\rho(y)\, dy \leq C/\epsilon_n,
\end{aligned}
\end{equation*}
since $\rho$ is bounded.

The same argument applied to $m_n$ gives the required inequalities.
\end{proof}

Recall, from the proof of Lemma \ref{lemma:gtvasconverge}, that  $h_n(X_i,X_j) =
\frac{1}{\epsilon_n}\eta_{\epsilon_n}(X_i-X_j)|u(X_i)-u(X_j)| - \int_D
f_n(X_i,y)\rho(y)\,dy - \int_D f_n(x,X_j)\rho(x)\,dx +
TV_{\epsilon_n}(u;\rho)$.

\begin{corollary}
\label{cor:gtvestimates}
There exists a constant $C$, such that,
for sufficiently large $n$, we have the following bounds:

\setlength{\tabcolsep}{10pt}
\renewcommand{\arraystretch}{1.5}
\begin{center}
\begin{tabular}{l }

$\mathbb{E}h_n^2 \leq C / \epsilon_n^{d+1}$, \\
$\|\mathbb{E}_Y h_n^2\|_{L^\infty} \leq C / \epsilon_n^{d+2}$, \\
$\|\mathbb{E}_X h_n^2\|_{L^\infty} \leq C  / \epsilon_n^{d+2}$, \\
$\|h_n\|_{L^\infty} \leq C / \epsilon_n^{d+1}$, \\
$\|h_n\|_{L^2 \to L^2} \leq C/\epsilon_n$. 
\end{tabular}
\end{center}
\end{corollary}

\begin{proof}
Since $h_n(X_i,X_j) = f_n(X_i,X_j) - l_n(X_i) - m_n(X_j) +
\mathbb{E}f_n$, we have
\begin{equation*}
\begin{aligned}
\sqrt{\mathbb{E}h_n^2} \leq \sqrt{\mathbb{E}f_n^2} + \sqrt{\mathbb{E}l_n^2}
+ \sqrt{\mathbb{E}_nm^2} + \sqrt{(\mathbb{E}f_n)^2}.
\end{aligned}
\end{equation*}
All terms in the right hand side may be bounded by
$\sqrt{C/\epsilon_n^{d+1}}$, and hence
\begin{equation*}
\mathbb{E}h_n^2 \leq  C / \epsilon_n^{d+1}.
\end{equation*}

Similarly, in the bound
\begin{equation*}
\sqrt{\|\mathbb{E}_X h_n^2\|_{L^\infty}} \leq \sqrt{\|\mathbb{E}_X
  f_n^2\|_{L^\infty}} + \sqrt{\|\mathbb{E}_X l_n^2\|_{L^\infty}} +
\sqrt{\|\mathbb{E}_X m_n^2\|_{L^\infty}} + \sqrt{(\mathbb{E}f_n)^2},
\end{equation*}
all terms in the right hand side are dominated by
$\sqrt{C/\epsilon_n^{d+2}}$, and hence
\begin{equation*}
\|\mathbb{E}_X h_n^2\|_{L^\infty} \leq C / \epsilon_n^{d+2}.
\end{equation*}

By symmetry of $h_n$, this gives
\begin{equation*}
\|\mathbb{E}_Y h_n^2\|_{L^\infty} \leq C / \epsilon_n^{d+2}.
\end{equation*}

Likewise, the same triangle inequality gives
\begin{equation*}
\|h_n\|_{L^\infty} \leq C / \epsilon_n^{d+1}.
\end{equation*}

For the last bound, we consider
\begin{equation*}
\begin{aligned}
\|h_n\|_{L^2\to L^2} &\leq \|f_n\|_{L^2 \to L^2} + \|l_n\|_{L^2 \to L^2} +
\|m_n\|_{L^2 \to L^2} + |\mathbb{E}f_n| \\
&\leq \|f_n\|_{L^2 \to L^2} + \|l_n\|_{L^\infty} + \|m_n\|_{L^\infty}
+|\mathbb{E}f_n|,
\end{aligned}
\end{equation*}
and each term in the right hand side is bounded by $C / \epsilon_n^d$.

\end{proof}

\subsection{Estimates for GF Limsup}

Recall, from the proof of Lemma \ref{lemma:quadasconvergea1}, the notation
 $f_n(X_i,X_j) =
\frac{1}{\sqrt{\epsilon_n}}\eta_{\epsilon_n}(X_i-X_j)u(X_j)$, $l_n(X_i) =
  \frac{1}{\sqrt{\epsilon_n}}\int_D \eta_{\epsilon_n}(X_i - y)
  u(X_i)\rho(y)\, dy$, and $m_n(X_j) = \frac{1}{\sqrt{\epsilon_n}}\int_D
  \eta_{\epsilon_n}(x-X_j)u(x)\rho(x)\,dx$.  Here, $u\in I(D)$.

\begin{lemma} \label{lemma:glambdaestimates}
 There exists a constant
  $C$, such that, for sufficiently large $n$, we have the following
  bounds:
\setlength{\tabcolsep}{10pt}
\renewcommand{\arraystretch}{1.5}
\begin{center}
\begin{tabular}{l  l }

$|\mathbb{E}f_n| \leq C /  \epsilon_n^{1/2}$, & $\|f_n\|_{L^2 \to L^2} \leq
                           C / \epsilon_n^{1/2}$, \\

$\mathbb{E}f_n^2 \leq C / \epsilon_n^{d+1}$, & $\mathbb{E}l_n^2 \leq C / \epsilon_n$, \\
$\|\mathbb{E}_Y f_n^2\|_{L^\infty} \leq C / \epsilon_n^{d+1}$, & $ \|l_n\|_{L^\infty} \leq C /
                                                  \epsilon_n^{1/2}$, \\
$\|\mathbb{E}_X f_n^2\|_{L^\infty} \leq C / \epsilon_n^{d+1}$, & $ \mathbb{E}m_n^2
                                                  \leq C / \epsilon_n$,
  \\
$\|f_n\|_{L^\infty} \leq C/\epsilon_n^{d+1/2}$, & $\|m_n\|_{L^\infty} \leq C / \epsilon_n^{1/2}.$
\end{tabular}
\end{center}
\end{lemma}

\begin{proof}
These inequalities are easier than the ones in Lemma \ref{lemma:gtvestimates}, and follow from the boundedness of $u$ and $\rho_{\epsilon_n}(x) = \int_D \eta_{\epsilon_n}(x-y) \rho(y)\, dy$.
\end{proof}

Recall that  $h_n(X_i,X_j) =
f_n(X_i,X_j) - l_n(X_i) - m_n(X_j) + \mathbb{E}f_n$. The proof of the following is similar to that of Corollary \ref{cor:gtvestimates}.

\begin{corollary} \label{cor:glambdaestimates}
There exists a constant $C$, such that,
for sufficiently large $n$, we have the following bounds:
\setlength{\tabcolsep}{10pt}
\renewcommand{\arraystretch}{1.5}
\begin{center}
\begin{tabular}{l }

$\mathbb{E}h_n^2 \leq C / \epsilon_n^{d+1}$, \\
$\|\mathbb{E}_Y h_n^2\|_{L^\infty} \leq C / \epsilon_n^{d+1}$, \\
$\|\mathbb{E}_X h_n^2\|_{L^\infty} \leq C  / \epsilon_n^{d+1}$, \\
$\|h_n\|_{L^\infty} \leq C / \epsilon_n^{d+1/2}$, \\
$\|h_n\|_{L^2 \to L^2} \leq C/\epsilon_n^{1/2}$. 
\end{tabular}
\end{center}
\end{corollary}

\subsection{Transport Distance in $d=1$}

In this section, we provide the transport maps $\{T_n\}_{n \in
  \mathbb{N}}$ in $d=1$ and establish a bound on the rate at which $\|Id -
T_n\|_{L^\infty} \to 0$ as $n\rightarrow\infty$. 

Recall that by assumption $(M)$ of Subsection \ref{results},
$\nu$ is a probability measure on $D = (c,d)$ with distribution function $F_\nu$ and density $\rho$ that is differentiable, Lipschitz, and bounded above
and below by positive constants. Further,  $\rho$ is increasing in some interval with left endpoint $c$ and decreasing in some interval with right endpoint $d$.

 Given a sample $\mathcal{X}_n =
\{X_1,\ldots,X_n\}$, we let $F_n$ denote the distribution
function of the empirical measure $\nu_n = \frac{1}{n} \sum_{i=1}^n
\upsilon_{X_i}$. We define $T_n$ by
\begin{equation} \label{eq:1dtransportdef}
T_nx = F_n^{-1}(F x),
\end{equation} where $F_n^{-1}(t) = \inf \{x \in
\mathbb{R} : t \leq F_n(x) \}$. The map $T_n$ is a valid transport
map, i.e.
$ T_{n\sharp} \nu = \nu_n.$

By the assumptions on $\rho$, it follows that
\begin{equation} \label{eq:shorackassumption}
\sup_{c<x<d} F_\nu (x)(1- F_\nu
(x))|\rho'(x)|/\rho^2(x)<\infty
\end{equation}

It is known -- see Theorem 3
on p. 650 of \cite{shorack2009empirical} -- that when i) $\rho > 0$ on $(c,d)$, ii) inequality
\eqref{eq:shorackassumption} is satisfied, and iii) $\rho$ is
increasing in some interval with left endpoint $c$ and decreasing in
some interval with right endpoint $d$, the standardized quantile process
\begin{equation*}
\mathcal{Q}_n(t) \coloneqq g(t)\sqrt{n}[F_n^{-1}(t) - F^{-1}(t)],
\end{equation*}
with $g(t) = \rho(F^{-1}(t))$, satisfies, almost surely,
\begin{equation} \label{eq:quantile_limit}
\limsup_{n \to \infty} \sup_{0< t< 1}|\mathcal{Q}_n(t)/\sqrt{2\log\log n}| \leq 1.
\end{equation}

Since $\rho$ is bounded above and below by non-negative constants, so
is $g$, and so we have constants $C,C' > 0$ such that
\begin{equation*}
C |\mathcal{Q}_n(t)| \leq \sqrt{n}|F_n^{-1}(t) -
F^{-1}(t)| \leq C' |\mathcal{Q}_n(t)|.
\end{equation*}
Since $\rho$ is positive, $F$ is strictly increasing, and hence we
have
\begin{equation*}
C \sup_{0 < t < 1} |\mathcal{Q}_n(t)| \leq \sup_{c < x < d} \sqrt{n}|F_n^{-1}(Fx) -
F^{-1}(Fx)| \leq C' \sup_{0 < t < 1} |\mathcal{Q}_n(t)|.
\end{equation*}
Recalling our definition of $T_n$, this may be rewritten as
\begin{equation*}
C \sup_{0 < t < 1} |\mathcal{Q}_n(t)| \leq \sqrt{n} \|Id - T_n\|_{L^\infty} \leq C' \sup_{0 < t < 1} |\mathcal{Q}_n(t)|.
\end{equation*}

In light of \eqref{eq:quantile_limit} and the above inequality, we obtain the following estimate.

\begin{proposition} \label{lemma:transportd1}

There is a constant $C$ such that, almost surely, the transport maps $T_n$, defined by (\ref{eq:1dtransportdef}), satisfy
\begin{equation*}
\lim_{n \to \infty} \frac{\sqrt{n}\|Id - T_n\|_{L^\infty}}{\sqrt{2\log\log n}} \leq C.
\end{equation*}
\end{proposition}

  \vskip .2cm
\noindent {\bf Acknowledgement.}  This work was partially supported by ARO W911NF-14-1-0179.


\vskip .1cm


\begin{thebibliography}{99}

  
  \bibitem{adams2003sobolev}
  {\sc Adams, Robert A and Fournier, John JF} (2003).
  \newblock {\it Sobolev Spaces}
  \newblock Academic Press



\bibitem{agarwal2008modularity}
{\sc Agarwal, Gaurav and Kempe, David} (2008).
  \newblock Modularity-maximizing graph communities via mathematical programming.
  \newblock {\em European Physical Journal B\/} {\bf 66,} 409--418.

\bibitem{alberti1998non}
  {\sc Alberti, Giovanni and Bellettini, Giovanni} (1998).
  \newblock A non-local anisotropic model for phase transitions:
  asymptotic behaviour of rescaled energies.
\newblock {\em European Journal of Applied Mathematics \/} {\bf 9,} 261--284.

  
  \bibitem{Ambrosio_Gigli_Savare}
{\sc Ambrosio, Luigi, Gigli, Nicola and Savar\'e, Giuseppe} (2005).
\newblock {\it Gradient Flows in Metric Spaces and in the Space of Probability Measures.}
\newblock Birkh\"auser Verlag. Basel

\bibitem{ambrosio2000functions}
  {\sc Ambrosio, Luigi and Fusco, Nicola and Pallara, Diego} (2000).
  \newblock {\it Functions of Bounded Variation and Free Discontinuity Problems}
  \newblock Clarendon Press, Oxford.

\bibitem{antonioni}
{\sc Antonioni, Alberto and Eglof, Mattia and Tomassini, Marco} (2013).
\newblock An energy-based model for spatial social networks
\newblock In {\it Advances in Artificial Life ECAL 2013}
\newblock pp. 226--231, MIT Press

\bibitem{Arias_Pelletier}
{\sc Arias-Castro, E., and Pelletier, B.} (2013).
\newblock 
On the Convergence of Maximum Variance Unfolding
\newblock {\em Journal of Machine Learning Research} {\bf 14} 1747--1770.

\bibitem{APP}
{\sc Arias-Castro, E., Pelletier, B., Pudlo, P.} (2012).
\newblock 
The normalized graph cut and Cheeger constant: from discrete to continuous
\newblock {\em Adv. in Appl. Probab.} {\bf 44} 907--937.
   

\bibitem{Belkin-Niyogi}
{\sc Belkin, M. and Niyogi, P.} (2003).
\newblock Laplacian eigenmaps for dimensionality reduction and data representation
\newblock { \em Neural Comput.} {\bf 15}  1373--1396.

\bibitem{Belkin-Niyogi1}
{\sc Belkin, M. and Niyogi, P.} (2008).
\newblock Towards a theoretical foundation for Laplacian-based manifold methods.
\newblock {\em J. Comput. System Sci.} {\bf 74} 1289--1308.


\bibitem{bettstetter2002minimum}
{\sc Bettstetter, Christian} (2002).
  \newblock On the minimum node degree and connectivity of a wireless multihop network.
\newblock In {\em Proceedings of the 3rd ACM International Symposium on Mobile ad hoc Networking \& Computing}
\newblock pp. 80--91, ACM.

\bibitem{Bickel-Chen}
{\sc Bickel, P. and Chen, A.} (2009).
\newblock A nonparametric view of network models and Newman-Girvan and other modularities.
\newblock {\em PNAS} {\bf 106} 21068--21073.

\bibitem{blondel2008fast}
{\sc Blondel, Vincent D and Guillaume, Jean-Loup and Lambiotte, Renaud and Lefebvre, Etienne} (2008).
 \newblock Fast unfolding of communities in large networks.
 \newblock {\em J. Stat. Mech.: Theory and Experiment\/} {\bf 2008,} 10008--10020.


\bibitem{bourgain2001another}
  {\sc Bourgain, Jean and Brezis, Haim and Mironescu, Petru} (2001).
  \newblock Another look at Sobolev spaces.

\bibitem{braides2002gamma}
{\sc Braides, Andrea} (2002).
\newblock  {\it Gamma convergence for Beginners}
\newblock Oxford University Press, Oxford.

\bibitem{braides_gelli}
{\sc Braides, A. and Gelli, M.S.} (2006).
\newblock From discrete systems to continuous variational problems:  an introduction.
\newblock In {\it Topics on Concentration Phenomena and Problems with Multiple Scales}, Lecture Notes of the Unione Matematica Italiana {\bf 2}, 3--77.

\bibitem{brakke1992surface}
{\sc Brakke, Kenneth A} (1992).
  \newblock The surface evolver.
  \newblock {\em Experimental mathematics} {\bf 15,} 1(2), 519--527.

\bibitem{brandes2008modularity}
\newblock {\sc Brandes, Ulrik and Delling, Daniel and Gaertler, Marco and G{\"o}rke, Robert and Hoefer, Martin and Nikoloski, Zoran and Wagner, Dorothea} (2008).
\newblock On modularity clustering.
\newblock {\em  IEEE Transactions on Knowledge and Data Engineering\/} {\bf 20,} 172--188.

\bibitem{CR03}
{\sc Ca$\tilde{n}$ete, A., Ritor\'e, M.} (2003).
\newblock Least-permiter partitions of the disk into three regions of given areas. 
\newblock {\em ArXiv preprint} ArXiv: 0307207
  

\bibitem{clauset2004finding}
{\sc Clauset, Aaron and Newman, Mark EJ and Moore, Cristopher} (2004).
\newblock Finding community structure in very large networks.
\newblock {\em Phys. Rev. E\/} {\bf 70,} 066111

\bibitem{Coiffman_Lafon}
{\sc Coiffman, R. and Lafon, S.} (2006).
\newblock Diffusion maps.
\newblock {\em Appl. Comput. Harmon. Anal.} {\bf 21}, 5--30.

\bibitem{cox2010minimal}
{\sc  Cox, SJ  and Flikkema, E.} (2010).
 \newblock The minimal perimeter for N confined deformable bubbles of equal areas.
 \newblock {\em The Electronic Journal of Combinatorics} {\bf 17(R45)}.

\bibitem{davila2002open}
{\sc D{\'a}vila, J} (2002).
\newblock On an open question about functions of bounded variation.
\newblock {\em Calculus of Variations and Partial Differential
  Equations\/} {\bf 15,} 519--527.

\bibitem{de1995decoupling}
  {\sc de la Pe{\~n}a, Victor H and Montgomery-Smith, Stephen J} (1995).
  \newblock Decoupling inequalities for the tail probabilities of multivariate U-statistics
  \newblock {\em Annals of Probability} {\bf } 806--816.
 
\bibitem{dhara}
{\sc Dhara, M. and Shukla, K.K.} (2012).
\newblock Advanced cost based graph clustering algorithm for random geometric graphs
{\em Int. J. Computer Appl.} {\bf 60,} 20--34.

  
 \bibitem{diaz2001approximating}
{\sc D{\i}az, Josep and Penrose, Mathew D and Petit, Jordi and Serna, Mar{\i}a} (2001).
 \newblock Approximating layout problems on random geometric graphs.
 \newblock {\em J. Algorithms\/} {\bf 39,} 78--116.
 
 \bibitem{diaz2002survey}
{\sc D{\'\i}az, Josep and Petit, Jordi and Serna, Maria} (2002).
 \newblock A survey of graph layout problems.
 \newblock {\em ACM Computing Surveys\/}
{\bf 34,} 313--356.

\bibitem{dudley}
{\sc Dudley, R.M.} (2004)
\newblock {\it Real Analysis and Probability}
\newblock Cambridge University Press, Cambridge.
 
 \bibitem{durrett}
 {\sc Durrett, R. } (2010)
 \newblock {\it Probability:  Theory and Examples} 4th Ed.
 \newblock Cambridge University Press, Cambridge.

\bibitem{folland2013real}
  {\sc Folland, Gerald B} (2013).
  \newblock {\it Real analysis: Modern Techniques and Their Applications}
\newblock  John Wiley \& Sons

\bibitem{fortuna2010nestedness}
  {\sc Fortuna, Miguel A and Stouffer, Daniel B and Olesen, Jens M and Jordano, Pedro and Mouillot, David and Krasnov, Boris R and Poulin, Robert and Bascompte, Jordi} (2010).
  \newblock Nestedness versus modularity in ecological networks: two sides of the same coin?
  \newblock {\em Journal of Animal Ecology} {\bf 79,} 811--817.

\bibitem{fortunato2007resolution}
{\sc Fortunato, S. and Barth\'elemy, M.} (2006).
\newblock Resolution limit in community detection.
\newblock {\em PNAS} {\bf 104} 36--41.

\bibitem{Fortunato2010}
{\sc Fortunato, S.} (2010).
\newblock Community detection in graphs
\newblock {\em Physics Reports} {\bf 486} 75--174.

\bibitem{franceschetti_meester}
{\sc Franceschetti, M. and Meester, R.} (2007).
\newblock {\it Random Networks for Communication:  From Statistical Physics to Information Systems.}
\newblock Cambridge University Press, Cambridge.

\bibitem{gamal2004throughput}
{\sc Gamal, A El and Mammen, James and Prabhakar, Bharat and Shah, Devavrat} (2004).
  \newblock Throughput-delay trade-off in wireless networks.
  \newblock In {\em Twenty-third Annual Joint Conference Proceedings of the IEEE Computer and Communications Societies}
  
  
 \bibitem{trillos2014rate}
{\sc Garc{\'\i}a Trillos, Nicol{\'a}s  and Slep{\v{c}}ev, Dejan} (2015).
 \newblock On the rate of convergence of empirical measures in $\infty$-transportation distance.
 \newblock {\em Canadian Journal of Mathematics} {\bf 67} 1358--1383
  
  
  
\bibitem{trillos2014continuum}
{\sc Garc{\'\i}a Trillos, Nicol{\'a}s  and Slep{\v{c}}ev, Dejan} (2016).
 \newblock Continuum limit of total variation on point clouds.
 \newblock {\em Archive for Rational Mechanics and Analysis} {\bf 220} 193--241.
  
  \bibitem{Trillos_consistency}
{\sc Garc{\'\i}a Trillos, Nicol{\'a}s  and Slep{\v{c}}ev, Dejan} (2015).
\newblock A variational approach to the consistency of spectral clustering
\newblock {\em arXiv preprint}  arXiv:1508.01928
  
\bibitem{trillos2014consistency}
{\sc  Garc{\'\i}a Trillos, Nicol{\'a}s and Slep{\v{c}}ev, Dejan, and von Brecht, James and Laurent, Thomas and Bresson, Xavier} (2014).
 \newblock Consistency of Cheeger and Ratio Graph Cuts.
 \newblock {\em arXiv preprint\/} arXiv:1411.6590



\bibitem{Gine_Koltch}
{\sc Gin{\'e}, E. and Koltchinski, V.} (2006).
\newblock Empirical graph Laplacian approximation of Laplace-Beltrami operators:  large
sample results
\newblock In {\it High dimensional probability} {\bf 51} 238--259.
\newblock  IMS Lecture Notes Monogr. Ser., Inst. Math. Statist.,
Beachwood, OH

\bibitem{gine2000exponential}
  {\sc Gin{\'e}, Evarist and Lata{\l}a, Rafa{\l} and Zinn, Joel} (2000).
  \newblock Exponential and moment inequalities for U-statistics
  \newblock In {\em High Dimensional Probability II}
\newblock pp. 13--38, Springer.

\bibitem{vanGennip_Bertozzi}
{\sc van Gennip, Y. and Bertozzi, A.} (2012).
\newblock Gamma convergence of graph Ginzburg-Landau functionals
\newblock {\em Advances in Differential Equations} {\bf 17} 1115--1180. 
  


\bibitem{good2010performance}
  {\sc Good, Benjamin H and de Montjoye, Yves-Alexandre and Clauset, Aaron} (2010).
  \newblock Performance of modularity maximization in practical contexts
  \newblock {\em Phys. Rev. E} {\bf 81,} 046106
 
  


\bibitem{guimera2005functional}
 {\sc Guimera, Roger and Amaral, Luis A Nunes} (2005).
\newblock Functional cartography of complex metabolic networks
  {\em Nature} {\bf 433,} 895--900.



\bibitem{guimera2004modularity}
{\sc Guimera, Roger and Sales-Pardo, Marta and Amaral, Lu{\'\i}s A Nunes} (2004).
\newblock Modularity from fluctuations in random graphs and complex networks.
\newblock {\em Phys. Rev. E\/} {\bf 70,} 025101.





 \bibitem{gupta2000capacity}
{\sc Gupta, Piyush and Kumar, Panganmala R} (2000).
\newblock The capacity of wireless networks.
\newblock {\em IEEE Transactions on Information Theory\/} {\bf 46,} 288--404.



\bibitem{hagmann2008mapping}
  {\sc Hagmann, Patric and Cammoun, Leila and Gigandet, Xavier and Meuli, Reto and Honey, Christopher J and Wedeen, Van J and Sporns, Olaf} (2008).
  \newblock Mapping the structural core of human cerebral cortex.
  \newblock {\em PLoS Biol} {\bf 6,} e159
  
\bibitem{Hartigan}
{\sc Hartigan, J.A.} (1981).
\newblock
Consistency of Single Linkage for High-Density Clusters
\newblock
{\em Journal of the American Statistical Association} {\bf 76} 388--394.



\bibitem{Hein_et_al}
{\sc Hein, M., Audibert, J.-Y., Von Luxburg, U.} ( 2005).
\newblock From graphs to manifolds-weak and strong pointwise consistency of graph Laplacians
\newblock In {\it Learning Theory} 470--485, Springer
  
 
\bibitem{hu_et_al}
{\sc Hu, H., Laurent, T.,  Porter, M.,  Bertozzi, A.} (2013).
\newblock A Method Based on Total Variation for Network Modularity Optimization using the MBO Scheme
\newblock {\em SIAM J. Appl. Math.} {\bf 73} 2224--2246. 
 
\bibitem{lancichinetti2011limits}
 {\sc Lancichinetti, Andrea and Fortunato, Santo} (2011).
 \newblock Limits of modularity maximization in community detection.
  \newblock {\em Phys. Rev. E} {\bf 84,} 066122
 

\bibitem{Le-Levina-Vershynin}
{\sc Le, C., Levina, E., and Vershynin, R.} (2016).
\newblock Optimization via low-rank approximation for community detection in networks.
 \newblock {\em Ann. Stat.} {\bf 44} 373--400.





\bibitem{medus2005detection}
  {\sc Medus, A and Acuna, G and Dorso, CO} (2005).
  \newblock Detection of community structures in networks via global optimization
  \newblock {\em Physica A: Statistical Mechanics and its Applications} {\bf 358,} 593--604.

\bibitem{meester_roy}
{\sc Meester, R. and Roy, R.} (1996).
\newblock {\it Continuum Percolation.}
\newblock Cambridge University Press, Cambridge.

\bibitem{mill2008epigenomic}
  {\sc Mill, Jonathan and Tang, Thomas and Kaminsky, Zachary and Khare, Tarang and Yazdanpanah, Simin and Bouchard, Luigi and Jia, Peixin and Assadzadeh, Abbas and Flanagan, James and Schumacher, Axel and others} (2008).
  \newblock Epigenomic profiling reveals DNA-methylation changes associated with major psychosis
  \newblock {\em The American Journal of Human Genetics}, {\bf 82,} 696--711.

\bibitem{Mor08}
{\sc Morgan, F. } (2008).
\newblock {\it Geometric Measure Theory.  A Beginner's Guide} 4th Ed.
\newblock Academic Press. New York.


\bibitem{newman2006modularity}
  {\sc Newman, Mark EJ} (2006).
  \newblock Modularity and community structure in networks
  \newblock {\em PNAS} {\bf 103,} 8577--8582.




\bibitem{newman2006finding}
{\sc Newman, Mark EJ} (2006).
\newblock Finding community structure in networks using the eigenvectors of matrices.
\newblock {\em Phys. Rev. E\/} {\bf 74,} 036104




\bibitem{newman2013spectral}
  \newblock {\sc Newman, Mark EJ} (2013).
  \newblock Spectral methods for community detection and graph partitioning.
\newblock {\em Phys. Rev. E} {\bf 88,} 042822


\bibitem{newman2004finding}
 {\sc Newman, Mark EJ and Girvan, Michelle} (2004).
  \newblock Finding and evaluating community structure in networks
  \newblock {\em Phys. Rev. E} {\bf 69,} 026113
  
  \bibitem{Oudet}
  {\sc Oudet, E. } (2011).
  \newblock Approximation of partitions of least permiter by $\Gamma$-convergence:  Around Kelvin's conjecture.
  \newblock {\em Experimental Mathematics} {\bf 20} 260--270.


 

\bibitem{penrose2003random}
{\sc Penrose, Mathew} (2003).
 \newblock {\it Random geometric graphs}
 \newblock Oxford University Press, Oxford

\bibitem{pollard1981strong}
  {\sc Pollard, David} (1981).
  \newblock Strong consistency of $ k $-means clustering
  \newblock {\em Annals of Statistics} {\bf 9,} 135--140.

\bibitem{ponce2004new}
{\sc Ponce, Augusto} (2004).
\newblock A new approach to Sobolev spaces and connections to-convergence
\newblock {\em Calculus of Variations and Partial Differential
  Equations} {\bf 19,} 229-255.


\bibitem{porter2005network}
 {\sc Porter, Mason A and Mucha, Peter J and Newman, Mark EJ and Warmbrand, Casey M} (2005).
  \newblock A network analysis of committees in the US House of Representatives.
  \newblock {\em PNAS} {\bf 102,} 7057--7062


  
\bibitem{Porter_Onnela_Mucha} 
{\sc Porter, M., Onnela, J.P., Mucha, P.} (2009).
\newblock Communities in networks.
\newblock {\em Notices of the AMS} {\bf 56} 1082--1097.


\bibitem{prvzulj2004modeling}
{\sc Pr{\v{z}}ulj, Natasa and Corneil, Derek G and Jurisica, Igor} (2004).
\newblock Modeling interactome: scale-free or geometric?
\newblock {\em Bioinformatics\/} {\bf 20,} 
 3508--3515.

\bibitem{reichardt2006statistical}
  {\sc Reichardt, J{\"o}rg and Bornholdt, Stefan} (2006).
 \newblock Statistical mechanics of community detection.
 \newblock {\em Phys. Rev. E} {\bf 74,} 016110
  
  \bibitem{Rohe-Chatterjee-Yu}
 {\sc Rohe, Karl, Chatterjee, Sourav, and Yu, Bin} (2011).
\newblock Spectral clustering and the high-dimensional stochastic blockmodel.
 \newblock {\em Ann. Stat.} {\bf 39} 1878--1915.

\bibitem{Sabin}
{\sc Sabin, Michael} (1987).
\newblock Convergence and Consistency of Fuzzy c-means/ISODATA Algorithms.
\newblock {\em IEEE Transactions on Pattern Analysis and Machine Intelligence} {\bf 9} 661--668.

 
  
    \bibitem{shorack2009empirical}
{\sc Shorack, Galen R and Wellner, Jon A} (2009).
\newblock {\it Empirical processes with applications to statistics}
\newblock SIAM, Philadelphia.

\bibitem{Singer}
{\sc Singer, A.} (2006).
\newblock From graph to manifold Laplacian:  The convergence rate
\newblock {\em Appl. Comput.  Harmon. Anal. } {\bf 21} 128--134.

\bibitem{Singer_Wu}
{\sc Singer, A. and Wu, H.T.} (2015).
\newblock 
Spectral Convergence of the connection Laplacian from random samples
\newblock {\em arXiv preprint.} arXiv:1306.1587

  
 
  
\bibitem{thorpe_theil_joh_cade}
{\sc Thorpe, M., Theil, F., Johansen, P., Cade, N.} (2015). 
\newblock Convergence of the $k$-Means minimization problem using $\Gamma$-convergence
\newblock {\em arXiv preprint} arXiv: 1501.01320v2

\bibitem{Ting_Huang_Jordan}
{\sc Ting, D.,  Huang, L., and  Jordan, M.I.} (2010).
\newblock An analysis of the convergence of graph Laplacians
\newblock In Proceedings of
the 27th International Conference on Machine Learning.

\bibitem{van2000asymptotic}
{\sc Van der Vaart, Aad W} (2000).
 \newblock {\it Asymptotic statistics}
  \newblock Cambridge University Press, Cambridge.
  

  \bibitem{Villani_intro_transportation}
{\sc Villani, C\'edric} (2004).
\newblock {\it Topics in Optimal Transportation.}
\newblock American Mathematical Society, Providence
  
  \bibitem{Villani_old_new}
{\sc Villani, C\'edric} (2009).
\newblock {\it Optimal Transport Old and New.}
\newblock Springer Verlag, Berlin

 
\bibitem{von2008consistency}
  {\sc Von Luxburg, Ulrike and Belkin, Mikhail and Bousquet, Olivier} (2008).
  \newblock Consistency of spectral clustering
  \newblock {\em Annals of Statistics} {\bf } 555-586.

\bibitem{Wets}
{\sc Wets, Roger} (1999).
\newblock Statistical estimation from an optimization viewpoint.
\newblock {\em Annals of Operations Research} {\bf 85}  79–-102.
 
 
 
  
 \bibitem{Zhao-Levina-Zhu}
 {\sc Zhao, Y., Levina, E. and Zhu, J.} (2012). 
 \newblock Consistency of community detection in networks under degree-corrected stochastic block models. 
 \newblock {\em Ann. Statist.} {\bf 40} 2266-�2292.


\end{thebibliography}
\end{document}